\documentclass[final, 3p]{elsarticle}

\usepackage{amsthm}
\usepackage{latexsym,amsmath,graphicx}

\usepackage{amssymb}
\usepackage{amscd}

\usepackage{mathrsfs}
\usepackage{amsfonts}
\usepackage{bbm}
\usepackage{dsfont}

\usepackage{natbib}

\newtheorem{theorem}{Theorem}[section]
\newtheorem{lemma}[theorem]{Lemma}
\newtheorem{definition}[theorem]{Definition}
\newtheorem{proposition}[theorem]{Proposition}
\newtheorem{corollary}[theorem]{Corollary}

 \newtheorem{remark}[theorem]{Remark}

\newcommand{\nn}{\nonumber}

\numberwithin{equation}{section}

\title{Rough path properties for local time of symmetric $\alpha$ stable process}

\author[a,b]{Qingfeng Wang}
\address[a]{Department of Mathematical Sciences, Loughborough University, LE11 3TU, UK}

\author[a]{Huaizhong Zhao\corref{}}
\address[b]{Nottingham University Business School China, Ningbo, 315100, China}
\fntext[]{Email address: H.Zhao@lboro.ac.uk}
\cortext[]{Corresponding author.}

\journal{Stochastic Processes and their Applications}
\begin{document}

\begin{abstract}  
\noindent In this paper, we first prove that the local time associated with symmetric $\alpha$-stable processes is of bounded $p$-variation for any $p>\frac{2}{\alpha-1}$ partly based on Barlow's estimation of the modulus of the local time of such processes.\,\,The fact that the local time is of bounded $p$-variation for any  $p>\frac{2}{\alpha-1}$ enables us to define the integral of the local time  $\int_{-\infty}^{\infty}\triangledown_-^{\alpha-1}f(x)d_x L_t^x$ as a Young integral for less smooth functions being of bounded $q$-varition with $1\leq q<\frac{2}{3-\alpha}$. When $q\geq \frac{2}{3-\alpha}$, Young's integration theory is no longer applicable. However, rough path theory is useful in this case. The main purpose of this paper is to establish a rough path theory for the integration with respect to the local times of symmetric $\alpha$-stable processes for $\frac{2}{3-\alpha}\leq q< 4$.
\\
\\
\noindent{\it Key words}:{ Young integral, rough path, local time,  $p$-variation, $\alpha$-stable processes,  It\^o's formula}
\end{abstract}

\maketitle

\pagenumbering{arabic}
\section{Introduction}
It{\^o} \cite{Ito1944} developed the integration theory with respect to Brownian motion and the chain rule for Brownian motion known as It{\^o}'s formula.
 We first recall It\^o's formula developed in 1944 as follows. 
 
{\bf(\bf It\^o's Theorem \bf(\bf 1944\bf))} Let $f: \mathbb{R}\to\mathbb{R}$ be a function of the class $C^2$ and $B = \{B_t, \mathcal{F}_t: 0\leq t<\infty\}$ be any Brownian motion on $(\Omega, \mathcal{F}_t)$, then
\begin{equation*}
f(B_t)=f(B_0) + \int_{0}^{t} f'(B_s)dB_s + \frac{1}{2}\int_{0}^{t}f''(B_s)ds.
\end{equation*}
 
For It\^o's formula, the contribution lies in defining the stochastic integral $\int_{} f'(B_s)dB_s$.\,An integral $\int_{} XdZ$ can be defined as a Stieltjes integral pathwise when the integrator $Z$ is of finite variation and the integrand $X$ is continuous.\,However, Brownian motion is not of bounded variation a.s. and when the integrator is of infinite variation, there did not exist an integration theory in place to use before It\^o dealt with this issue.
 
Despite a huge success, It\^o's original formula has its own limitations as it applies for Brownian motion and for functions with twice differentiability only.\,This hinders the applicability of It\^o's formula.\,\,On one hand, one often encounters the need to define the stochastic integral for a wide class of stochastic processes besides Brownian motion.\,\,Doob \cite{Doob1953} emphasized the martingale property of It\^o's integral.\,Subsequently, Doob proposed a general martingale integral after discovering the key role of the martingale property of Brownian motion in defining It\^o's integral.\,In order to build the theory, he needed to decompose the square of an $L^2$-martigale.\,\,This was done latter in \cite{Meyer62}.\,\,Based on \cite{Meyer62}, Kunita and Watanabe in \cite{Kunita1967} defined an integral provided the integrand is previsible for the case the integrator is a square integrable martingale. They generalized It\^o's formula to continuous martingales only and proved the above It\^o's formula is still valid if $ds$ is replaced by $d<X>_s$ (see Section 2, \cite{Kunita1967}) and generalized It\^o's formula to discontinuous martingales (see section 5, \cite{Kunita1967}).\,\,Meanwhile, Meyer \cite{Meyer67} extended It\^o's formula to local martingales of which the concept was introduced in \cite{Watanabe1965}. Meyer \cite{Meyer76} further extended It\^o's formula to semimartingales. 
 
On the other hand, one also encounters the restriction of using It\^o's formula in the cases when the function is not $C^2$ in the space variable.\,\,The first result in this direction is the It\^o -Tanaka's formula derived in \cite{Tanaka63} for $f(x)=|x|$ with the help of the local time.\,\,The concept of local time was first introduced in L\'evy \cite{Levy1948}.\,\,It has been indeed the wellspring of much of the extensions of It\^o's formula for functions not being $C^2$ in the space variable.\,\,Wang \cite{wang77} extended It\^o's formula to a time independent convex function which is being absolutely continuous with its first order derivative being of bounded variation.\,\,Bouleau and Yor \cite{Bouleau1981} made a further extension to absolutely continuous functions with locally bounded first order derivative. 
The idea to extend It\^o's formula to less smooth functions of Feng and Zhao in their papers \cite{Feng2006, Feng2008, Feng2010} is to establish a Young's integration theory and a rough path integration theory for local time.\,\,Young \cite{Young36} showed that the pathwise integral $\int_{} XdZ$ makes sense if X is of finite $p$-variation and Z is of finite $q$-variation where $\frac{1}{p}+\frac{1}{q}>1$, together with the condition that X and Z have no common discontinuities. The theory of rough paths was developed by  Lyons and his co-authors in a series of papers (see, e.g. \cite{Bass2002,Coutin2002, Lyons1994,Lyons98,Lyons97,Lyons1997,Lyons2002}). Rough path theory removed the restriction of $\frac{1}{p}+\frac{1}{q}>1$, hence applicable to even rougher paths.

The purpose of this paper is to establish the integral $\int_{-\infty}^{\infty}\triangledown_-^{\alpha-1}f(x)d_x L_t^x$ for symmetric $\alpha$ stable processes as a Young integral as well as a rough path integral.
 
The rest of the paper is organized as follows.\,\,In Section 2 we recall some results from the Young's integration theory and establish a Young integral of local time. In Section 3 we recall some results from the rough path theory and establish a rough path integral of local time. 
 
\section{The Young integral}
 We first recall the definition of the bounded $p$-variation  (see e.g. \cite{Young36}, \cite{Lyons2002}).
\begin{definition}\label{pvariation}
A function $f : [x',x'']\to\mathbb{R} $ is of bounded $p$-variation if
\begin{equation}
\sup_{E} \sum_{i=1}^{m} {|f(x_i)-f(x_{i-1})|}^p < \infty
\end{equation}
where $E : =\{x'=x_0<x_1<\cdots<x_m=x''\}$ is an arbitrary partition of $[x',x''].$ Here $p \geq 1$ is a fixed real number. 
\end{definition}
We present Young's integration theorem from \cite{Young36} in the following.
\begin{theorem}(Young integral)\label{YoungIntegral}
Consider a function $f$ of finite $p$-variation and a function $g$ of finite $q$-variation where $p, q>0, \frac{1}{p}+\frac{1}{q}>1,$ such that $f(x)$ and $g(x)$ have no common discontinuities, then
\begin{equation}
\int_{x'}^{x''} f(x)dg(x)=\displaystyle\lim_{m(E)\to 0}\sum_{i=1}^{m}f(\xi_i)\bigl(g(x_i)-g(x_{i-1})\bigr)
\end{equation}
is well defined. Here $\xi_i\in [x_{i-1}, x_i], m(E)=\displaystyle\sup_{1\leq i\leq m}(x_i - x_{i-1}).$
\end{theorem}
We also recall the integration of a sequence of functions which is also due to Young (see \cite{Young36}).
\begin{theorem}(Term by term integration)\label{TermbytermIntegration}
Let $\{f_n\}$ be a sequence of functions of finite $p$-variation converging densely to a function $f$ of finite p-variation uniformly at each point of a set A. Let $\{g_n\}$ be a sequence of functions of finite $q$-variation converging densely, and at $x', x''$ to a function g of finite $q$-variation uniformly at each point of a set B. Suppose further that $p, q>0, \frac{1}{p}+\frac{1}{q}>1,$ and that A includes the discontinuities of $g,$ B those of $f$, $A\cup B$ represents all points of $(x',x'')$. Then as $n\to\infty$
\begin{equation}
\int_{x'}^{x''}f_ndg_n\to\int_{x'}^{x''}fdg.
\end{equation}
\end{theorem}

\subsection{Fractional derivative and fractional Laplacian}
 
We first recall the definition of the $\alpha^{th}$ right fractional integral of a function $g$ from \cite{Miller93}.\,\,The left fractional integral is defined similarly.\,\,The Riemann's definition of fractional integral for a suitable function $g$ is
 
\begin{equation}
_a I_x^{\alpha}g(x) = \frac{1}{\Gamma(\alpha)}\int_{a}^{x}(x-u)^{\alpha-1}g(u)du,
\end{equation}
for almost all $x$ with $-\infty\leq a<x<\infty$ and Re$(\alpha)>0,$ where $\alpha$ is a complex number in general and Re denotes its real part. For $a=0$, this is a Riemann-Liouville fractional integral
\begin{equation}\label{RL}
_0 I_x^{\alpha}g(x) = \frac{1}{\Gamma(\alpha)}\int_{0}^{x}(x-u)^{\alpha-1}g(u)du, \qquad Re(\alpha)>0.
\end{equation}

A sufficient condition that ensures the convergence of integral (\ref{RL}) is that
$
g(\frac{1}{x}) = O(x^{1-\eta}), \,\eta>0.
$
Functions with the above property are sometimes called functions of the Riemann class.\,For instance, constants are of Riemann class as well as functions such as $x^m, $ with $m>-1.$ \,For $a=-\infty$, this is the Liouville fractional integral
\begin{equation}\label{L}
_{-\infty} I_x^{\alpha}g(x) = \frac{1}{\Gamma(\alpha)}\int_{-\infty}^{x}(x-u)^{\alpha-1}g(u)du, \qquad \textrm{Re}(\alpha)>0.
\end{equation}

A sufficient condition that ensures the converges of integral (\ref{L}) is that $g(-x) = O(x^{-\alpha-\eta}), \,\eta>0,\, x\to\infty$.
Functions with the above property are sometimes called functions of the Liouville class.\,For example, functions such as $x^m,$ with $m<-\alpha<0$ are of the Liouville class.\,The fractional integral operator satisfy the semigroup property, namely
\begin{equation*}
_{a} I_x^{\alpha} ({_{a} I_x^{\beta}} )= _{a} I_x^{\alpha+\beta}, \qquad \textrm{Re}(\alpha, \beta)>0.
\end{equation*}

The right fractional derivative operator and left fractional derivative operator are defined in terms of the right fractional integral operator and left fractional integral operator in the following manner\begin{equation}
_{a}\triangledown^{\alpha}_x = \frac{d^n}{dx^n} \biggl({_{a}} I_x^{n-\alpha}\biggr),\qquad\quad\,\, \textrm{Re}(\alpha)>0, n=\lfloor \textrm{Re}(\alpha)\rfloor+1,
\end{equation}
and
\begin{equation}
_{x}\triangledown^{\alpha}_a = (-1)^n\frac{d^n}{dx^n} \biggl({_{x}} I_a^{n-\alpha}\biggr),\quad \textrm{Re}(\alpha)>0, n=\lfloor \textrm{Re}(\alpha)\rfloor+1.
\end{equation}
Next, we recall the definition of the Riesz fractional derivative in the following
\begin{definition} (see e.g. \cite{Samko87})The Riesz fractional derivative for $0<\alpha<2$ and for $-\infty< x< \infty$ is defined as $\triangledown^\alpha g(x) = -c_\alpha(_{-\infty}\triangledown^\alpha_x + {_x}\triangledown^\alpha_\infty)g(x),$
  where
\begin{align*}
c_\alpha &=\frac{1}{2cos(\frac{\pi\alpha}{2})},\,\alpha\neq 1,\\
_{-\infty}\triangledown^\alpha_x g(x) &= \frac{1}{\Gamma(n-\alpha)}\frac{d^n}{dx^n}\int_{-\infty}^{x}\frac{g(u)}{(x-u)^{\alpha+1-n}}du,\\
{_x}\triangledown^\alpha_\infty g(x) &=(-1)^n \frac{1}{\Gamma(n-\alpha)}\frac{d^n}{dx^n}\int_{x}^{\infty}\frac{g(u)}{(u-x)^{\alpha+1-n}}du,
\end{align*}
\end{definition}
\noindent where $n=\lfloor \textrm{Re}(\alpha)\rfloor+1.$
 
For a function $g$ satisfying the integrability condition
\begin{equation}\label{integrability condition}
\int_{\mathbb{R}}\frac{|g(y)|}{(1+|y|)^{1+\alpha}}dy<\infty,
\end{equation}
we define as in \cite{bogdan2003} that
\begin{equation}
\triangle_{\epsilon}^{\frac{\alpha}{2}}g(x)=\mathcal{A}(1,-\alpha)\int_{|y-x|>\epsilon}\frac{g(y)-g(x)}{|y-x|^{1+\alpha}}dy,
\end{equation}
and
\begin{equation}\label{fl}
\triangle^{\frac{\alpha}{2}}g(x)=\mathcal{A}(1,-\alpha)P.V.\int_{\mathbb{R}}\frac{g(y)-g(x)}{|y-x|^{1+\alpha}}dy:=\displaystyle\lim_{\epsilon\to 0^+}\triangle_{\epsilon}^{\frac{\alpha}{2}}g(x),
\end{equation}
whenever the limit exists. Here ``P.V." stands for the``principal value". The above limit exists and is finite if $g$ is of class $C^2$ in a neighborhood of $x$ and satisfies condition (\ref {integrability condition}); in this case
\[
\triangle^{\frac{\alpha}{2}}g(x)=\mathcal{A}(1,-\alpha)\int_{\mathbb{R}}\frac{g(y)-g(x)-\triangledown g(x). (y-x)1_{\{|y-x|<\epsilon\}}}{|y-x|^{1+\alpha}}dy\]
for any $\epsilon>0$, where $\mathcal{A}(1,-\alpha)=\frac{\alpha 2^{\alpha-1}\Gamma(\frac{\alpha+1}{2})}{\pi^{\frac{1}{2}}\Gamma(1-\frac{\alpha}{2})}$. Here $\triangle^{\frac{\alpha}{2}}$, is the fractional power of the Laplace operator $-(-\triangle)^{\frac{\alpha}{2}}$. The Fractional Laplacian is the infinitesimal generator of $\alpha$ stable process (see e.g. \cite{Landkof1972}). The Fractional Laplacian is usually defined by its Fourier transform (cf. \cite{stein1970}): $\mathcal{F}((-\triangle)^{\frac{\alpha}{2}}g)(\xi)=|\xi|^{\alpha}\mathcal{F}(g)(\xi)$, its proof can be found in \cite{Landkof1972}. Hence, $-(-\triangle)^{\frac{\alpha}{2}}g(\xi) = - \mathcal{F}^{-1}\Bigl(|\xi|^{\alpha}\mathcal{F}(g)(\xi)\Bigr).$ The definition we used in (\ref{fl}) coincides with the usual Fractional Laplacian defined by its Fourier transform (see \cite{adams03}).

By the Fourier transform method, one could show that the following relation holds (cf. \cite{Feller1952}):
\[
-(-\triangle)^{\frac{\alpha}{2}}g(x) = \triangledown^{\alpha}g(x).\]
 
 
\subsection{$p$-variation of local time of symmetric stable process}
We present the exact modulus of local time of stable processes from \cite{Barlow88} in the following theorem. Recall its characteristic function is given by \begin{equation*}
\chi(\theta)=|\theta|^\alpha +ih|\theta|^\alpha sgn(\theta),
\end{equation*}
where $|h|\leq$ tan$(\alpha\pi/2).$ We define its local time as $L_t^x$.
 
\begin{theorem}\label{barlowTheorem}
For a symmetric stable process of index $\alpha>1$, its local time satisfies
 
\begin{equation}
\lim_{\delta\downarrow 0}\sup_{\substack{|a-b|<\delta\\ a,b\in I\\ 0\leq s\leq t}}\frac{|L_s^a-L_s^b|}{|b-a|^{\frac{\alpha-1}{2}}\biggl({\rm{log}}\bigl(\frac{1}{|b-a|}\bigr)\biggr)^{1/2}}=\frac{2{c_{\alpha}}^{1/2}}{(1+h^2)^{1/2}}\biggl(\sup_{x\in I}L_t^x\biggr)^{1/2},
\end{equation}
for all intervals $I \subseteq \mathbb{R}$ and all $t>0$ a.s., where
\begin{equation}
c_{\alpha}=\frac{2}{\pi}\int_{0}^{\infty}(1-cosy)y^{-\alpha}dy=\frac{1}{\pi}\frac{\Gamma(2-\alpha)}{\alpha-1} sin\frac{(2-\alpha)\pi}{2}.
\end{equation}

\end{theorem}
 
Barlow \cite {Barlow88} gave a neccessary and sufficient condition for the joint continuity of the local time, and the exact modulus for a fairly wide class of L\'evy processes. 

Let $[a,b]$ be any finite interval. By its dyadic decompositions we mean partitions $\{a_0^n<a_1^n<\dots<a_{2^n}^n\}$ of $[a,b]$, where
\[
a_k^n= a + \frac{k}{2^n}(b-a),\,\, k\, =\,0, 1,\dots, 2^n,\]
and $n\in \mathbb{N}.$
 
 
\begin{proposition}\label {lemmaPvariation}
The family of local times $L_t^x$ of $\alpha$-stable processes with $\alpha\in (1,2)$ is of bounded $p$-variation in $x$ for any $t\geq 0$, and any $p>\frac{2}{\alpha-1}$ almost surely.
\end{proposition}
\begin{proof}
First from Theorem \ref{barlowTheorem}, we know that for almost all $\omega\in\Omega$, there exists a $\delta^*(\omega)>0$ such that when $0<|b-a|<\delta^*(\omega)$, we have the following\\
\begin{equation}\label{eq: 2.1}
{|{L_s}^a-{L_s}^b|}^p\
\leq \quad \Biggl|b-a\Biggr|^{\frac{(\alpha-1)p}{2}}\Biggl({\rm{log}}\biggl(\frac{1}{|b-a|}\biggr)\Biggr)^{\frac{p}{2}}\frac{2^{p+1} {c_{\alpha}}^{\frac{p}{2}}}{\bigl(1+h^2\bigr)^{\frac{p}{2}}}\Biggl(\sup_{x\in I} L_t^x\Biggr)^{\frac{p}{2}}.
\end{equation}
We can construct a dyadic decomposition which fully covers the interval $[a,b]$. Furthermore, by Proposition 4.1.1 from \cite{Lyons2002} (set $i=1, \gamma= p-1$), for any partition $\{a_l\}$ of $[a,b]$, we have
\begin{equation} \label{eq: 2.2}
\sup_{D}\sum_{l}|L_t^{a_{l+1}}-L_t^{a_{l}}|^p \leq C\bigl(p,\gamma\bigr)\sum_{n=1}^{\infty}n^{\gamma}\sum_{k=1}^{2^n}\Biggl|L_t^{a_k^n}-L_t^{a_{k-1}^n}\Biggr|^p,
\end{equation}
where $C\bigl(p,\gamma\bigr)$ is a constant depending on $p$ and $\gamma.$ Notice the right hand side of (\ref {eq: 2.2}) does not depend on partition $D$.
 
 We can choose an $n_0$ big enough such that $\frac{|b-a|}{2^n}\leq\delta^*(\omega)$ for any $n\geq n_0$. By substituting (\ref {eq: 2.1}) into (\ref {eq: 2.2}), it follows that
\begin{eqnarray*}\nn
&&
\quad\,\,\sum_{n=1}^{\infty}n^{\gamma}\sum_{k=1}^{2^n}\Biggl|L_t^{a_k^n}-L_t^{a_{k-1}^n}\Biggr|^p\nn\\
&&
\leq\sum_{n=1}^{n_0-1}n^{\gamma}\sum_{k=1}^{2^n}\Biggl|L_t^{a_k^n}-L_t^{a_{k-1}^n}\Biggr|^p\nn\\
&&
\quad+\sum_{n=n_0}^{\infty}n^{\gamma}\sum_{k=1}^{2^n}\Biggl|a_k^n-a_{k-1}^n\Biggr|^{\frac{(\alpha-1)p}{2}}\Biggl({\rm{log}}\biggl(\frac{1}{|a_k^n-a_{k-1}^n|}\biggr)\Biggr)^{\frac{p}{2}}\frac{2^{p+1} {c_{\alpha}}^{\frac{p}{2}}}{\bigl(1+h^2\bigr)^{\frac{p}{2}}}\Biggl(\sup_{x\in I} L_t^x\Biggr)^{\frac{p}{2}}.\nn\\
\end{eqnarray*}
In \cite{Barlow85}, Barlow showed that
\begin{equation}\label{eq: 2.3}
\Biggl(\sup_{x\in I} L_t^x\Biggr)^{\frac{p}{2}}<\infty, \quad a.s..
\end{equation}

Now, we consider the term 
\begin{equation*}
\sum_{n=n_0}^{\infty}n^{\gamma}\sum_{k=1}^{2^n}\Biggl|a_k^n-a_{k-1}^n\Biggr|^{\frac{(\alpha-1)p}{2}}\Biggl({\rm{log}}\biggl(\frac{1}{|a_k^n-a_{k-1}^n|}\biggr)\Biggr)^{\frac{p}{2}}
\end{equation*}
where $a_k^n=\frac{k}{2^n}(b-a)+a$, $k=0, 1,\cdots, 2^n$, then it is not difficult to see that
\begin{eqnarray*}\nn
&&
\quad\sum_{n=n_0}^{\infty}n^{\gamma}\sum_{k=1}^{2^n}\Biggl|a_k^n-a_{k-1}^n\Biggr|^{\frac{(\alpha-1)p}{2}}\Biggl({\rm{log}}\biggl(\frac{1}{|a_k^n-a_{k-1}^n|}\biggr)\Biggr)^{\frac{p}{2}}\nn\\
&&
\label{eq: 2.4}
=\sum_{n=n_0}^{\infty}n^{\gamma}\sum_{k=1}^{2^n}\Biggl|a_k^n-a_{k-1}^n\Biggr|^{\frac{(\alpha-1)p}{4}+\frac{1}{2}}\Biggl|a_k^n-a_{k-1}^n\Biggr|^{\frac{(\alpha-1)p}{4}-\frac{1}{2}}\Biggl({\rm{log}}\biggl(\frac{1}{|a_k^n-a_{k-1}^n|}\biggr)\Biggr)^{\frac{p}{2}}\nn\\
&&
\leq \beta_1 \sum_{n=n_0}^{\infty}n^{\gamma}\biggl(\frac{b-a}{2^n}\biggr)^{\frac{(\alpha-1)p}{4}-\frac{1}{2}}\nn\\
&&
<\infty,
\end{eqnarray*}
where $\beta_1$ is a constant depends only on $p$. It turns out for any interval $[a,b] \subset$$\mathbb{R}$
\begin{equation*}
\sup_{D}\sum_{m}|L_t^{a_{m+1}}-L_t^{a_{m}}|^p< \infty.
\end{equation*}
As $L_t^a(\omega)$ has a compact support in $a$ for each $\omega$, say $[-N,N]$ contains its support.\,\,We still denote its partition by $D: = D_{-N,N} =\{-N=x_0<x_1<\cdots<x_t =N\}$ and attain that
\begin{equation}
\sup_{D}\sum_{l}|L_t^{a_{l+1}}-L_t^{a_{l}}|^p< \infty.
\end{equation}
\end{proof}
\subsection{The Young integral with respect to local time}
We start with functions which are smooth, then proceed to functions which are less smooth. The following works for sufficiently smooth functions, unless we explicitly say otherwise. It\^o's formula for L\'evy process in general (cf.\,\,\cite{applebaum04}) is given by
\begin{eqnarray}\nn
&&
g(X_t)= g(X_0)+\int_{0}^{t} \triangledown g(X_s)d{X_s}+\frac{1}{2}\sigma^2\int_{0}^{t}\triangle g(X_s)ds\nn\\
&&\qquad\qquad+\int_{0}^{t}\int_{\mathbb{R}}\biggl[g(X_{s-}+y)-g(X_{s-})\biggr]1_{\{|y|\geq1\}}{N}(dy,ds)\nn\\
&&\qquad\qquad+\int_{0}^{t}\int_{\mathbb{R}}\biggl[g(X_{s-}+y)-g(X_{s-})\biggr]1_{\{|y|<1\}}\tilde{N}(dy,ds)\nn\\
&&
\qquad\qquad+ \int_{0}^{t}\int_{\mathbb{R}} \biggl[g(X_{s-}+y)-g(X_{s-})-y\triangledown g(X_{s-})\biggr]1_{\{|y|<1\}}\nu(dy)ds,\nn\\
\end{eqnarray}
where $\tilde{N}(dy,ds)$ is the compensated Poisson measure defined as the difference of Poisson point measure $N(dy,ds)$ and its intensity measure $\nu(dy)ds$. And, the L\'evy measure of stable process $\nu(dx)$ is
\[
\nu(dx) = C|x|^{-\alpha -1}dx,\]
where $C$ is a constant.\,\,The last term of the above It\^o's formula can be further simplified by using the definition of fractional Laplacian as
\begin{eqnarray*}\nn
&&
\quad\int_{\mathbb{R}\backslash\left\{0\right\}} \biggl[g(X_{s-}+y)-g(X_{s-})-y\triangledown g(X_{s-})\biggr]1_{\{|y|<1\}}\nu(dy)\nn\\
&&
=C\int_{\mathbb{R}\backslash\left\{0\right\}} \frac{g(X_{s-}+y)-g(X_{s-})-y\triangledown g(X_{s-})1_{|y|<1}}{|y|^{\alpha+1}}dy -C\int_{|y|\geq1} \frac{g(X_{s-}+y)-g(X_{s-})}{|y|^{\alpha+1}}dy \nn\\
&&
=C_{\alpha}\triangle^{\frac{\alpha}{2}}g(X_{s-})-\int_{|y|\geq1} [g(X_{s-}+y)-g(X_{s-})]\nu({dy})
\end{eqnarray*}
for every $g\in C_b^2(\mathbb{R})$, where $C_{\alpha} = \frac{C\pi^{\frac{1}{2}}\Gamma(1-\frac{2}{\alpha})}{\alpha 2^{\alpha-1}\Gamma(\frac{1+\alpha}{2})}$.
 
Then, one can derive the following It\^o's formula for stable processes from (2.18) based on the fact that $\sigma =0$ for stable processes which are pure jump processes and the definition of the compensated Poisson measure
\begin{eqnarray}\nn
&&
g(X_t)= g(X_0)+\int_{0}^{t} \triangledown g(X_s)d{X_s}\nn\\
&&\qquad\qquad+\int_{0}^{t}\int_{\mathbb{R}\backslash\{0\}}\biggl[g(X_{s-}+y)-g(X_{s-})\biggr] \tilde{N}(dy,ds)\nn\\
&&
\qquad\qquad +C_{\alpha} \int_{0}^{t} \triangle^{\frac{\alpha}{2}}g(X_{s-})ds.
\end{eqnarray}
By the occupation times formula, one can show that
\begin{equation}
\int_{0}^{t} \triangledown^{\alpha}g(X_{s-})ds = \int_{-\infty}^{\infty} \triangledown^{\alpha}g(x) L_t^xdx =\int_{-\infty}^{\infty}  L_t^xd_x\bigl(\triangledown^{\alpha-1}g(x)\bigr).
\end{equation}

Note the last integral $\int_{-\infty}^{\infty}  L_t^xd_x\bigl(\triangledown^{\alpha-1}g(x)\bigr)$ can be defined without the need to assume that $g\in C^2$. In fact, if $\triangledown^{\alpha-1}g(x)$ is of finite $q$-variation ($1\leq q <\frac{2}{3-\alpha}$), then the integral $ \int_{-\infty}^{\infty} \triangledown^{\alpha-1}g(x) d_xL_t^x$ is well defined as a Young integral. Similar to \cite{Feng2006}, we have the following remark.

\begin{remark}\label{updatedremark} If ${ \triangledown^{\alpha-1} g(x)}$ is a $C^1$ function, then $\int_{-\infty}^{\infty} {\triangledown^{\alpha-1} g(x)} d_x L_t^x$ exists as a Riemann integral and we have
\begin{equation}
\int_{-\infty}^{\infty} {\triangledown^{\alpha-1} g(x)} d_x L_t^x = - \int_{-\infty}^{\infty} L_t^x  d_x \,\bigl({\triangledown^{\alpha-1}g(x)}\bigr).
\end{equation}
\end{remark}
In fact, since $L_t^.$ has a  compact support for each $t$, one can always add some points in the partition to make $L_t^{x_1} = 0$ and $L_t^{x_m} = 0.$ Then, we can show that
\begin{eqnarray}\nn
&&
\quad\int_{-\infty}^{\infty} {\triangledown^{\alpha-1} g(x)} d_x L_t^x\nn\\
&&
=\lim_{m(D)\to 0} \sum_{j=1}^{m} {\triangledown^{\alpha-1} g(x_{j-1})} \bigl(L_t^{x_j} - L_t^{x_{j-1}}\bigr)\nn\\
&&
=\lim_{m(D)\to 0} \biggl[\sum_{j=1}^{m} {\triangledown^{\alpha-1} g(x_{j-1})} L_t^{x_j} - \sum_{j=0}^{m-1} {\triangledown^{\alpha-1} g(x_{j})} L_t^{x_j}\biggr]\nn\\
&&
= - \lim_{m(D)\to 0} \sum_{j=1}^{m}  \bigl( {\triangledown^{\alpha-1} g(x_{j})}-{\triangledown^{\alpha-1} g(x_{j-1})}\bigr) L_t^{x_j} \nn\\
&&
=- \int_{-\infty}^{\infty} L_t^x  d_x \,\bigl({\triangledown^{\alpha-1} g(x)}\bigr).
\end{eqnarray}
In the following, we will consider the integral for less smooth functions. For this, we define a mollifier
\begin{equation}
\rho(x)=\left\{
\begin{array}{lc}
  ce^{\frac{1}{(x-1)^2-1}},\quad if\,{x\in(0,2),} \\ \\0,  \,\,\,\,\qquad\quad\quad otherwise. \\
\end{array}
\right.
\end{equation}

Here $c$ is chosen so that $\int_{0}^{2}\rho(x)dx=1.$ Take $\rho_n(x)=n\rho(nx)$ as the mollifier which will be used to smootherise less smooth functions.

In this paper, we extend It\^o's formula for less smooth function $f:\mathbb{R}\to\mathbb{R}$ which is absolutely continuous. Such function $f$ has a $(\alpha-1)^{th}$ fractional derivative which is assumed to be left continuous and is of finite $q$ variation, where $1\leq q<4$. And, we denote the left limit of this $(\alpha-1)^{th}$ fractional derivative of $f$ as $\triangledown_-^{{\alpha-1}} f(x)$.

\begin{theorem}\label{LessSmoothTheorem}  Let $f: \mathbb{R}\to \mathbb{R}$ be an absolutely continuous,\,locally bounded function,\,have the $(\alpha-1)^{th}$ fractional derivative 
which is left continuous and of finite $q$-variation, where $1\leq q <\frac{2}{3-\alpha}$.\,Define $ f_n(x) = \int_{-\infty}^{\infty}\rho_n(x-y)f(y)dy\,\text{with}\,\, n\geq 1$.\,Then
\begin{equation}\label{eq: 2.15}
\int_{-\infty}^{\infty}{ \triangledown^{\alpha-1} f_n(x)}d_x L_t^x \to \int_{-\infty}^{\infty}{ \triangledown^{\alpha-1}_- f(x)}d_x L_t^x,\quad\hbox{as}\quad n\to\infty.
\end{equation}
\end{theorem}
\begin{proof} Note that $f_n(x)$ can be rewritten as
\[
f_n(x) = \int_{0}^{2}\rho(z)f(x-\frac{z}{n})dz, \quad n\geq 1,\]
and note it is smooth.
In particular
\begin{equation}\label{fd}
\triangledown^{\alpha-1} f_n(x) =  \int_{0}^{2}\rho(z)\triangledown^{{\alpha-1}}f(x-\frac{z}{n})dz, \quad n\geq 1.
\end{equation}

To see this, one can use Fubini's theorem and the absolutely continuity property of the function $f$ to get
\begin{eqnarray}
&&
_{-\infty} \triangledown^{\alpha-1}_x  f_n(x) = \frac{1}{\Gamma(2-\alpha)}\frac{d}{dx}\biggl(\int_{-\infty}^{x}\frac{f_n(t)}{(x-t)^{\alpha-1}}dt\biggr)\nn\\
&&
\,\,\qquad\quad\qquad\quad=\frac{1}{\Gamma(2-\alpha)}\frac{d}{dx}\biggl(\int_{0}^{\infty}\frac{f_n(x-t)}{t^{\alpha-1}}dt\biggr)\nn\\
&&
\,\,\qquad\quad\qquad\quad=\frac{1}{\Gamma(2-\alpha)}\biggl(\int_{0}^{\infty}\frac{f'_n(x-t)}{t^{\alpha-1}}dt\biggr)\nn\\
&&
\,\,\qquad\quad\qquad\quad=\frac{1}{\Gamma(2-\alpha)}\biggl(\int_{-\infty}^{x}(x-t)^{1-\alpha}{\int_{0}^{2}\rho(z)f'(t-\frac{z}{n})}dzdt\biggr)\nn\\
&&
\,\,\qquad\quad\qquad\quad=\frac{1}{\Gamma(2-\alpha)}\biggl(\int_{0}^{2}\rho(z)\int_{-\infty}^{x}\frac{f'(t-\frac{z}{n})}{(x-t)^{\alpha-1}}dtdz\biggr)\nn\\
&&
\,\,\qquad\quad\qquad\quad=\frac{1}{\Gamma(2-\alpha)}\biggl(\int_{0}^{2}\rho(z)\frac{d}{dx}\bigl(\int_{-\infty}^{x}\frac{f(t-\frac{z}{n})}{(x-t)^{\alpha-1}}dt\bigr)dz\biggr)\nn\\
&&
\,\,\qquad\quad\qquad\quad=\int_{0}^{2}\rho(z){_{-\infty}} \triangledown^{\alpha-1}_xf(x-\frac{z}{n})dz.
\end{eqnarray}
Similarly, we can derive $
_x\triangledown^{\alpha-1}_{\infty} f_n(x) = \int_{0}^{2}\rho(z){_{x}} \triangledown^{\alpha-1}_{\infty}f(x-\frac{z}{n})dz.$
Hence, we have proved (\ref{fd}).

Similar to \cite{Young38}, for any partition $D:=\{-N=x_0<x_1<\cdots<x_r=N\}$, there is an increasing function $\textbf{w}$ such that
\[\bigl|{\triangledown^{\alpha-1} f(x_{l+1})}-{\triangledown^{\alpha-1} f(x_{l})}\bigr|\leq \bigl(\textbf{w}(x_{l+1})-\textbf{w}(x_{l})\bigr)^{\frac{1}{q}},\quad x_l,\,x_{l+1}\in D,\]
where $\textbf{w}(x)$ is the total $q$-variation of $\triangledown^{\alpha-1} f(x)$ in the interval [-N-2, $x$].
 
Then, by Jensen's inequality, we obtain
\begin{eqnarray}\nn
&&
\quad\sup_D \sum_{l=1}^{r}\bigl|{\triangledown^{\alpha-1} f_n(x_l)}-{\triangledown^{\alpha-1} f_n(x_{l-1})}\bigr|^q\nn\\
&&
=\sup_D \sum_{l=1}^{r}\biggl|\int_{0}^{2}\rho(z)\bigl[{ \triangledown^{\alpha-1} f(x_l-\frac{z}{n})}-{ \triangledown^{\alpha-1} f(x_{l-1}-\frac{z}{n})}\bigr]dz\biggr|^q\nn\\
&&
\leq M_1 \sup_D \sum_{l=1}^{r}\biggl(\int_{0}^{2}\bigl|{\triangledown^{\alpha-1} f(x_l-\frac{z}{n})}-{ \triangledown^{\alpha-1} f(x_{l-1}-\frac{z}{n})}\bigr|^q dz\biggr)\nn\\
&&
\leq M_1\int_{0}^{2}\sup_D \sum_{l=1}^{r}\bigl|{ \triangledown^{\alpha-1} f(x_l-\frac{z}{n})}-{ \triangledown^{\alpha-1} f(x_{l-1}-\frac{z}{n})}\bigr|^q dz\nn\\
&&
\leq M_1 \int_{0}^{2} \biggl(\textbf{w}(N-\frac{z}{n})-\textbf{w}(-N-\frac{z}{n})\biggr)dz,\nn\\
\end{eqnarray}
where $M_1$ is a constant. In addition, as we have
\begin{equation}
\textbf{w}(N-\frac{z}{n})-\textbf{w}(-N-\frac{z}{n})\leq \textbf{w}(N),
\end{equation}
it follows that
\begin{equation}
\sup_D \sum_{l=1}^{r}\bigl|{\triangledown^{\alpha-1} f_n(x_l)}-{ \triangledown^{\alpha-1} f_n(x_{l-1})}\bigr|^q\leq 2M_1 \textbf{w}(N)<\infty.
\end{equation}
This implies that ${\triangledown^{\alpha-1} f_n(x)}$ is of bounded $q$-variation in $x$ uniformly in $n$. \noindent Moreover,\,by Lebesgue's dominated convergence theorem and (2.27), we have that
\begin{equation}
 \triangledown^{\alpha-1} f_n(x) \to \triangledown_-^{\alpha-1} f(x)\quad\text{as}\quad n\to\infty.
\end{equation}Now the theorem follows from Theorem \ref{TermbytermIntegration} immediately.\\
\end{proof}
We present the It\^o's formula for stable processes defined in terms of Young integral in the following theorem.
\begin{theorem}\label{MainYoungTheorem} Let  $X = (X_t)_{t\geq0}$ be a symmetric $\alpha$-stable process, $1<\alpha<2$, and $f: \mathbb{R} \to \mathbb{R}$ be an absolutely continuous, locally bounded function that has $(\alpha-1)^{th}$ fractional order derivative ${{\triangledown_-^{{\alpha-1}}} f(x)}$. Assume that ${\triangledown_-^{\alpha-1} f(x)}$ is locally bounded, and of bounded $q$-variation, where $1\leq q < \frac{2}{3-\alpha}$. Then we have the following It\^o's formula

\begin{eqnarray}\nn
&&
f(X_t) = f(X_0) + \int_{0}^{t}\triangledown_- f(X_s)dX_s\nn\\
&&
\qquad\quad\quad+ \int_{0}^{t}\int_{\mathbb{R}}\biggl(f(X_{s-} +y)-f(X_{s-})\biggr) \tilde{N}(dy,ds)-{C_{\alpha}}\int_{-\infty}^{\infty}\triangledown_-^{\alpha-1}f(x)d_x L_t^x,\nn\\
\end{eqnarray}
where $C_{\alpha} = \frac{\pi^{\frac{1}{2}}\Gamma(1-\frac{2}{\alpha})}{\alpha 2^{\alpha-1}\Gamma(\frac{1+\alpha}{2})}$.
\end{theorem}
 
\begin{proof} Define a smooth function $f_n$ as in Theorem \ref{LessSmoothTheorem}. We apply It\^o's formula (2.19) for stable process to the smooth function $f_n$.\,Then when we take the limit as $n\to\infty$, the convergence of all terms except the fractional Laplacian term are clear.\,For the fractional Laplacian term, we use the occupation times formula, Remark \ref{updatedremark} and equation (\ref{eq: 2.15}), to obtain that
\begin{equation}
\int_{0}^{t}\triangle^{\frac{\alpha}{2}} f_n(X_{s-}) ds
= - \int_{-\infty}^{\infty} \triangledown^{\alpha-1}f_n(x) d_x L_t^x \to  - \int_{-\infty}^{\infty} \triangledown_-^{\alpha-1}f(x) d_x L_t^x.
\end{equation}
\end{proof}

\section{Local time as rough path}
For a function $g$ is of bounded $q$-variation for $1\leq q< \frac{2}{3-\alpha}$ and local time is of bounded $p$-variation $(p>\frac{2}{\alpha-1})$, one can find a real number $p>\frac{2}{\alpha-1}$ such that $ \frac{1}{p}+\frac{1}{q}>1$, hence $\int_{-\infty}^{\infty}g(x)d_x L_t^x$ can be defined as a Young integral. But when $q\geq\frac{2}{3-\alpha}$, the condition  $ \frac{1}{p}+\frac{1}{q}>1$ is not satisfied, hence one can no longer use Young integral to define  $\int_{-\infty}^{\infty}g(x)d_x L_t^x$. Moreover, $\int_{}L^x_tdL^x_t$ cannot be defined as a Young integral as pointed out in  \cite{Feng2008,Feng2010}. However, the rough path integration theory can provide a way to overcome this obstacle. In the rest of the paper, we deal with the case where $\frac{3}{2}<\alpha<2$ and $\frac{2}{3-\alpha}\leq q<4$. For this, one needs to treat $Z_x :=(L_t^x, g(x))$ as a process of variable $x$ in $\mathbb{R}^2$.  In this case, $Z_x$ is of bounded $\hat q$-variation in $x$, where $\hat q =max\{p,q\}$ with $p\in(\frac{2}{\alpha -1},4)$. 
 
In the following,  we will first construct a continuous and bounded path $Z(m)$ from $Z_.$ on a certain time interval. The smooth rough path ${\bf {Z}}(m)$ is thus built by taking its iterated integrals with respect to $Z(m)$. The final stage is to show the existence of the geometric rough path $\bf{Z}=(1, \bf{Z}^1,\bf{Z}^2, \bf{Z}^3)$ associated with $Z_.$. 
 
In order to show the existence of the geometric rough path $\bf{Z}$, we need to prove that the space of rough path we have defined is complete under the $\theta$-variation distance which was pointed out in Lemma 3.3.3 in \cite{Lyons2002}\begin{equation}
d_{\theta}(X, Y) = \displaystyle\max_{1\leq i\leq[\theta]}d_{i,\theta}(X^i,Y^i)=\displaystyle\max_{1\leq i\leq[\theta]}\displaystyle\sup_{D}\biggl(\displaystyle\sum_l{|{X_{x_{l-1},x_l}^i - Y_{x_{l-1},x_l}^i}|}^{\frac{\theta}{i}}\biggr)^{\frac{i}{\theta}}.
\end{equation}
Therefore, the strategy consists of verifying that the smooth rough path ${\bf {Z}}(m)$ is a Cauchy sequence in the $\theta$-variation metric $d_{\theta}$ on $C_{0,\theta}(\Delta, T^{([\theta])}(\mathbb{R}^2))$.
%
 

We recall the following lemma from \cite{Marcusr2006}.
\begin{lemma}\label{key lemma for fifth term}
Let $X=\{X(t),t\in \mathbb{R}_+\}$ be a real-valued symmetric stable process of index $1<\alpha\leq 2$ and  $\{L_t^x, (t,x)\in\mathbb{R}_+\times\mathbb{R}\}$ be the local time of $X$.\,Then, for all $x,y\in\mathbb{R}$, and integers $m\geq1$, 

\begin{equation}||L_t^x-L_t^y||_{2m}\leq C(\alpha,m) t^{\frac{\alpha-1}{2\alpha}}|x-y|^{\frac{\alpha-1}{2}},\end{equation}
where $C(\alpha,m)$ is constant depending on $\alpha$ and m.
\end{lemma}


We obtain that for any $p>\frac{2}{\alpha-1}$, the following relation as a special case of Lemma 3.1
\begin{equation}\label{eq: 4.1}
E|L_t^b - L_t^a|^p\leq c|b - a|^{\frac{(\alpha-1)p}{2}},
\end{equation}
 
\noindent with a constant $c>0$.\,\,This means that $L_t^x$ satisfies the H$\ddot {\rm o}$lder condition with exponent $\frac{\alpha-1}{2}$.  A control function $\textbf{w}$ is a non-negative continuous function on the simplex $\Delta:=\{(a,b) : x'\leq a<b\leq x''\}$ with values in $[0,\infty)$ such that $\textbf{w}(a,a)=0.$ It is super-additive, namely
\begin{equation}
\textbf{w}(a,b) + \textbf{w}(b,c)\leq\textbf{w}(a,c),
\end{equation}
for any $(a,b), (b,c)\in\Delta.$
In the case when $g(x)$ is of bounded $q$-variation, one has a control $\textbf{w}$ such that
\begin{equation}
|g(b)-g(a)|^q\leq\textbf{w}(a,b),
\end{equation}
for any $(a,b)\in\Delta.$ It is clear that $\textbf{w}_1(a,b):=\textbf{w}(a,b)+(b-a)$ is also a control of $g$.  For $h=\frac{1}{\hat{q}}$ and  $\theta>\hat{q}$, one can verify that
\begin{equation}
|{g(b)-g(a)|^{\theta}\leq\textbf{w}_1(a,b)^{h\theta}}
\end{equation}
for any $(a,b)\in\Delta$ and $h\theta>1.$ Hence, there exists a constant $c$ such that
\begin{equation}\label{eq: 4.2}
E|Z_b-Z_a|^{\theta}\leq c\textbf{w}_1(a,b)^{h\theta}, \quad\forall (a,b)\in\Delta.
\end{equation}

Following the idea in \cite{Feng2008}, one could define a continuous and bounded variation path $Z(m)$ on $[x' , x'']$ for any $m\in N$ by
\begin{equation}\label{eq: 4.3}
Z(m)_x: = Z_{x_{l-1}^m} + \frac{\textbf{w}_1(x) -\textbf{w}_1(x_{l-1}^m)}{\textbf{w}_1(x_{l}^m)-\textbf{w}_1(x_{l-1}^m)}\Delta_l^mZ,
\end{equation}
where $x_{l-1}^m\leq x<x_l^m$ with $l=1,\dots,2^m$, 
 and $\Delta_l^mZ= Z_{x_l^m} - Z_{x_{l-1}^m}$. Take a partition $D_m:=\{x'=x_0^m<x_1^m<\dots<x_{2^m}^m=x''\}$ of $[x',x'']$ such that 
\begin{equation}\label{usedlater}
\textbf{w}_1(x_l^m)-\textbf{w}_1(x_{l-1}^m) =\frac{1}{2^m}\textbf{w}_1(x',x''),\end{equation}
where $\textbf{w}_1(x):=\textbf{w}_1(x',x).$ 
In addition, by the superadditivity of the control function $\textbf{w}_1$, it is clear that
\[
\textbf{w}_1(x_{l-1}^m,x_l^m)\leq \textbf{w}_1(x_l^m)-\textbf{w}_1(x_{l-1}^m) = \frac{1}{2^m}\textbf{w}_1(x',x'').\]
The smooth rough path ${\bf {Z}}(m)$ associated with $Z(m)$ is constructed by taking its iterated path integrals. That is
\begin{equation}\label{eq: 4.4}
{\bf{Z}}(m)_{a,b}^j = \displaystyle\int_{a<x_1<\dots<x_j<b}dZ(m)_{x_1}\otimes\dots\otimes dZ(m)_{x_j}
\end{equation}
for any $(a,b)\in\Delta$, where $j=0, 1, 2, 3$. 

We will need the slightly modified version of Proposition 4.1.1 from \cite{Lyons2002}.
\begin{proposition}\label{ThetaVariationProposition}
Let $Z\in C_0(\Delta, T^{(N)}(V))$ be a multiplicative functional with a fixed running time interval, say $[0,1].$ Then for any $1\leq i\leq N,$ $\theta$ satisfying $\theta/i>1,$ and any $\gamma>\theta/i-1,$ there exists a constant $ C_i(\theta,\gamma)$ depending only on $\theta, \gamma,$ and $i$, such that
\begin{equation}
\displaystyle\sup_{D}\displaystyle\sum_{l}|Z_{x_{l-1}, x_l}^i|^{\theta/i}\leq C_i(\theta,\gamma)\sum_{n=1}^{\infty}n^\gamma\sum_{k=1}^{2^n}\sum_{j=1}^{i}|Z_{x_{k-1}^n, x_k^n}^j|^{\theta/j},
\end{equation}
where $\displaystyle\sup_D$ runs over all finite partitions D of $[0,1],$ and ${x^k_n}$ satisfies equation (\ref{usedlater}). 
\end{proposition}
The aim of the remaining part of this section is to prove that $\{{\bf{Z}}(m)\}_{m\in N}$ converges to a geometric rough path $\bf{Z}$ in the $\theta$-variation topology. 

\subsection{First level path}
We first consider the convergence of the first level path ${\bf{Z}}(m)_{a,b}^1$.   
\begin{proposition}\label{FirstLevelProposition}
Let $(Z_x)$ be a continuous path and $h\theta\geq1$, ${\bf{Z}}(m)$ be defined as above. Then for all $n\in N$ \begin{equation}
m\to \sum_{k=1}^{2^n}|{\bf{Z}}(m)_{x_{k-1}^n,x_k^n}^1|^{\theta}
\end{equation}
is increasing. Hence,
\begin{equation}
\displaystyle\sup_m\sum_{k=1}^{2^n}|{\bf{Z}}(m)_{x_{k-1}^n,x_k^n}^1|^{\theta} =\displaystyle\lim_{m\to\infty}\sum_{k=1}^{2^n}|{\bf{Z}}(m)_{x_{k-1}^n,x_k^n}^1|^{\theta}.
\end{equation}
\end{proposition}
\begin{proof}
By  (\ref{eq: 4.3}) and (\ref{eq: 4.4}), we can derive for $n\leq m$ 
\[
{\bf{Z}}(m)_{x_{k-1}^n,x_k^n}^1 = \Delta_k^n Z, \quad k=1,\dots,2^n.
\]
On the other hand, if $n> m$, it is possible to find a unique integer $1\leq l \leq 2^m$ satisfying
\begin{equation}\label{eq: 4.5}
x_{l-1}^m\leq x_{k-1}^n< x_k^n<x_l^m.
\end{equation}
Based on  (\ref{eq: 4.3}) and (\ref{eq: 4.4}), one can get
\begin{equation}
Z(m)_{x_j^n}=Z_{x_{l-1}^m} + \frac{\textbf{w}_1(x_j^n)-\textbf{w}_1(x_{l-1}^m)}{\textbf{w}_1(x_{l}^m)-\textbf{w}_1(x_{l-1}^m)}\Delta_l^mZ,\quad j=k-1, k.
\end{equation}
\noindent It turns out that
\begin{equation}
{\bf{Z}}(m)_{x_{k-1}^n,x_{k}^n}^1 = Z(m)_{x_{k}^n}-Z(m)_{x_{k-1}^n}=2^{m-n}\Delta_l^m Z,\quad \forall n>m.
\end{equation}
\noindent For $n>m$, from the inequality (\ref{eq: 4.5}), we can compute the range for the integer $k$ for a given integer $l$. That is $2^{n-m}(l-1)+1\leq k< 2^{n-m}l$. In other words, there are $2^{n-m}$ points of  the form of $\{x_k^n\}_{2^{n-m}(l-1)+1\leq k<2^{n-m}l}$ embedded inside $[x_{l-1}^m, x_l^m).$ Therefore, for $n>m$,
\begin{equation}
\sum_{k=1}^{2^n}|{\bf{Z}}(m)_{x_{k-1}^n,x_k^n}^1|^{\theta}= \biggl(\frac{1}{2^n}\biggr)^{\theta-1}(2^m)^{\theta-1}\sum_{l=1}^{2^m}|\Delta_l^m Z|^{\theta}.
\end{equation}
It is interesting to notice that
\[
\Delta_l^m Z = \Delta _{2l}^{m+1}Z+\Delta _{2l-1}^{m+1}Z,\]
which gives
\begin{eqnarray}\nn
&&
\quad (2^m)^{\theta-1}\sum_{l=1}^{2^m}|\Delta_l^m Z|^{\theta}\nn\\
&&
=(2^{m+1})^{\theta-1}\sum_{l=1}^{2^{m}}\biggl(\frac{1}{2}\biggr)^{\theta-1}|\Delta _{2l}^{m+1}Z+\Delta _{2l-1}^{m+1}Z|^{\theta}\nn\\
&&
\leq (2^{m+1})^{\theta-1}\sum_{l=1}^{2^{m}}\biggl(|\Delta _{2l}^{m+1}Z|^{\theta}+|\Delta _{2l-1}^{m+1}Z|^{\theta}\biggr)\nn\\
&&
= (2^{m+1})^{\theta-1}\sum_{l=1}^{2^{m+1}}|\Delta _{l}^{m+1}Z|^{\theta}.
\end{eqnarray}
This proves the claim.
\end{proof}
As a consequence of 
Proposition \ref{FirstLevelProposition}, one can show that ${\bf {Z}}(m)_{x',x''}^1$ on any finite interval have finite $\theta$-variations uniformly in $m$ using a similar method in the proof of  Proposition 4.3.1 in \cite{Lyons2002}.
\begin{proposition}
For a continuous path $Z_x$ satisfying (\ref{eq: 4.2}) and $h\theta>1$.\,Then $ {\bf{Z}}(m)_{x',x''}^1$ have finite $\theta$ variation  uniformly in m.
\end{proposition}
 
We present the convergence result of the first level path in the next theorem.\,\,Let ${\bf{Z}}_{a,b}^1= Z_b-Z_a.$ By (\ref{eq: 4.2}), one can show that $E|{\bf{Z}}_{a,b}^1|^{\theta}\leq c\textbf{w}_1(a,b)^{h\theta}.$ In particular, $E|{\bf{Z}}_{x_{k-1}^n,x_{k}^n}^1|^{\theta}\leq c\textbf{w}_1(x_{k-1}^n,x_{k}^n)^{h\theta}\leq c\bigl(\frac{1}{2^n}\bigr)^{h\theta}\textbf{w}_1(x',x'')^{h\theta}.$
 
\begin{theorem}\label{FirstLevelMainTheorem}
For $h\theta>1$, if a continuous path $Z_x$ satisfying the inequality (\ref{eq: 4.2}), then we have
\begin{equation}
\sum_{m=1}^{\infty}\displaystyle\sup_D\biggl(\displaystyle\sum_l|{\bf{Z}}(m)_{x_{l-1},x_l}^1-{\bf{Z}}_{x_{l-1},x_l}^1|^{\theta}\biggr)^{\frac{1}{\theta}}<\infty\quad a.s..
\end{equation}
In particular, ${\bf{Z}}(m)_{a,b}^1$ converges to ${\bf{Z}}_{a,b}^1$ in the $\theta$-variation distance almost surely for any $(a,b)\in\Delta.$
\end{theorem}
\begin{proof}
For $n\leq m$, we have ${\bf{Z}}(m)_{x_{k-1}^n,x_{k}^n}^1={\bf{Z}}_{x_{k-1}^n,x_{k}^n}^1$, while if $n> m$ then
\begin{equation*}
|{\bf{Z}}(m)_{x_{k-1}^n,x_{k}^n}^1-{\bf{Z}}_{x_{k-1}^n,x_{k}^n}^1|^{\theta}\leq 2^{\theta-1}\biggl(|{\bf{Z}}(m)_{x_{k-1}^n,x_{k}^n}^1|^{\theta} + |{\bf{Z}}_{x_{k-1}^n,x_{k}^n}^1|^{\theta}\biggr).
\end{equation*}
By 
(\ref{eq: 4.2}) and Proposition \ref{ThetaVariationProposition}, we have
\begin{eqnarray*}\label{eq: p12}
&&
\quad E\sum_{m=1}^{\infty}\displaystyle\sup_D\biggl(\displaystyle\sum_l|{\bf{Z}}(m)_{x_{l-1},x_l}^1-{\bf{Z}}_{x_{l-1},x_l}^1|^{\theta}\biggr)^{\frac{1}{\theta}}\nn\\
&&
\leq C(\theta,\gamma) \sum_{m=1}^{\infty} \biggl(E\sum_{n=m+1}^{\infty} n^{\gamma}\sum_{k=1}^{2^n}|{\bf{Z}}(m)_{x_{k-1}^n,x_{k}^n}^1-{\bf{Z}}_{x_{k-1}^n,x_{k}^n}^1|^{\theta}\biggr)^{\frac{1}{\theta}}\nn\\
&&
\leq C \sum_{m=1}^{\infty} \biggl(E\sum_{n=m+1}^{\infty} n^{\gamma}\sum_{k=1}^{2^n}|{\bf{Z}}(m)_{x_{k-1}^n,x_{k}^n}^1|^{\theta}+|{\bf{Z}}_{x_{k-1}^n,x_{k}^n}^1|^{\theta}\biggr)^{\frac{1}{\theta}}
\end{eqnarray*}

\begin{eqnarray}\label{eq: p12}
&&
\leq C \sum_{m=1}^{\infty} \biggl(\sum_{n=m+1}^{\infty} n^{\gamma}\biggl(\frac{1}{2^n}\biggr)^{h\theta-1}\textbf{w}_1(x',x'')^{h\theta}\biggr)^{\frac{1}{\theta}}\nn\\
&&
\leq C\sum_{m=1}^{\infty}\biggl(\frac{1}{2^m}\biggr)^{\frac{{h\theta-1}}{2\theta}}  \sum_{n=m+1}^{\infty} n^{\frac{\gamma}{\theta}}\biggl(\frac{1}{2^n}\biggr)^{\frac{{h\theta-1}}{2\theta}}\nn\\
&&
\leq C \sum_{m=1}^{\infty}\biggl(\frac{1}{2^m}\biggr)^{\frac{{h\theta-1}}{2\theta}}\nn\\
&&
<\infty,
\end{eqnarray}
for $h\theta>1,$ where $C$ is a generic constant depending on $\theta, h, \textbf{w}_1(x',x'')$ and $c$ in (\ref{eq: 4.2}).\,\,This completes the proof.
\end{proof}
 
\subsection{Second level path}
Next, we consider the convergence of second level path ${\bf{Z}}_{a,b}^2$. From \cite{Lyons2002}, for $n\geq m$,
\begin{equation}\label{eq: 4.6}
{\bf{Z}}(m)_{x_{k-1}^n, x_k^n}^N = \frac{1}{N!}2^{N(m-n)}\bigl(\Delta_l^mZ\bigr)^{\otimes N}
\end{equation}
for all level paths with $N=1,2,\dots$. For the second level path, we take $N=2.$ In the case of $n< m$,
\begin{eqnarray}
&&
\quad {\bf{Z}}(m)_{x_{k-1}^n, x_k^n}^2\nn\\
&&
=\frac{1}{2}\Delta_k^n Z\otimes \Delta_k^n Z+\frac{1}{2}\displaystyle\sum_{l=2^{m-n}(k-1)+1}^{2^{m-n}k}\displaystyle\sum_{r=2^{m-n}(k-1)+1}^{l}\bigl(\Delta_r^m Z\otimes \Delta_l^m Z - \Delta_l^m Z\otimes \Delta_r^m Z\bigr),
\end{eqnarray}
therefore,
\begin{eqnarray}\label{eq: 4.7}
&&
\quad {\bf{Z}}(m+1)_{x_{k-1}^n, x_k^n}^2-{\bf{Z}}(m)_{x_{k-1}^n, x_k^n}^2\nn\\
&&
=\frac{1}{2}\displaystyle\sum_{l=2^{m-n}(k-1)+1}^{2^{m-n}k}\bigl(\Delta_{2l-1}^{m+1} Z\otimes \Delta_{2l}^{m+1} Z - \Delta_{2l}^{m+1} Z\otimes \Delta_{2l-1}^{m+1} Z\bigr),
\end{eqnarray}
where $k=1,\dots,2^n.$
 
We first give the result for the second level path ${\bf{Z}}(m)_{a,b}^2$ when $n\geq m$.
\begin{proposition}\label{SecondLevelVariationProposition}
For a continuous path $Z_x$ which satisfies (\ref{eq: 4.2}) with $h\theta>1$, then for $n\geq m$
\begin{equation}
\sum_{k=1}^{2^n}E\biggl|{\bf{Z}}(m+1)_{x_{k-1}^n, x_k^n}^2-{\bf{Z}}(m)_{x_{k-1}^n, x_k^n}^2\biggr|^{\frac{\theta}{2}}\leq C\biggl(\frac{1}{2^{n+m}}\biggr)^{\frac{h\theta-1}{2}},
\end{equation}
 
\end{proposition}
\noindent where $C$ is a generic constant that depends on  $\theta, h, \textbf{w}_1(x',x'')$, and $c$ in (\ref{eq: 4.2}).
\begin{proof} For $n\geq m$, it follows from (\ref{eq: 4.6})
\begin{eqnarray}\nn
&&
\quad\sum_{k=1}^{2^n}E\biggl|{\bf{Z}}(m+1)_{x_{k-1}^n, x_k^n}^2-{\bf{Z}}(m)_{x_{k-1}^n, x_k^n}^2\biggr|^{\frac{\theta}{2}}\nn\\
&&
=\sum_{l=1}^{2^{m+1}}\displaystyle\sum_{x_{l-1}^{m+1}\leq x_{k-1}^n<x_l^{m+1}}E\biggl|\frac{1}{2}2^{2(m+1-n)}\bigl(\Delta_l^{m+1}Z\bigr)^{\otimes 2}-\frac{1}{2}2^{2(m-n)}\bigl(\Delta_l^mZ\bigr)^{\otimes 2}\biggr|^{\frac{\theta}{2}}\nn\\
&&
=\sum_{l=1}^{2^{m+1}}2^{n-m-1}E\biggl|\frac{1}{2}2^{2(m+1-n)}\bigl(\Delta_l^{m+1}Z\bigr)^{\otimes 2}-\frac{1}{2}2^{2(m-n)}\bigl(\Delta_l^mZ\bigr)^{\otimes 2}\biggr|^{\frac{\theta}{2}}\nn\\
&&
\leq C\biggl(\frac{2^m}{2^n}\biggr)^{\theta} \sum_{l=1}^{2^{m+1}}2^{n-m-1}\biggl(\frac{1}{2^m}\biggr)^{h\theta}\textbf{w}_1(x',x'')^{h\theta}\nn\\
&&
\leq C\biggl(\frac{2^m}{2^n}\biggr)^{\theta-h\theta}  \biggl(\frac{1}{2^n}\biggr)^{h\theta-1}\nn\\
&&
\leq C\biggl(\frac{2^m}{2^n}\biggr)^{\theta-h\theta}  \biggl(\frac{1}{2^n}\biggr)^{\frac{h\theta-1}{2}}\biggl(\frac{1}{2^m}\biggr)^{\frac{h\theta-1}{2}}\nn\\
&&
\leq C\biggl(\frac{1}{2^{m+n}}\biggr)^{\frac{h\theta-1}{2}},
\end{eqnarray}
where $C$ is a generic constant depending on $\theta, h, \textbf{w}_1(x',x'')$, and $c$ in (\ref{eq: 4.2}).
\end{proof}
 
The proof of the above result in the case when $n< m$ is more involved as suggested by (\ref{eq:  4.7}).\,\,In order to establish the convergence of the second level paths, it is crucial to estimate $\displaystyle\sum_i E(L_t^{x_{i+1}}-L_t^{x_i})(L_t^{x_{j+1}}-L_t^{x_j})$, and to obtain the correct order in terms of the increments $x_{j+1}-x_j$ as suggested in \cite{Feng2008}. This point will be made clear through the proof of the convergence of the second level path. 
 
 
First, define
\begin{equation}\label{eq: 4.8}
\sigma^2(h) = E\bigl(L_t^{x+h}-L_t^{x}\bigr)^2 
\end{equation}
and a covariance matrix
\begin{equation*}
\rho_{i,j}(D)= E\bigl(L_t^{x_i}-L_t^{x_{i-1}}\bigr)\bigl(L_t^{x_j}-L_t^{x_{j-1}}\bigr),
\end{equation*}
where $D =\{x_i\}_i$ is a partition of a given interval.\,By (\ref {eq: 4.8}), using the same elementary algebraic manipulation, one can deduce that for $i<j$
\begin{eqnarray}\nn\label{eq: 4.9}
&&
\rho_{i,j} = -\frac{1}{2}\bigl[\sigma^2(x_{j-1}-x_{i-1})-\sigma^2(x_{j-1}-x_{i})\bigr]\nn\\
&&
\quad\quad\quad\quad + \frac{1}{2}\bigl[\sigma^2(x_{j}-x_{i-1})-\sigma^2(x_{j}-x_{i})\bigr].
\end{eqnarray}

It was proved in \cite{Marcusr92} that
\begin{equation}\label{macineq}
\displaystyle\sum_i\bigl|\rho_{i,j}(D)\bigr|\leq \frac{1}{2}\sigma^2(x_j-x_{j-1}).
\end{equation}
for the Gaussian case 
which is purely based on the concavity and monotonicity of $\sigma^2$. As the function $ |.|^{\alpha-1}$ for $\frac{3}{2}<\alpha<2$ is both concave and monotone,\,therefore, the inequality (\ref{macineq}) is also applicable for the local time of stable process. Hence, by (\ref{eq: 4.1}) and (\ref{macineq}), it follows that 
\begin{align}\label{isomorphism ineq}
\displaystyle\sum_i\bigl|\rho_{i,j}(D)\bigr|\leq \frac{1}{2}\sigma^2(x_j-x_{j-1})
\leq C_1  |x_j-x_{j-1}|^{\alpha-1}
\end{align}
where  $C_1$ is a constant related to the constant c in (\ref{eq: 4.1}). 
Now we are in the position to prove the following proposition
\begin {proposition}\label{SecondLevelConvergence}
Let $\frac{2}{3-\alpha}\leq q<3$,  $2\leq\theta<3$.\,Then for $n< m$
\begin{equation}
E\bigl|{\bf Z}(m+1)_{x_{k-1}^n, x_k^n}^2 -{\bf Z}(m)_{x_{k-1}^n, x_k^n}^2\bigr|^{\frac{\theta}{2}}\leq C\Biggl[\biggl(\frac{1}{2^n}\biggr)^{\frac{\theta}{4}}\biggl(\frac{1}{2^{m}}\biggr)^{\frac{2h+\alpha-2}{4}\theta}\Biggr],
\end{equation}
\end{proposition}
\noindent where $C$ is a generic constant depends on $\theta, h, \textbf{w}_1(x',x'')$, $C_1$ and $c$ in (\ref {eq: 4.2}).
\begin{proof}
First, by (\ref{eq: 4.7}), we have
\begin{eqnarray*}\nn
&&
\quad E\bigl|{\bf Z}(m+1)_{x_{k-1}^n, x_k^n}^2 -{\bf Z}(m)_{x_{k-1}^n, x_k^n}^2\bigr|^2\nn\\
&&
=\frac{1}{4}E\biggl|\sum_{l=2^{m-n}(k-1)+1}^{2^{m-n}k}\bigl(\Delta_{2l-1}^{m+1}Z\otimes\Delta_{2l}^{m+1}Z - \Delta_{2l}^{m+1}Z \otimes\Delta_{2l-1}^{m+1}Z\bigr)\biggr|^2\nn\\
&&
=\frac{1}{4}\sum_{i,j=1,{i\neq j}}^{2}\biggl[E\displaystyle\sum_{l}\bigl(\Delta_{2l-1}^{m+1}Z^i\Delta_{2l}^{m+1}Z^j - \Delta_{2l}^{m+1}Z^i \Delta_{2l-1}^{m+1}Z^j\bigr)^2\nn\\
&&
\quad\quad\quad\quad\quad+2 E\displaystyle\sum_{r<l}\bigl(\Delta_{2l-1}^{m+1}Z^i\Delta_{2l}^{m+1}Z^j - \Delta_{2l}^{m+1}Z^i \Delta_{2l-1}^{m+1}Z^j\bigr)\nn\\
&&
\quad\quad\quad\quad\quad\quad\quad\times\quad\quad\bigl(\Delta_{2r-1}^{m+1}Z^i\Delta_{2r}^{m+1}Z^j - \Delta_{2r}^{m+1}Z^i \Delta_{2r-1}^{m+1}Z^j\bigr)\biggr]\nn\\
&&
=\frac{1}{4}\displaystyle\sum_{l}E\biggl[\bigl(\Delta_{2l-1}^{m+1}L_t^x\Delta_{2l}^{m+1}g(x)\bigr)^2 -2\Delta_{2l-1}^{m+1}L_t^x\Delta_{2l}^{m+1}g(x)\Delta_{2l}^{m+1}L_t^x\Delta_{2l-1}^{m+1}g(x)\nn\\
&&
\quad\quad\quad\quad\quad\quad\quad\quad\quad\quad+\bigl( \Delta_{2l}^{m+1}L_t^x\Delta_{2l-1}^{m+1}g(x)\bigr)^2\biggr]\nn\\
&&
\,\,\,\,\,+\frac{1}{4}\displaystyle\sum_{l}E\biggl[\bigl(\Delta_{2l-1}^{m+1}g(x)\Delta_{2l}^{m+1}L_t^x\bigr)^2 -2\Delta_{2l-1}^{m+1}g(x)\Delta_{2l}^{m+1}L_t^x\Delta_{2l}^{m+1}g(x)\Delta_{2l-1}^{m+1}L_t^x\nn\\
&&
\quad\quad\quad\quad\quad\quad\quad\quad\quad\quad+\bigl(\Delta_{2l}^{m+1}g(x)\Delta_{2l-1}^{m+1}L_t^x\bigr)^2\biggr]\nn\\
&&
\,\,\,\,\,+\frac{1}{2} \displaystyle\sum_{r<l}\biggl(E\big(\Delta_{2l-1}^{m+1}L_t^x\Delta_{2r-1}^{m+1}L_t^x\bigr)(\Delta_{2l}^{m+1}g(x)\Delta_{2r}^{m+1}g(x)\bigr)\nn\\
&&
\quad\quad\quad\quad\quad\quad\quad\quad +E\big(\Delta_{2l}^{m+1}L_t^x\Delta_{2r}^{m+1}L_t^x\bigr)(\Delta_{2l-1}^{m+1}g(x)\Delta_{2r-1}^{m+1}g(x)\bigr)\biggr)\nn\\
&&
\,\,\,\,\,-\frac{1}{2} \displaystyle\sum_{r<l}\biggl(E\big(\Delta_{2l}^{m+1}L_t^x\Delta_{2r-1}^{m+1}L_t^x\bigr)(\Delta_{2l-1}^{m+1}g(x)\Delta_{2r}^{m+1}g(x)\bigr)\nn\\
&&
\quad\quad\quad\quad\quad\quad\quad\quad +E\bigl(\Delta_{2l-1}^{m+1}L_t^x\Delta_{2r}^{m+1}L_t^x\bigr)(\Delta_{2l}^{m+1}g(x)\Delta_{2r-1}^{m+1}g(x)\bigr)\biggr)\nn\\
&&
\,\,\,\,\,-\frac{1}{2}\displaystyle\sum_{r<l}\biggl(\big(\Delta_{2l}^{m+1}g(x)\Delta_{2r-1}^{m+1}g(x)\bigr)E(\Delta_{2l-1}^{m+1}L_t^x\Delta_{2r}^{m+1}L_t^x\bigr)\nn\\
&&
\quad\quad\quad\quad\quad\quad\quad\quad +\bigl(\Delta_{2l-1}^{m+1}g(x)\Delta_{2r}^{m+1}g(x)\bigr)E\big(\Delta_{2l}^{m+1}L_t^x\Delta_{2r-1}^{m+1}L_t^x\bigr)\biggr)\nn\\
&&
\,\,\,\,\,+\frac{1}{2} \displaystyle\sum_{r<l}\biggl(\bigl(\Delta_{2l-1}^{m+1}g(x)\Delta_{2r-1}^{m+1}g(x)\bigr)E\big(\Delta_{2l}^{m+1}L_t^x\Delta_{2r}^{m+1}L_t^x\bigr)\nn\\
&&
\quad\quad\quad\quad\quad\quad\quad\quad +\bigl(\Delta_{2l}^{m+1}g(x)\Delta_{2r}^{m+1}g(x)\bigr)E\bigl(\Delta_{2l-1}^{m+1}L_t^x\Delta_{2r-1}^{m+1}L_t^x\bigr)\biggr).
\end{eqnarray*}
We estimate the following term using (\ref{isomorphism ineq})
\begin{eqnarray}\nn
&&
\quad\displaystyle\sum_{r<l}\biggl|\bigl(\Delta_{2l}^{m+1}g(x)\Delta_{2r}^{m+1}g(x)\bigr)E\bigl(\Delta_{2l-1}^{m+1}L_t^x\Delta_{2r-1}^{m+1}L_t^x\bigr)\biggr|\nn\\
&&
=\sum_{l=2^{m-n}(k-1)+1}^{2^{m-n}k}\biggl|\Delta_{2l}^{m+1}g(x)\biggr|\sum_{r=1}^{l-1} \biggl|\Delta_{2r}^{m+1}g(x)\biggr| E\biggl|\Delta_{2l-1}^{m+1}L_t^x\Delta_{2r-1}^{m+1}L_t^x\biggr|\nn\\
&&
\leq \sum_{l=2^{m-n}(k-1)+1}^{2^{m-n}k}C\biggl(\frac{1}{2^{m+1}}\biggr)^h\omega_1(x',x'')^h\sum_{r=1}^{l-1}\biggl(\frac{1}{2^{m+1}}\biggr)^h\omega_1(x',x'')^hE\bigl(|\Delta_{2l-1}^{m+1}L_t^x\Delta_{2r-1}^{m+1}L_t^x|\bigr)\nn\\
&&
\leq C\sum_{l}\biggl(\frac{1}{2^{m+1}}\biggr)^{2h}\sum_{r}E\bigl(|\Delta_{2l-1}^{m+1}L_t^x\Delta_{2r-1}^{m+1}L_t^x|\bigr)\nn\\
&&
\leq C\sum_{l}\biggl(\frac{1}{2^{m+1}}\biggr)^{2h} |x_{2l-1}^{m+1}-x_{2l-2}^{m+1}|^{\alpha-1}\nn\\
&&
\leq C\Biggl[ 2^{m-n}\biggl(\frac{1}{2^{m+1}}\biggr)^{2h+\alpha-1}\Biggr]\nn\\
&&
=C \Biggl[\frac{1}{2^n}\biggl(\frac{1}{2^{m}}\biggr)^{2h+\alpha-2}\Biggr].
\end{eqnarray}
\noindent The other terms can be estimated similarly.
It then follows from Jesen's inequality that
\begin{eqnarray}
&&
\quad E\bigl|{\bf Z}(m+1)_{x_{k-1}^n, x_k^n}^2 -{\bf Z}(m)_{x_{k-1}^n, x_k^n}^2\bigr|^{\frac{\theta}{2}}\nn\\
&&
\leq \biggl(E\bigl|{\bf Z}(m+1)_{x_{k-1}^n, x_k^n}^2 -{\bf Z}(m)_{x_{k-1}^n, x_k^n}^2\bigr|^2\biggr)^{\frac{\theta}{4}}\nn\\
&&
\leq C\Biggl[\biggl(\frac{1}{2^n}\biggr)^{\frac{\theta}{4}}\biggl(\frac{1}{2^{m}}\biggr)^{\frac{2h+\alpha-2}{4}\theta}\Biggr].
\end{eqnarray}
\end{proof}
By Propositions 3.6 and 3.7, we showed the convergence of the second level path.\,The convergence result is presented in the next theorem. As $\theta>max\{p,q\}$, where $p>\frac{2}{\alpha-1}$ and $q\geq\frac{2}{3-\alpha}$ is the variation of the local time associated with the symmetric stable process and of the function $g$ respectively, together with the fact that $\frac{2}{3-\alpha}<\frac{2}{\alpha-1}$ holds as long as $\alpha<2$, therefore, the smallest possible value that $\alpha$ can take must satisfy $\alpha>\frac{2}{\theta}+1$. As $\theta$ can be chosen very close to 3, hence, the smallest possible value of $\alpha$ that we can take for the second level path is $\alpha>\frac{5}{3}$.\,\,This means for any $\alpha\in(\frac{5}{3},2)$, there exists a $\theta\in(2,3)$ such that $\alpha>\frac{2}{\theta}+1$.
\begin{theorem}\label{them 3.4.4}
Let  $\frac{5}{3}<\alpha<2$,  $\frac{2}{3-\alpha}\leq q <3$. 
Then for a continuous path $Z_x$ satisfing (\ref{eq: 4.2}), there exists a unique $\bf Z^2$ on the simplex $\triangle$ taking values in $\mathbb{R}^2\otimes \mathbb{R}^2$ such that
\begin{equation}
\displaystyle\sup_{D}\biggl(\displaystyle\sum_{l}\bigl|{\bf Z}(m)_{x_{l-1},x_l}^2-{\bf Z}_{x_{l-1},x_l}^2\bigr|^{\frac{\theta}{2}}\biggr)^{\frac{2}{\theta}}\to 0,
\end{equation}
both almost surely and in $L^1(\Omega,\mathcal{F},\mathcal{P})$ as $m\to\infty$, for some $\theta$ such that $\frac{4}{2h+\alpha-1}<\theta<3$.
\end{theorem}
\begin{proof}
By Proposition 4.1.2 in \cite{Lyons2002}, we have
\begin{eqnarray}
&&
\quad E\displaystyle\sup_{D}\sum_{l}\bigl|{\bf Z}(m+1)_{x_{l-1},x_l}^2 - {\bf Z}(m)_{x_{l-1},x_l}^2 \bigr|^{\frac{\theta}{2}}\nn\\
&&
\leq C(\theta,\gamma)E\biggl(\sum_{n=1}^{\infty}n^{\gamma}\sum_{k=1}^{2^n}\bigl|{\bf Z}(m+1)_{x_{k-1}^n,x_k^n}^1 - {\bf Z}(m)_{x_{k-1}^n,x_k^n}^1\bigr|^{\theta}\biggr)^{\frac{1}{2}}\nn\\
&&
\quad\quad\quad\times\biggl(\sum_{n=1}^{\infty}n^{\gamma}\sum_{k=1}^{2^n}\bigl|{\bf Z}(m+1)_{x_{k-1}^n,x_k^n}^1\bigr|^{\theta}+ \bigl|{\bf Z}(m)_{x_{k-1}^n,x_k^n}^1\bigr|^{\theta}\biggr)^{\frac{1}{2}}\nn\\
&&
\quad\quad\quad+ C(\theta,\gamma)E\sum_{n=1}^{\infty}n^{\gamma}\sum_{k=1}^{2^n}\bigl|{\bf Z}(m+1)_{x_{k-1}^n,x_k^n}^2 - {\bf Z}(m)_{x_{k-1}^n,x_k^n}^2\bigr|^{\frac{\theta}{2}}\nn\\
&&
:=A + B.
\end{eqnarray}

We have proved the convergence of the first level path in Theorem \ref{FirstLevelMainTheorem}.  The result from Theorem \ref{FirstLevelMainTheorem} is used to estimate the part A, that is
\begin{eqnarray}
&&
A\leq C \biggl(E\sum_{n=1}^{\infty}n^{\gamma}\sum_{k=1}^{2^n}\bigl|{\bf Z}(m+1)_{x_{k-1}^n,x_k^n}^1-{\bf Z}_{x_{k-1}^n,x_k^n}^1\bigr|^{\theta}+\bigl| {\bf Z}(m)_{x_{k-1}^n,x_k^n}^1-{\bf Z}_{x_{k-1}^n,x_k^n}^1\bigr|^{\theta}\biggr)^{\frac{1}{2}}\nn\\
&&
\quad\quad\quad\times\biggl(E\sum_{n=1}^{\infty}n^{\gamma}\sum_{k=1}^{2^n}\bigl|{\bf Z}(m+1)_{x_{k-1}^n,x_k^n}^1\bigr|^{\theta}+ \bigl|{\bf Z}(m)_{x_{k-1}^n,x_k^n}^1\bigr|^{\theta}\biggr)^{\frac{1}{2}}\nn\\
&&
\quad\leq C\Biggl[\biggl(\frac{1}{2^m}\biggr)^{\frac{h\theta-1}{4}}\biggl(\sum_{n=1}^{\infty}n^\gamma\biggl(\frac{1}{2^n}\biggr)^{{h\theta-1}}\biggr)^\frac{1}{2}\Biggr]\nn\\
&&
\quad\leq C \biggl(\frac{1}{2^m}\biggr)^{\frac{h\theta-1}{4}}.
\end{eqnarray}


For $\frac{5}{3}<\alpha<2$ and $h>\frac{1}{3}$, we can choose $\theta$ satisfies $\frac{4}{2h+\alpha-1}<\theta<3$. 
Therefore, $\frac{2h+\alpha-2}{4}\theta>1-\frac{\theta}{4}$. Hence, we can choose an $\epsilon$ such that $1-\frac{\theta}{4}<\epsilon<\frac{2h+\alpha-2}{4}\theta$.
Then by Proposition \ref{SecondLevelVariationProposition} and \ref{SecondLevelConvergence}, it follows that 
\begin{eqnarray*}\nn
&&
B\leq C\sum_{n=m}^{\infty}n^\gamma\biggl(\frac{1}{2^{m+n}}\biggr)^{\frac{h\theta-1}{2}} + C\sum_{n=1}^{m-1}n^\gamma\biggl(\frac{1}{2^{n}}\biggr)^{\frac{\theta}{4}-1}\biggl(\frac{1}{2^{m}}\biggr)^{\frac{2h+\alpha-2}{4}\theta}\nn\\
&&
\quad\leq C\Biggl[\biggl(\frac{1}{2^{m}}\biggr)^{\frac{h\theta-1}{2}} +\sum_{n=1}^{m-1}n^\gamma\biggl(\frac{1}{2^{n}}\biggr)^{\frac{\theta}{4}-1 + \epsilon}\biggl(\frac{1}{2^{m}}\biggr)^{\frac{2h+\alpha-2}{4}\theta-\epsilon}  \Biggr]\nn\\
&&
\quad\leq C\Biggl[\biggl(\frac{1}{2^{m}}\biggr)^{\frac{h\theta-1}{2}} +\biggl(\frac{1}{2^{m}}\biggr)^{\frac{2h+\alpha-2}{4}\theta-\epsilon} \Biggr].
\end{eqnarray*}
Hence, we proved the convergence of the term B. With the above observation and the fact that $h\theta>1$, it is clear that
\begin{equation*}
E\displaystyle\sup_{D}\sum_{l}\bigl|{\bf Z}(m+1)_{x_{l-1},x_l}^2 - {\bf Z}(m)_{x_{l-1},x_l}^2 \bigr|^{\frac{\theta}{2}}\leq C\Biggl[\biggl(\frac{1}{2^{m}}\biggr)^{\frac{h\theta-1}{2}} +\biggl(\frac{1}{2^{m}}\biggr)^{\frac{2h+\alpha-2}{4}\theta-\epsilon}  \Biggr].
\end{equation*}

If we sum up for all $m$ as we did in Theorem \ref{FirstLevelMainTheorem}, one can show that $\bigl({\bf Z}(m)^2\bigr)_{m\in N}\in (\mathbb{R}^2)^{\otimes 2}$ is a Cauchy sequence in $\theta$-variation distance. In other words, it has a limit as $m\to\infty$, denote it by ${\bf Z}^2\in (\mathbb{R}^2)^{\otimes 2}$. By Lemma 3.3.3 in \cite{Lyons2002}, we can conclude that ${\bf Z}^2$ is also finite under $\theta$-variation distance. Thus, we have proved the theorem.
\end{proof}

\noindent As local time $L_t^x$ has a compact support for each $\omega$ and $t$, so the integral of local time in $\mathbb{R}$ can be defined.\,\,We take $[x',x'']$ which contains the support of $L_t^x$.\,\,By Chen's identity, one can see that for any $(a,b)\in\Delta$,
\begin{equation}
{\bf Z}_{a,b}^2 = \displaystyle\lim_{m(D_{[a,b]})\to 0}\sum_{i=0}^{r-1}({\bf Z}_{x_{i-1},x_i}^2 + {\bf Z}_{a, x_i}^1\otimes {\bf Z}_{x_{i-1},x_i}^1).
\end{equation}

In particular, similar to the proof in \cite{Feng2010}, we have
\begin{eqnarray}
&&
({\bf Z}_{a,b}^2)_{2,1} = \displaystyle\lim_{m(D_{[a,b]})\to 0}\sum_{i=0}^{r-1}(({\bf Z}_{x_{i-1},x_i}^2)_{2,1} +( {\bf Z}_{a, x_i}^1\otimes {\bf Z}_{x_{i-1},x_i}^1)_{2,1})\nn\\
&&
\quad\quad\quad\quad=\displaystyle\lim_{m(D_{[a,b]})\to 0}\sum_{i=0}^{r-1}(({\bf Z}_{x_{i-1},x_i}^2)_{2,1} +( g(x_i)-g(a))(L_t^{x_i}-L_t^{ x_{i-1}}).
\end{eqnarray}
Here $({\bf Z}_{x_{i-1},x_i}^2)_{2,1} $ denotes the lower-left element of the $2\times2$ matrix ${\bf Z}_{x_{i-1},x_i}^2$. Hence, the following
\begin{eqnarray}
&&
\quad\displaystyle\lim_{m(D_{[a,b]})\to 0}\sum_{i=0}^{r-1}(({\bf Z}_{x_{i-1},x_i}^2)_{2,1} + g(x_i)(L_t^{x_i}-L_t^{ x_{i-1}})\nn\\
&&
=\displaystyle\lim_{m(D_{[a,b]})\to 0}\sum_{i=0}^{r-1}(({\bf Z}_{x_{i-1},x_i}^2)_{2,1} +  ( g(x_i)-g(a))(L_t^{x_i}-L_t^{ x_{i-1}})) + g(a)(L_t^{b}-L_t^{a})\nn\\
\end{eqnarray}
holds.
   \,Therefore, we have the following corollary.
\begin{corollary}\label{roughpathintegraldef}
Under the same conditions of the previous theorem, then 
for $(a,b)\in\Delta$,
\begin{equation}
\int_{a}^{b} L_t^xdL_t^x = \displaystyle\lim_{m(D_{[a,b]})\to 0}\sum_{i=0}^{r-1}(({\bf Z}_{x_{i-1},x_i}^2)_{1,1} + L_t^{x_i}(L_t^{x_i}-L_t^{ x_{i-1}})).
\end{equation}
Moreover, if g is a continuous function with bounded $q$-variation, $q\geq\frac{2}{3-\alpha}$, then we have
\begin{equation}
\int_{a}^{b} g(x)dL_t^x = \displaystyle\lim_{m(D_{[a,b]})\to 0}\sum_{i=0}^{r-1}(({\bf Z}_{x_{i-1},x_i}^2)_{2,1} + g(x_i)(L_t^{x_i}-L_t^{ x_{i-1}})).
\end{equation}
\end{corollary}

\subsection{Convergence of rough path integrals for the second level path}
In this section, we will prove the convergence of the second level path in the $\theta$-variation topology.
\begin{proposition}\label{continuitysecondlevel}
Let $\frac{5}{3}<\alpha<2$, $\frac{2}{3-\alpha}\leq q<3$, one can choose a $\theta$ such that $\frac{4}{2h+\alpha-1}<\theta<3$. 
Moreover, let $Z_j(x): = (L_t^x, g_j(x)), Z(x):=(L_t^x, g(x))$, where $g_j(\cdot), g(\cdot)$ are both continuous and of bounded $q$-variation.\,Suppose $g_j(x)\to g(x)$ as $j\to\infty$ uniformly and the control function ${\bf{w}}_j(x,y)$ of $g_j$ converges to the control function ${\bf{w}}(x,y)$ of $g$ as $j\to\infty$ uniformly.\,Then the geometric rough path ${\bf Z}_{j}(\cdot)$ associated with $Z_{j}(\cdot)$ converges to the geometric rough path ${\bf Z}(\cdot)$ associated with $Z(\cdot)$ a.s. in the $\theta$-variation topology as $j\to\infty$. In particular, $\int_{-\infty}^{\infty} g_{j}(x)dL_t^x\to \int_{-\infty}^{\infty} g(x)dL_t^x$ a.s. as $j\to\infty.$
\end{proposition}
\begin{proof} 
For each $j$, one can obtain the geometric rough path ${\bf Z}_j(\cdot)$ associated with $Z_j(\cdot)$, and also the smooth rough path ${\bf Z}_j(m)$ in the same way as ${\bf Z}(m)$.  Here the ${\bf Z}_j$ is defined as  ${\bf Z}_j = (1, {\bf Z}_j^1, {\bf Z}_j^2).$ 
\,Similarly, we have ${\bf Z}_j(m) = (1, {\bf Z}_j^1(m), {\bf Z}_j^2(m)).$ 
\,First, we prove the convergence of ${\bf Z}_j^1\to{\bf Z}^1$  in the $\theta$-variation topology and in the uniform topology. To see this, we consider for any finite interval $[x', x'']$ in $\mathbb{R}$. As local time $L_t^x$ has a compact support in $x$ a.s., so the following proof can be extended to $\mathbb{R}$.\,To prove that ${\bf {Z}}_{j}^1\to\bf Z^1$ as $j\to\infty$ in the $d_{2,\theta}$ topology, note first 

\begin{equation}\label{eq: A1}
d_{2,\theta}({\bf Z}_j^1,{\bf Z}^1)\leq d_{2,\theta}({\bf Z}_j^1,{\bf Z}_j^1(m))+d_{2,\theta}({\bf Z}_j^1(m),{\bf Z}^1(m))+d_{2,\theta}({\bf Z}^1(m),{\bf Z}^1).
\end{equation}
From  (\ref{eq: p12}), we know that $d_{2,\theta}({\bf Z}^1(m),{\bf Z}^1)\to 0$ as $m\to\infty$ and $d_{2,\theta}({\bf Z}_j^1,{\bf Z}_j^1(m))\to 0$ as $m\to\infty$ uniformly in $j$. Thus there exists an integer $m_0$ such that $d_{2,\theta}({\bf Z}^1(m_0),{\bf Z}^1)<\frac{\epsilon}{3}$ and  $d_{2,\theta}({\bf Z}_j^1,{\bf Z}_j^1(m_0))<\frac{\epsilon}{3}$. Consider ${\bf Z}_j^1(m_0)$ and ${\bf Z}^1(m_0)$, which are bounded variation processes and  ${\bf Z}_j^1(m_0)(x)\to{\bf Z}^1(m_0)(x)$ as $j\to\infty$ uniformly in $x$. Moreover,
\begin{equation}
E \displaystyle\sup_D \displaystyle\sum_l\Biggl|\biggl({\bf Z}_j^1(m_0)(x_l)-{\bf Z}^1(m_0)(x_l)\biggr)-\biggl({\bf Z}_j^1(m_0)(x_{l-1})-{\bf Z}^1(m_0)(x_{l-1})\biggr)\Biggr|^2
\end{equation}
exists and bounded uniformly in $j$. Thus, by Fatou's Lemma, we have
\begin{eqnarray*}
&&
\quad\displaystyle\limsup_{j\to\infty} E \displaystyle\sup_D \displaystyle\sum_l\Biggl|\biggl({\bf Z}_{j}^1(m_0)(x_l)-{\bf Z}^1(m_0)(x_l)\biggr)-\biggl({\bf Z}_{j}^1(m_0)(x_{l-1})-{\bf Z}^1(m_0)(x_{l-1})\biggr)\Biggr|^2\nn\\
&&
\leq E \displaystyle\lim_{j\to\infty}\displaystyle\sup_{D_{[a,b]}} \displaystyle\sum_l\Biggl|\biggl({\bf Z}_{j}^1(m_0)(x_l)-{\bf Z}^1(m_0)(x_l)\biggr)-\biggl({\bf Z}_{j}^1(m_0)(x_{l-1})-{\bf Z}^1(m_0)(x_{l-1})\biggr)\Biggr|^2\nn\\
&&
 =E \displaystyle\sup_{D_{[a,b]}} \displaystyle\sum_l\displaystyle\lim_{j\to\infty}\Biggl|\biggl({\bf Z}_{j}^1(m_0)(x_l)-{\bf Z}^1(m_0)(x_l)\biggr)-\biggl({\bf Z}_{j}^1(m_0)(x_{l-1})-{\bf Z}^1(m_0)(x_{l-1})\biggr)\Biggr|^2\nn\\
&&
=0.
\end{eqnarray*}


The exchange of $\displaystyle\lim_{j\to\infty}$ and $\displaystyle\sup_D$ is due to the fact that 
\begin{equation*}
\displaystyle\lim_{j\to\infty}\displaystyle\sum_l\Biggl|\biggl({\bf Z}_{j}^1(m_0)(x_l)-{\bf Z}^1(m_0)(x_l)\biggr)-\biggl({\bf Z}_{j}^1(m_0)(x_{l-1})-{\bf Z}^1(m_0)(x_{l-1})\biggr)\Biggr|^2=0
\end{equation*}
uniformly with the partition $D_{[a,b]}$. Thus, we revisit (\ref{eq: A1}) and apply $m=m_0$  to conclude there exists $J_0$ such that when $j\geq J_0$
\begin{equation*}
d_{2,\theta}\biggl({\bf Z}_{j}^1,\bf Z^1\biggr)<\epsilon.
\end{equation*}

Thus, 
\begin{equation*}\displaystyle\limsup_{j\to\infty} E \displaystyle\sup_D \displaystyle\sum_l\Biggl|\biggl({\bf Z}_{j}^1(m_0)(x_l)-{\bf Z}^1(m_0)(x_l)\biggr)-\biggl({\bf Z}_{j}^1(m_0)(x_{l-1})-{\bf Z}^1(m_0)(x_{l-1})\biggr)\Biggr|^2=0.\end{equation*}
For the convergence of  ${\bf Z}_j^2\to{\bf Z}^2$, similarly, we have
\begin{equation}\label{secondlevelinequality}
d_{2,\theta}({\bf Z}_j^2, {\bf Z}^2)\leq d_{2,\theta}({\bf Z}_j^2, {\bf Z}_j^2(m))+d_{2,\theta}({\bf Z}_j^2(m), {\bf Z}^2(m))+d_{2,\theta}({\bf Z}^2(m), {\bf Z}^2).
\end{equation}

The convergence of the last term of (\ref{secondlevelinequality}) as $m\to\infty$ is clear from Theorem 3.8.\,\,From the proofs of Proposition 3.6, 3.7 and Theorem 3.8, one can show the convergence of the first term of (\ref{secondlevelinequality}) uniformly in $j$ as $m\to\infty$.\,\,That is to say, for any given $\epsilon>0$, one can find a N such that for $m\geq N$,  $d_{2,\theta}({\bf Z}_j^2, {\bf Z}_j^2(m))<\frac{\epsilon}{3}$ for all $j$, and $d_{2,\theta}({\bf Z}^2(m), {\bf Z}^2)<\frac{\epsilon}{3}$.\,\,In particular, the above inequality also holds if we replace $m$ by N. For a fixed partition of $[x', x'']$ and this N, one can show by the same method as in the proof of $d_{2,\theta}({\bf Z}_{j}^1(m_0), {\bf Z}^1(m_0))$ as $j\to\infty$ that $d_{2,\theta}({\bf Z}_j^2(N), {\bf Z}^2(N))<\frac{\epsilon}{3}$ by the bounded variation property of the smooth rough path.\,This can be seen as ${\bf Z}_j^2(N)$ and ${\bf Z}^2(N)$ are just tensor product of bounded variation paths
${\bf Z}_j(N)$ and ${\bf Z}(N)$.\,\,Thus ${\bf Z}_j^2(N)(x)$ also converge to ${\bf Z}^2(N)(x)$ uniformly in $x$.\,By using a similar method as in the proof $d_{2,\theta}({\bf Z}_{j}^1(m), {\bf Z}^1(m))\to 0$ as $j\to\infty$, we can prove that $d_{2,\theta}({\bf Z}_{j}^2(N), {\bf Z}^2(N))\to 0$ as $j\to\infty$ so there exists an integer $J>0$ such that $j\geq J$, $d_{2,\theta}({\bf Z}_{j}^2(N), {\bf Z}^2(N))<\frac{\epsilon}{3}.$
Hence, for $j\geq J$, it follows from (\ref{secondlevelinequality}) for $m=N$ that $d_{2,\theta}({\bf Z}_j^2, {\bf Z}^2)\leq\epsilon$. The first claim is asserted. By the definition of $\int_{-\infty}^{\infty} g_j(x)dL_t^x$, one can conclude the second claim.
\end{proof}
Proposition \ref{continuitysecondlevel} is also true for $g$ being of bounded $q$-variation $(\frac{2}{3-\alpha}\leq q<3$) but not being continuous. For the discontinuous case, we use the method from \cite{Williams2001} by adding a fictitious space interval during which linear segments remove the discontinuity, also bear in mind that a function with bounded q-variation has at most countable jumps. 
\begin{definition}\label{deftau}
Let $g(x)$ is c\`adl\`ag in x of finite $q$-variation and set $G(x):= (g(x), L_t(x))$. Let $\delta>0$, for each $n\geq 1$, let $x_n$ be the point of the n-th largest jump of g. Define a map
\[
\tau_{\delta}: [x', x'']\to [x', x''+\delta\sum_{n=1}^{\infty}|h(x_n)|^q]\]
in the following way
\[
\tau_{\delta}(x) =  x+\delta\sum_{n=1}^{\infty}|h(x_n)|^q 1_{x_n\leq x} (x), \]
where $h(x_n) := G(x_n) - G(x_n-).$
 
The map $\tau_{\delta}: [x', x'']\to [x', \tau_{\delta}(x'')]$ extends the space interval into one where we define the continuous path $G_{\delta}(y)$ from a  c\`adl\`ag path G by
\[
G_{\delta}(y)=
\begin{cases}\label{cases}
G(x), &\text{if $y=\tau_{\delta}(x)$;}\\
G(x_n-)+(y-\tau_{\delta}(x_n-))h(x_n)\delta^{-1}|h(x_n)|^{-q}, &\text{if $y\in[\tau_{\delta}(x_n-),\tau_{\delta}(x_n))$.}\\
\end{cases}
\]
\end{definition}
Notice that $L_{t,\delta}(y):=L_{t,\delta}(\tau_{\delta}(x)) = L_t^x$ as $L_t^x$ is continuous.
Let $g(x)$ be a  c\`adl\`ag path with bounded $q$-variation $(\frac{2}{3-\alpha}\leq q<3)$, we define 
\begin{equation}\label{final}
\int_{x'}^{x''} L_t^xdg(x) = \int_{x'}^{x''} L_t^xdg^c(x)+\sum_r L_t^{x_r}(h(x_r)-h({x_r}-)),
\end{equation}
where the discontinuous $g$ is decomposed into its continuous part $g^c$ and its jump part $h$.
\begin{theorem}\label{continuousthmsecondlevel}
Let g(x) be a  c\`adl\`ag path with bounded $q$-variation $(\frac{2}{3-\alpha}\leq q<3)$.\,\,Then
\begin{equation}\label{final}
\int_{x'}^{x''} L_t^xdg(x) = \int_{x'}^{\tau_{\delta}(x'')} L_{t,\delta}(y)dg_{\delta}(y).
\end{equation}
\end{theorem}
\begin{proof} The right hand side of (\ref{final}) is a rough path as defined in the previous section. As local time is continuous, hence the integral $\int_{x'}^{x''}L_t^xdg^c(x)$ is a rough path can be defined as in the previous section.\,\,For the integral associated with the jump part, we need the method pointed out before the theorem.\,\,At each discontinuous point $x_r$,
\[
\int_{x_r-}^{x_r}L_t^xdg(x) = L_t (x_r)(g(x_r) - g(x_r-)) = L_t (x_r)(h(x_r) - h(x_r-)).\]
By Definition \ref{deftau}, we have that
\[
L_t (x_r)(g(x_r) - g(x_r-)) = L_{t,\delta}(\tau_{\delta}(x_r-))(g_{\delta}(\tau_{\delta}(x_r))-g_{\delta}(\tau_{\delta}(x_r-))).\]
Hence, it follows that
\begin{equation}\label{useforonce}
\sum_{r} L_t (x_r)(g(x_r) - g(x_r-)) = \sum_{r} L_{t,\delta}(\tau_{\delta}(x_r-))(g_{\delta}(\tau_{\delta}(x_r))-g_{\delta}(\tau_{\delta}(x_r-))).
\end{equation}
From Corollary 3.9, we know that the right hand side of (\ref{useforonce}) alone is not well defined, but together with $\sum_{r}(({\bf Z})^2_{1,2})$ it is well defined. In this case, we need to check that
\[
\sum_{r}(({\bf Z}_{\delta})_{\tau_{\delta}(x_r-),\tau_{\delta}(x_r)}^2)_{1,2} = 0\]
in order to have (\ref{useforonce}) to be well defined. From Corollary \ref{roughpathintegraldef} and the continuity of local time, we obtain that
\[
\sum_{r}(({\bf Z}_{\delta})_{\tau_{\delta}(x_r-),\tau_{\delta}(x_r)}^2)_{1,2} = \sum_{r}\int_{\tau_{\delta}(x_r-)}^{\tau_{\delta}(x_r)}(L_{t,\delta}(y)-L_{t,\delta}(\tau_{\delta}(x_r-)))dg_{\delta}(y) = 0,\]
where $Z_{\delta}(y): =(L_{t,\delta}(y), g_{\delta}(y)).$
Therefore, we have
\[
\sum_{r}\int_{x_r-}^{x_r} L_t^xdh(x) = \sum_{r}\int_{\tau_{\delta}(x_r-)}^{\tau_{\delta}(x_r)}L_t^{\delta}(y)dg_{\delta}(y)=\sum_{r}\int_{\tau_{\delta}(x_r-)}^{\tau_{\delta}(x_r)}L_t^{\delta}(y)dh_{\delta}(y).\]

As the continuous part $g^c$ is the same as $g_{\delta}$ on the space interval where $g$ is continuous and $g^c$ does not contribute on the interval where the function jump. Simimlary, where $g$ is jump discontinuous, we denote as $h$, does not contribute on the interval where the function is continuous, hence
\begin{align*}
\int_{x'}^{x''}L_t^xdg(x)&=\int_{x'}^{x''}L_t^xdg^c(x) + \int_{x'}^{x''}L_t^xdh(x)\\
&= \int_{x'}^{\tau_{\delta}(x'')}L_{t,\delta}(y)dg_{\delta}^c(y) + \int_{x'}^{\tau_{\delta}(x'')}L_{t,\delta}(y)dh_{\delta}(y)\\
&=\int_{x'}^{\tau_{\delta}(x'')}L_{t,\delta}(y)dg_{\delta}(y).
\end{align*}
This completes the proof.
\end{proof}
The convergence for discontinuous functions can be proved by applying the method in the above theorem and Proposition \ref{continuitysecondlevel} to $(h-g)(x)$.\,Note the function $h_{\delta}$ is piecewise linear for fixed $\delta>0$.\,It is certainly of bounded q-variation with a control function ${\bf w}_\delta.$ If $h_{\delta j}$ is a sequence of bounded q-variation functions with control function ${\bf w}_{\delta j}$ such that  $h_{\delta j}\to h_{\delta}$ and ${\bf w}_{\delta j}\to{\bf w}_{\delta}$ as $j\to\infty$ uniformly. Then, we have 
\begin{equation*}
\int_{x'}^{\tau_{\delta}(x'')}L_{t,\delta}(y)dh_{\delta j}(y)\to \int_{x'}^{\tau_{\delta}(x'')}L_{t,\delta}(y)dh_{\delta}(y)
\end{equation*}
as $j\to\infty$.
Hence, we have the following proposition.
\begin{proposition}\label{propadded}
Let $\alpha\in(\frac{5}{3},2)$, $\frac{2}{3-\alpha}\leq q<3$, $\theta$ be chosen such that $\frac{4}{2h+\alpha-1}<\theta<3$.
\,Consider h, the jump part of the function g and assume there is a sequence of continuous functions $h_j\to h$ as $j\to\infty$. 
Let $h_{j\delta}$ and $h_\delta$ be defined in the same way as $G_\delta(y)$.\,\,Let $h_{\delta j}$ be a sequence of continuous function satisfying $h_{\delta j}\to h_\delta$ as $j\to\infty$ together with their control functions and 
\begin{equation}\label{three star}
 \int_{-\infty}^{\infty}L_t^\delta(y) dh_{j\delta}(y)=\int_{-\infty}^{\infty}L_t^\delta(y) dh_{\delta j}(y).
\end{equation}
Then as $j\to\infty$
\begin{equation*}
 \int_{-\infty}^{\infty}L_t^x dh_{j}(x)\to\int_{-\infty}^{\infty}L_t^x dh(x).
\end{equation*}
\end{proposition}
 
\begin{proof}
By Theorem \ref{continuousthmsecondlevel}, integration by parts formula and assumption (3.47), one can see that 
\begin{eqnarray*}\nn
&&
\int_{-\infty}^{\infty} h_j(x)dL_t^x=\int_{-\infty}^{\infty} h_{j\delta}(y)dL_t^\delta(y)\nn\\
&&
\,\,\,\,\,\,\quad\,\qquad\qquad\,\,= -\int_{-\infty}^{\infty} L_t^\delta(y)dh_{j\delta}(y)\nn\\
&&
\,\,\,\,\,\,\,\quad\qquad\qquad\,\,= -\int_{-\infty}^{\infty} L_t^\delta(y)dh_{\delta j}(y).\nn\\
\end{eqnarray*}
By the assumption that $h_{\delta j}\to h_\delta$ as $j\to\infty$ together with their control functions, then by Proposition \ref{continuitysecondlevel} we have
$\int_{-\infty}^{\infty}L_t^\delta(y) dh_{\delta j}(y)\to\int_{-\infty}^{\infty}L_t^\delta(y) dh_{\delta}(y).$ By using Theorem 3.12,\,we have that $\int_{-\infty}^{\infty}L_t^\delta(y) dh_{\delta}(y)=\int_{-\infty}^{\infty}L_t^x dh(x).$ Hence, the result of the proposition follows.
\end{proof}

\begin{corollary}
Let $\alpha\in(\frac{5}{3},2)$, $\frac{2}{3-\alpha}\leq q<3$, $\theta$ be chosen such that $\frac{4}{2h+\alpha-1}<\theta<3$.
\,Moreover, let $Z_j(x): = (L_t^x, g_j(x)), Z(x):=(L_t^x, g(x))$, where $g_j(.), g(.)$ are both of bounded q-variation, and $g_j$ is continuous and $g$ is c\`adl\`ag with decomposition $g=g_c+h$, where $g_c$ is the continuous part of $g$ and $h$ is the jump part of $g$. Suppose $g_j=g_{cj}+h_j$ with control function ${\bf w}_{cj}$ and ${\bf w}_{hj}$ such that $g_{cj}\to g_c$ and  ${\bf w}_{cj}\to{\bf w}_{c}$ uniformly, $h_j$ satisfying conditions in Proposition \ref{propadded}.\,\,Then we have 
\begin{equation*}
\int_{-\infty}^{\infty} g_j(x)dL_t^x\to\int_{-\infty}^{\infty} g(x)dL_t^x \quad a.s.\quad as\quad j\to\infty.
\end{equation*}
\end{corollary}
\begin{proof}
The Corollary follows from Theorem \ref{continuousthmsecondlevel} and Proposition \ref{propadded}.
\end{proof}

\subsection{Third level rough path}
The $\theta$-variation formula for third level path in \cite{Lyons2002} is given as 

\begin{eqnarray}\nn
&&
\,\quad\displaystyle\sup_{D}\bigl|{\bf Z}(m+1)_{x_{l-1},x_l}^3-{\bf Z}(m)_{x_{l-1},x_l}^3\bigr|^{\frac{\theta}{3}}\nn\\
&&
\leq C_2\sum_{n=1}^{\infty}n^\gamma\sum_{k=1}^{2^n}\bigl|{\bf Z}(m+1)_{x_{k-1}^n,x_k^n}^3-{\bf Z}(m)_{x_{k-1}^n,x_k^n}^3\bigr|^{\frac{\theta}{3}}\nn\\
&&
\quad + C_3\biggl(\sum_{n=1}^{\infty}n^\gamma\sum_{k=1}^{2^n}\bigl|{\bf Z}(m+1)_{x_{k-1}^n,x_k^n}^1-{\bf Z}(m)_{x_{k-1}^n,x_k^n}^1\bigr|^{{\theta}}\biggr)^{\frac{1}{3}}\nn\\
&&
\quad\quad\quad \times\biggl(\sum_{n=1}^{\infty}n^\gamma\sum_{k=1}^{2^n}\bigl|{\bf Z}(m+1)_{x_{k-1}^n,x_k^n}^2\bigr|^{\frac{\theta}{2}}+\bigl|{\bf Z}(m)_{x_{k-1}^n,x_k^n}^2\bigr|^{\frac{\theta}{2}}\biggr)^{\frac{2}{3}}\nn\\
&&
\quad+ C_4\biggl(\sum_{n=1}^{\infty}n^\gamma\sum_{k=1}^{2^n}\bigl|{\bf Z}(m+1)_{x_{k-1}^n,x_k^n}^2-{\bf Z}(m)_{x_{k-1}^n,x_k^n}^2\bigr|^{\frac{\theta}{2}}\biggr)^{\frac{2}{3}}\nn\\
&&
\quad\quad\quad \times\biggl(\sum_{n=1}^{\infty}n^\gamma\sum_{k=1}^{2^n}\bigl|{\bf Z}(m+1)_{x_{k-1}^n,x_k^n}^1\bigr|^{{\theta}}+\bigl|{\bf Z}(m)_{x_{k-1}^n,x_k^n}^1\bigr|^{{\theta}}\biggr)^{\frac{1}{3}}\nn\\
&&
\quad +C_5\biggl(\sum_{n=1}^{\infty}n^\gamma\sum_{k=1}^{2^n}\bigl|{\bf Z}(m+1)_{x_{k-1}^n,x_k^n}^1-{\bf Z}(m)_{x_{k-1}^n,x_k^n}^1\bigr|^{\theta}\biggr)^{\frac{1}{3}}\nn\\
&&
\quad\quad\quad \times\biggl(\sum_{n=1}^{\infty}n^\gamma\sum_{k=1}^{2^n}\bigl|{\bf Z}(m+1)_{x_{k-1}^n,x_k^n}^1\bigr|^{{\theta}}+\bigl|{\bf Z}(m)_{x_{k-1}^n,x_k^n}^1\bigr|^{{\theta}}\biggr)^{\frac{2}{3}},
\end{eqnarray}
where $\{x_k^n\}$ satisfing (\ref{usedlater}).

We have obtained estimations for the first and second level paths.\,As there is a connection between the sample path of local time of symmetric stable processes and its associated Gaussian processes by Dynkin isomorphism theorem (cf.\,\cite{Marcusr92}), therefore, we present some relevant results for the Gaussian processes first.\,The importance of the following result regarding Gaussian random variables will be made clear throughout the estimation of the $\theta$-variation on the third level path.
 
Again, we have
\begin{equation}
|g(b)-g(a)|^2\leq{\textbf{w}}_1(a,b)^{2h}
\end{equation}
for any $(a,b)\in\Delta$ and $h\theta>1.$
 
We consider the cross product term on the increments  of the $g$ with parameter h. By a purely algebraic procedure as in (\ref{eq: 4.9}), for $b>a$, $\varrho>0,$ we have
\begin{eqnarray}\label{consumed}&&
\biggl(g(b)-g(a)\biggr)\biggl(g(b+\varrho)-g(a+\varrho)\biggr)\leq \frac{1}{2}\biggl(\textbf{w}_1(b+\varrho, a)^{2h} -  {\textbf{w}_1(b+\varrho, b)}^{2h}\nn\\
&&
\qquad\qquad\qquad\qquad\qquad\qquad\qquad\quad\qquad\qquad+ \textbf{w}_1(a+\varrho, b)^{2h} - \textbf{w}_1(a+\varrho, a)^{2h}\biggr).
\end{eqnarray}
 
All the estimations of the variance or covariance of the local time one encounter later in this paper are similar to one of the following formats.
For $n>m$, the proof on the convergence of the third level path is a straightforward exercise as pointed out in the proof of the Proposition 4.5.1 in \cite{Lyons2002}.
\,\,For $m\geq n,\, k=1,2,\dots,2^n,\, 2^{m-n}(k-1)+1\leq r<l\leq 2^{m-n}k$, the estimations of the variance of local time on different intervals are given by
\begin{equation}
E\bigl(L_t^{x_{2l-2}^{m+1}} - L_t^{x_{k-1}^{n}}\bigr)^2
\leq
c \bigl(\frac{1}{2^n}\bigr)^{\alpha-1}
\end{equation}
\qquad\qquad\qquad\qquad\qquad\qquad\qquad\qquad\qquad\qquad\qquad\qquad  and
\begin{equation}
E\bigl(L_t^{x_{2l-1}^{m+1}} - L_t^{x_{2l-2}^{m+1}}\bigr)^2 \leq c \bigl(\frac{2l-1}{2^{m+1}}-\frac{2l-2}{2^{m+1}}\bigr)^{\alpha-1}\leq c \bigl(\frac{1}{2^m}\bigr)^{\alpha-1}
\end{equation}
where c is a generic constant. The covariance of the local time on non-overlapping intervals $[x_{k-1}^n, x_{2l-2}^{m+1}]$, $[x_{2l-2}^{m+1}, x_{2l-1}^{m+1}]$ satisfies
\begin{eqnarray}\nn\label{another ineq}
&&
\quad E\bigl(L_t^{x_{2l-2}^{m+1}} - L_t^{x_{k-1}^{n}}\bigr)\bigl(L_t^{x_{2l-1}^{m+1}}-L_t^{x_{2l-2}^{m+1}}\bigr)\nn\\
&&
=\frac{1}{2}\biggl[-\sigma^2(x_{2l-2}^{m+1} - x_{k-1}^{n}) + \sigma^2(x_{2l-1}^{m+1} - x_{k-1}^{n}) - \sigma^2(x_{2l-1}^{m+1} - x_{2l-2}^{m+1}) \biggr]\nn\\
&&
\leq 0.
\end{eqnarray}
Here we have used (\ref {eq: 4.9}), the concavity of $\sigma^2$ and the following two observations for a non-negative concave function $f$:
\begin{align*}
f(tx)&\geq tf(x)\,\,\quad\text{for $t\in [0,1]$},\\
f(a)+f(b)&\geq f(a+b).
\end{align*}

On the other hand, by the increasing property of $\sigma^2$ and the first part of (\ref{another ineq})
\begin{equation}
E\bigl(L_t^{x_{2l-2}^{m+1}} - L_t^{x_{k-1}^{n}}\bigr)\bigl(L_t^{x_{2l-1}^{m+1}}-L_t^{x_{2l-2}^{m+1}}\bigr)\geq- \frac{1}{2}\sigma^2(x_{2l-1}^{m+1} - x_{2l-2}^{m+1})\geq - c \bigl(\frac{1}{2^{m+1}}\bigr)^{\alpha-1},
\end{equation}
where c is a generic constant.\,\,The covariance of local time on two overlapping intervals $[x_{k-1}^{n}, x_{2r-2}^{m+1}]$, $[x_{k-1}^{n}, x_{2l-2}^{m+1}]$ is given by (without loss of generality, assuming $l>r$)
\begin{eqnarray}\nn
&&
\quad E(L_t^{x_{2l-2}^{m+1}} - L_t^{x_{k-1}^{n}})(L_t^{x_{2r-2}^{m+1}} - L_t^{x_{k-1}^{n}})\nn\\
&&
=E(L_t^{x_{2l-2}^{m+1}} -L_t^{x_{2r-2}^{m+1}})(L_t^{x_{2r-2}^{m+1}} - L_t^{x_{k-1}^{n}}) + E(L_t^{x_{2r-2}^{m+1}} - L_t^{x_{k-1}^{n}})(L_t^{x_{2r-2}^{m+1}} - L_t^{x_{k-1}^{n}})\nn\\
&&
=\frac{1}{2}\bigl[-\sigma^2(x_{2r-2}^{m+1}-x_{k-1}^{n}) +\sigma^2(x_{2l-2}^{m+1}-x_{k-1}^{n})\nn\\
&&
\quad\quad\qquad\quad-\sigma^2(x_{2l-2}^{m+1}-x_{2r-2}^{m+1})\bigr]+\sigma^2(x_{2r-2}^{m+1}-x_{k-1}^{n})\nn\\
&&
= \frac{1}{2}\bigl[\sigma^2(x_{2l-2}^{m+1}-x_{k-1}^{n})-\sigma^2(x_{2l-2}^{m+1}-x_{2r-2}^{m+1})+\sigma^2(x_{2r-2}^{m+1}-x_{k-1}^{n})\bigr].
\end{eqnarray}

One can conclude the above is non-negative based on the non-negativity and monotonically increasing property of $\sigma^2$. Moreover, it is bounded from above as
\begin{equation}
E(L_t^{x_{2l-2}^{m+1}} - L_t^{x_{k-1}^{n}})(L_t^{x_{2r-2}^{m+1}} - L_t^{x_{k-1}^{n}})
\leq 
\sigma^2(x_{2r-2}^{m+1}-x_{k-1}^{n})\leq c\bigl(\frac{1}{2^n}\bigr)^{\alpha-1}.
\end{equation}
Similarly, one can estimate the following term
\begin{eqnarray}
&&
\quad E(L_t^{x_{2l-2}^{m+1}} - L_t^{x_{k-1}^{n}})(L_t^{x_{2r-1}^{m+1}} - L_t^{x_{2r-2}^{m+1}} )\nn\\
&&
= E(L_t^{x_{2l-2}^{m+1}} -L_t^{x_{2r-1}^{m+1}})(L_t^{x_{2r-1}^{m+1}} - L_t^{x_{2r-2}^{m+1}})+E(L_t^{x_{2r-1}^{m+1}} - L_t^{x_{2r-2}^{m+1}})^2\nn\\
&&
\quad\quad+ E( L_t^{x_{2r-2}^{m+1}}- L_t^{x_{k-1}^{n}})(L_t^{x_{2r-1}^{m+1}} - L_t^{x_{2r-2}^{m+1}})\nn\\
&&
=\frac{1}{2}\bigl[-\sigma^2(x_{2r-1}^{m+1}-x_{2r-2}^{m+1})+\sigma^2(x_{2l-2}^{m+1}-x_{2r-2}^{m+1})-\sigma^2(x_{2l-2}^{m+1}-x_{2r-1}^{m+1})\nn\\
&&
\quad\quad\quad\quad+2\sigma^2(x_{2r-1}^{m+1}-x_{2r-2}^{m+1})-\sigma^2(x_{2r-2}^{m+1}-x_{k-1}^{n})\nn\\
&&
\quad\quad\quad\quad+\sigma^2(x_{2r-1}^{m+1}-x_{k-1}^{n})-\sigma^2(x_{2r-1}^{m+1}-x_{2r-2}^{m+1})\bigr]\nn\\
&&
=\frac{1}{2}\bigl[ \sigma^2(x_{2l-2}^{m+1}-x_{2r-2}^{m+1})-\sigma^2(x_{2l-2}^{m+1}-x_{2r-1}^{m+1})-\sigma^2(x_{2r-2}^{m+1}-x_{k-1}^{n})\nn\\
&&
\quad\quad\quad\quad+\sigma^2(x_{2r-1}^{m+1}-x_{k-1}^{n}) \bigr].
\end{eqnarray}

The above quantity is nonnegative by monotonically increasing property of $\sigma^2.$ Moreover, it can be shown that it is bounded from above 
\begin{eqnarray}\nn
&&
\quad E(L_t^{x_{2l-2}^{m+1}} - L_t^{x_{k-1}^{n}})(L_t^{x_{2r-1}^{m+1}} - L_t^{x_{2r-2}^{m+1}} )\nn\\
&&
=\frac{1}{2}\bigl[ \sigma^2(x_{2l-2}^{m+1}-x_{2r-2}^{m+1})-\sigma^2(x_{2l-2}^{m+1}-x_{2r-1}^{m+1})-\sigma^2(x_{2r-2}^{m+1}-x_{k-1}^{n})\nn\\
&&
\quad\quad\quad\quad+\sigma^2(x_{2r-1}^{m+1}-x_{k-1}^{n}) \bigr]\nn\\
&&
\leq
\sigma^2(x_{2r-1}^{m+1}-x_{2r-2}^{m+1})\nn\\
&&
\leq c\bigl(\frac{1}{2^m}\bigr)^{\alpha-1},
\end{eqnarray}
where c is a generic constant.  
 
For the case when $n> m$, we refer to (\ref {eq: 4.6}). Similar to Proposition \ref{SecondLevelVariationProposition}, we can prove the following proposition.
\begin{proposition}\label{ThirdLevelVariationProposition}
For a continuous path $Z_x$ which satisfies (\ref{eq: 4.2}) with $h\theta>1$, then for $n> m$
\begin{equation}\label{eq: 4.11}
\sum_{k=1}^{2^n}E\biggl|{\bf{Z}}(m+1)_{x_{k-1}^n, x_k^n}^3-{\bf{Z}}(m)_{x_{k-1}^n, x_k^n}^3\biggr|^{\frac{\theta}{3}}\leq C\biggl(\frac{1}{2^{n+m}}\biggr)^{\frac{h\theta-1}{2}},
\end{equation}
\end{proposition}
\noindent where $C$ is a generic constant depends on $\theta, h,$ ${\textbf{w}}_1(x',x'')$ and $c$ in (\ref{eq: 4.2}).
\begin{proof} If $n> m$, by (\ref{eq: 4.6}), we show that
\begin{eqnarray*}\nn
&&
\,\quad\sum_{k=1}^{2^n}E\biggl|{\bf{Z}}(m+1)_{x_{k-1}^n, x_k^n}^3-{\bf{Z}}(m)_{x_{k-1}^n, x_k^n}^3\biggr|^{\frac{\theta}{3}}\nn\\
&&
=\sum_{l=1}^{2^{m+1}}\displaystyle\sum_{x_{l-1}^{m+1}\leq x_{k-1}^n<x_l^{m+1}}E\biggl|\frac{1}{3!}2^{3(m+1-n)}\bigl(\Delta_l^{m+1}Z\bigr)^{\otimes 3}-\frac{1}{3!}2^{3(m-n)}\bigl(\Delta_l^mZ\bigr)^{\otimes 3}\biggr|^{\frac{\theta}{3}}\nn\\
&&
=\sum_{l=1}^{2^{m+1}}2^{n-m-1}E\biggl|\frac{1}{3!}2^{3(m+1-n)}\bigl(\Delta_l^{m+1}Z\bigr)^{\otimes 3}-\frac{1}{3!}2^{3(m-n)}\bigl(\Delta_l^mZ\bigr)^{\otimes 3}\biggr|^{\frac{\theta}{3}}\nn\\
&&
\leq C\biggl(\frac{2^m}{2^n}\biggr)^{\theta} \sum_{l=1}^{2^{m+1}}2^{n-m-1}\biggl(\frac{1}{2^m}\biggr)^{h\theta}\textbf{w}_1(x',x'')^{h\theta}\nn\\
\end{eqnarray*}
\begin{eqnarray}\nn
&&
\leq C\biggl(\frac{2^m}{2^n}\biggr)^{\theta-h\theta}  \biggl(\frac{1}{2^n}\biggr)^{h\theta-1}\nn\\
&&
\leq C\biggl(\frac{2^m}{2^n}\biggr)^{\theta-h\theta}  \biggl(\frac{1}{2^n}\biggr)^{\frac{h\theta-1}{2}}\biggl(\frac{1}{2^m}\biggr)^{\frac{h\theta-1}{2}}\nn\\
&&
\leq C\biggl(\frac{1}{2^{m+n}}\biggr)^{\frac{h\theta-1}{2}},
\end{eqnarray}
where $C$ is a generic constant depends on $\theta, h, \textbf{w}_1(x',x'')$ and $c$ in (\ref{eq: 4.2}).
\end{proof}
However, to estimate the left-hand side of (\ref {eq: 4.11}) is more complicated when $m\geq n$. Recall the following formula in \cite{Lyons2002}
\begin{eqnarray}\nn\label{eq: 4.12}
&&
\quad {\bf Z}(m+1)_{x_{k-1}^n,x_k^n}^3-{\bf Z}(m)_{x_{k-1}^n,x_k^n}^3\nn\\
&&
= \frac{1}{2}\displaystyle\sum_{l}\bigl(Z_{x_{2l-2}^{m+1}} - Z_{x_{k-1}^{n}}\bigr)\otimes\bigl(\Delta_{2l-1}^{m+1}Z\otimes\Delta_{2l}^{m+1}Z-\Delta_{2l}^{m+1}Z\otimes\Delta_{2l-1}^{m+1}Z\bigr)\nn\\
&&
\quad+\frac{1}{2}\displaystyle\sum_{l}\bigl(\Delta_{2l-1}^{m+1}Z\otimes\Delta_{2l}^{m+1}Z-\Delta_{2l}^{m+1}Z\otimes\Delta_{2l-1}^{m+1}Z\bigr)\otimes\bigl( Z_{x_{k}^{n}}-Z_{x_{2l+2}^{m+1}}\bigr)\nn\\
&&
\quad+\frac{1}{3}\displaystyle\sum_{l}\Delta_{2l-1}^{m+1}Z\bigl(\Delta_{2l}^{m+1}Z\otimes\Delta_{2l}^{m+1}Z+\Delta_{2l-1}^{m+1}Z\otimes\Delta_{2l}^{m+1}Z\bigr)\nn\\
&&
\quad-\frac{1}{6}\displaystyle\sum_{l}\Delta_{2l}^{m+1}Z\bigl(\Delta_{2l}^{m+1}Z\otimes\Delta_{2l-1}^{m+1}Z+\Delta_{2l-1}^{m+1}Z\otimes\Delta_{2l}^{m+1}Z\bigr)\nn\\
&&
\quad-\frac{1}{6}\displaystyle\sum_{l}\bigl(\Delta_{2l}^{m+1}Z\otimes\Delta_{2l-1}^{m+1}Z+\Delta_{2l-1}^{m+1}Z\otimes\Delta_{2l}^{m+1}Z\bigr)\otimes\Delta_{2l-1}^{m+1}Z\nn\\
&&
: =A_1+A_2+A_3+A_4+A_5,
\end{eqnarray}
where the sum runs over $2^{m-n}(k-1)+1\leq l<2^{m-n}k.$
As suggested by Jensen's inequality, we have
\begin{equation}\label{proposition1}
E\bigl|{\bf{Z}}(m+1)_{x_{k-1}^n, x_k^n}^3-{\bf{Z}}(m)_{x_{k-1}^n, x_k^n}^3\bigr|^{\frac{\theta}{3}}\leq \biggl(E\bigl|{\bf{Z}}(m+1)_{x_{k-1}^n, x_k^n}^3-{\bf{Z}}(m)_{x_{k-1}^n, x_k^n}^3\bigr|^2\biggr)^{\frac{\theta}{6}}.
\end{equation}
In order to estimate (\ref{proposition1}), we first use (\ref {eq: 4.12}) to estimate
\begin{equation*}
E\bigl|{\bf{Z}}(m+1)_{x_{k-1}^n, x_k^n}^3-{\bf{Z}}(m)_{x_{k-1}^n, x_k^n}^3\bigr|^2.
\end{equation*}
We will only estimate the term $A_1$ and $A_3$, as other terms can be estimated similarly.  For $m\geq n$, we first estimate the term $A_1^2$. Define $\varphi_l =Z_{x_{2l-2}^{m+1}} - Z_{x_{k-1}^{n}}$ and $\varphi_l^i =  Z_{x_{2l-2}^{m+1}}^i - Z_{x_{k-1}^{n}}^i$, then
\begin{eqnarray*}
&&
A_1^2=\biggl|\displaystyle\sum_{l}\bigl(Z_{x_{2l-2}^{m+1}} - Z_{x_{k-1}^{n}}\bigr)\otimes\bigl(\Delta_{2l-1}^{m+1}Z\otimes\Delta_{2l}^{m+1}Z-\Delta_{2l}^{m+1}Z\otimes\Delta_{2l-1}^{m+1}Z\bigr)\biggr|^2\nn\\
&&
\quad\,\,\,=\displaystyle\sum_{i,j,u=1,i\neq j,u}^{2}\biggl(\displaystyle\sum_{l}\varphi_l^u\bigl(\Delta_{2l-1}^{m+1}Z^i\Delta_{2l}^{m+1}Z^j-\Delta_{2l}^{m+1}Z^i\Delta_{2l-1}^{m+1}Z^j\bigr)\biggr)^2\nn\\
&&
\quad\,\,\,=\displaystyle\sum_{i\neq j,u}\displaystyle\sum_{l}(\varphi_l^u)^2\bigl(\Delta_{2l-1}^{m+1}Z^i\Delta_{2l}^{m+1}Z^j-\Delta_{2l}^{m+1}Z^i\Delta_{2l-1}^{m+1}Z^j\bigr)^2\nn\\
\end{eqnarray*}
\begin{eqnarray}\nn
&&
\quad\,\,\,\quad+2\displaystyle\sum_{i\neq j,u}\displaystyle\sum_{r<l}\biggl(\varphi_l^u\varphi_r^u{\Delta_{2l-1}^{m+1}}Z^i{\Delta_{2r-1}^{m+1}}Z^i{\Delta_{2r}^{m+1}}Z^j{\Delta_{2l}^{m+1}}Z^j\nn\\
&&
\quad\qquad\qquad\quad-\varphi_l^u\varphi_r^u{\Delta_{2l}^{m+1}}Z^i{\Delta_{2r-1}^{m+1}}Z^i{\Delta_{2r}^{m+1}}Z^j{\Delta_{2l-1}^{m+1}}Z^j\nn\\
&&
\quad\qquad\qquad\quad-\varphi_l^u\varphi_r^u{\Delta_{2l-1}^{m+1}}Z^i{\Delta_{2r}^{m+1}}Z^i{\Delta_{2r-1}^{m+1}}Z^j{\Delta_{2l}^{m+1}}Z^j\nn\\
&&
\quad\qquad\qquad\quad+\varphi_l^u\varphi_r^u{\Delta_{2r}^{m+1}}Z^i{\Delta_{2l}^{m+1}}Z^i{\Delta_{2l-1}^{m+1}}Z^j{\Delta_{2r-1}^{m+1}}Z^j\biggr)\nn\\
&&
\quad\,\,\,:=J_1+2(J_2-J_3-J_4+J_5).
\end{eqnarray}
$Z^i$ denotes the $i^{th}$ element of $Z_x:=(L_t^x, g(x))$ where $i=1,2.$  When $i= j$, the above equation vanishes. When $i\neq j, l\neq r$, we first consider the case when $u=i$
\begin{eqnarray}\nn
&& 
J_1 = \displaystyle\sum_{l}(\varphi_l^1)^2\bigl(\Delta_{2l-1}^{m+1}Z^1\Delta_{2l}^{m+1}Z^2-\Delta_{2l}^{m+1}Z^1\Delta_{2l-1}^{m+1}Z^2\bigr)^2\nn\\
&&
\quad\quad+ \displaystyle\sum_{l}(\varphi_l^2)^2\bigl(\Delta_{2l-1}^{m+1}Z^2\Delta_{2l}^{m+1}Z^1-\Delta_{2l}^{m+1}Z^2\Delta_{2l-1}^{m+1}Z^1\bigr)^2\nn\\
&&
\,\,\quad= \displaystyle\sum_{l}\bigl(L_t^{x_{2l-2}^{m+1}} - L_t^{x_{k-1}^{n}}\bigr)^2\biggl[\bigl(L_t^{x_{2l-1}^{m+1}}-L_t^{x_{2l-2}^{m+1}}\bigr)\bigl(g(x_{2l}^{m+1})-g(x_{2l-1}^{m+1})\bigr)\nn\\
&&
\quad\quad\quad\quad-\bigl(L_t^{x_{2l}^{m+1}}-L_t^{x_{2l-1}^{m+1}}\bigr)\bigl(g(x_{2l-1}^{m+1})-g(x_{2l-2}^{m+1})\bigr)\biggr]^2\nn\\
&&
\quad\quad+ \displaystyle\sum_{l}\bigl(g(x_{2l-2}^{m+1})-g(x_{k-1}^{n})\bigr)^2\biggl[\bigl(g(x_{2l-1}^{m+1})-g(x_{2l-2}^{m+1})\bigr)\bigl(L_t^{x_{2l}^{m+1}}-L_t^{x_{2l-1}^{m+1}}\bigr)\nn\\
&&
\quad\quad\quad\quad-\bigl(g(x_{2l}^{m+1})-g(x_{2l-1}^{m+1})\bigr)\bigl(L_t^{x_{2l-1}^{m+1}}-L_t^{x_{2l-2}^{m+1}}\bigr)\biggr]^2\nn\\
&&
\,\,\quad:=I_1+I_2.
\end{eqnarray}

First, we estimate $EI_1$.\,The estimation of $EI_2$ can be done similarly.\,Set $H_1=L_t^{x_{2l-2}^{m+1}} - L_t^{x_{k-1}^{n}}$, $H_2 = L_t^{x_{2l-1}^{m+1}}-L_t^{x_{2l-2}^{m+1}}$ and $H_3 = L_t^{x_{2l}^{m+1}}-L_t^{x_{2l-1}^{m+1}}$, then
\begin{eqnarray}\nn
&&
I_1= \displaystyle\sum_{l}\biggl[  {H_1}^2 {H_2}^2 \bigl(g(x_{2l}^{m+1})-g(x_{2l-1}^{m+1})\bigr)^2+ {H_1}^2 {H_3}^2 \bigl(g(x_{2l-1}^{m+1})-g(x_{2l-2}^{m+1})\bigr)^2\nn\\
&&
\,\quad\quad-2{H_1}^2 {H_2} H_3\bigl(g(x_{2l}^{m+1})-g(x_{2l-1}^{m+1})\bigr)\bigl(g(x_{2l-1}^{m+1})-g(x_{2l-2}^{m+1})\bigr)\biggr].
\end{eqnarray}
We estimate each term in $I_1$ in the following. The bound of the first term in $I_1$ is given by
\begin{eqnarray}\nn
&&
\quad E\displaystyle\sum_{l} {H_1}^2 {H_2}^2 \bigl(g(x_{2l}^{m+1})-g(x_{2l-1}^{m+1})\bigr)^2\nn\\
&&
\leq c\:2^{m-n} \bigl(\frac{1}{2^n}\bigr)^{\alpha-1}\bigl(\frac{1}{2^{m+1}}\bigr)^{\alpha-1}\bigl(\frac{1}{2^m}\bigr)^{2h}\nn\\
&&
\leq c\:\bigl(\frac{1}{2^n}\bigr)^{\alpha}\bigl(\frac{1}{2^m}\bigr)^{2h+\alpha-2},
\end{eqnarray}
where c is a generic constant.
The bound of the second term in $I_1$ is similar to the first term with
\begin{equation}
\displaystyle\sum_{l} E{H_1}^2 {H_3}^2 \bigl(g(x_{2l-1}^{m+1})-g(x_{2l-2}^{m+1})\bigr)^2\leq c\:\bigl(\frac{1}{2^n}\bigr)^{\alpha}\bigl(\frac{1}{2^m}\bigr)^{2h+\alpha-2}.
\end{equation}
The bound of the third term in $I_1$ is given by
\begin{eqnarray}
&&
\quad\displaystyle\sum_{l}E {H_1}^2 {H_2} H_3\bigl(g(x_{2l}^{m+1})-g(x_{2l-1}^{m+1})\bigr)\bigl(g(x_{2l-1}^{m+1})-g(x_{2l-2}^{m+1})\bigr)\nn\\
&&
\leq c \bigl(\frac{1}{2^m}\bigr)^{2h}2^{m-n} \bigl( \frac{1}{2^n}\bigr)^{\alpha-1} \bigl( \frac{1}{2^m}\bigr)^{\alpha-1}\nn\\
&&
\leq c \biggl(\frac{1}{2^n}\biggr)^{\alpha} \bigl( \frac{1}{2^m}\bigr)^{2h+\alpha-2}.
\end{eqnarray}

Hence, we have
\begin{equation}\label{eq: 4.13}
EI_1\leq c \bigl(\frac{1}{2^n}\bigr)^{\alpha}\bigl(\frac{1}{2^m}\bigr)^{2h+\alpha-2}.
\end{equation}
Now, we compute the estimation for $EI_2$ term, that is
\begin{eqnarray}\label{eq: 4.14}\nn
&&
\quad E\biggl\{\displaystyle\sum_{l}\bigl(g(x_{2l-2}^{m+1})-g(x_{k-1}^{n})\bigr)^2\biggl[\bigl(g(x_{2l-1}^{m+1})-g(x_{2l-2}^{m+1})\bigr)\bigl(L_t^{x_{2l}^{m+1}}-L_t^{x_{2l-1}^{m+1}}\bigr)\nn\\
&&
\quad\quad\qquad\quad-\bigl(g(x_{2l}^{m+1})-g(x_{2l-1}^{m+1})\bigr)\bigl(L_t^{x_{2l-1}^{m+1}}-L_t^{x_{2l-2}^{m+1}}\bigr)\biggr]^2\biggr\}\nn\\
&&
\leq c\,2^{m-n} \bigl(\frac{1}{2^{m+1}}\bigr)^{2h+\alpha-1}(\frac{1}{2^n})^{2h}\nn\\
&&
\leq c\,(\frac{1}{2^n})^{1+2h} \bigl(\frac{1}{2^{m}}\bigr)^{2h+\alpha-2}.
\end{eqnarray}
As $J_2, J_3, J_4,$ and $J_5$ have similar structure, we only need to estimate one of them. The estimation of $J_2$
\begin{eqnarray*}\nn
&&
J_2 =\displaystyle\sum_{r<l}(L_t^{x_{2l-2}^{m+1}} - L_t^{x_{k-1}^{n}})(L_t^{x_{2r-2}^{m+1}} - L_t^{x_{k-1}^{n}})(L_t^{x_{2l-1}^{m+1}} - L_t^{x_{2l-2}^{m+1}})(L_t^{x_{2r-1}^{m+1}} - L_t^{x_{2r-2}^{m+1}})\nn\\
&&
\quad\quad\quad\qquad\times\bigl(g(x_{2r}^{m+1})-g(x_{2r-1}^{m+1})\bigr)\bigl(g(x_{2l}^{m+1})-g(x_{2l-1}^{m+1})\bigr)\nn\\
&&
\,\quad\quad+\displaystyle\sum_{r<l}\bigl(g(x_{2l-2}^{m+1}) - g({x_{k-1}^{n}})\bigr)\bigl(g({x_{2r-2}^{m+1}}) -g({x_{k-1}^{n}})\bigr)\bigl(g({x_{2l-1}^{m+1}}) - g({x_{2l-2}^{m+1}})\bigr)\nn\\
&&
\quad\quad\quad\qquad\times\bigl(g({x_{2r-1}^{m+1}}) - g({x_{2r-2}^{m+1}})\bigr)\bigl(L_t^{x_{2r}^{m+1}}-L_t^{x_{2r-1}^{m+1}}\bigr)\bigl(L_t^{x_{2l}^{m+1}}-L_t^{x_{2l-1}^{m+1}}\bigr).
\end{eqnarray*}

We estimate the first term  in $J_2$ first. By (\ref{consumed}),  we have
\begin{equation}
\bigl(g(x_{2r}^{m+1})-g(x_{2r-1}^{m+1})\bigr)\bigl(g(x_{2l}^{m+1})-g(x_{2l-1}^{m+1})\bigr)
\leq \frac{{\varrho}^{2h}}{2}\biggl(|1+\frac{1}{l-r}|^{2h}+|1-\frac{1}{l-r}|^{2h}-2\biggr),
\end{equation}
which does not vanish if $h\neq\frac{1}{2}$, $l\neq r$ and $\varrho=\frac{l-r}{2^m}> 0$. Using Taylor expansion at $\frac{1}{l-r}=0$, one can show that there is a constant $C_h$ depending on h such that
\begin{equation}
\biggl||1+\frac{1}{l-r}|^{2h}+|1-\frac{1}{l-r}|^{2h}-2\biggr|\leq C_h \biggl(\frac{1}{l-r}\biggr)^2.
\end{equation}
Therefore, we have
\begin{equation}
\bigl(g(x_{2r}^{m+1})-g(x_{2r-1}^{m+1})\bigr)\bigl(g(x_{2l}^{m+1})-g(x_{2l-1}^{m+1})\bigr)\leq c \biggl(\frac{1}{2^m}\biggr)^{2h}\biggl(\frac{1}{l-r}\biggr)^{2-2h}.
\end{equation}
Hence, it follows that
\begin{eqnarray}\label{similar}\nn
&&
\quad\displaystyle\sum_{r<l}E\biggl[(L_t^{x_{2l-2}^{m+1}} - L_t^{x_{k-1}^{n}})(L_t^{x_{2r-2}^{m+1}} - L_t^{x_{k-1}^{n}})(L_t^{x_{2l-1}^{m+1}} - L_t^{x_{2l-2}^{m+1}})(L_t^{x_{2r-1}^{m+1}} - L_t^{x_{2r-2}^{m+1}})\nn\\
&&
\qquad\qquad\quad\times\bigl(g(x_{2r}^{m+1})-g(x_{2r-1}^{m+1})\bigr)\bigl(g(x_{2l}^{m+1})-g(x_{2l-1}^{m+1})\bigr)\biggr]\nn\\
&&
\leq \displaystyle\sum_{r<l}c\biggl(\frac{1}{2^m}\biggr)^{2h}\biggl(\frac{1}{l-r}\biggr)^{2-2h}\bigl(E(L_t^{x_{2l-2}^{m+1}} - L_t^{x_{k-1}^{n}})^4\bigr)^{\frac{1}{4}}\bigl(E(L_t^{x_{2r-2}^{m+1}} - L_t^{x_{k-1}^{n}})^4\bigr)^{\frac{1}{4}}\bigl(E(L_t^{x_{2l-1}^{m+1}} - L_t^{x_{2l-2}^{m+1}})^4\bigr)^{\frac{1}{4}}\nn\\
&&
\qquad\quad\quad\quad\times\bigl(E(L_t^{x_{2r-1}^{m+1}} - L_t^{x_{2r-2}^{m+1}})^4\bigr)^{\frac{1}{4}}\nn\\
&&
\leq \displaystyle\sum_{r<l} c \biggl(\frac{1}{2^m}\biggr)^{2h}\biggl(\frac{1}{l-r}\biggr)^{2-2h}t^{\frac{2(\alpha-1)}{\alpha}}\bigl(\frac{1}{2^n}\bigr)^{\alpha-1}\bigl(\frac{1}{2^m}\bigr)^{\alpha-1}\nn\\
&&
\leq c\: t^{\frac{2(\alpha-1)}{\alpha}}\:\bigl(\frac{1}{2^n}\bigr)^{\alpha}\bigl(\frac{1}{2^m}\bigr)^{2h+\alpha-2},
\end{eqnarray}
where one summation is consumed by $\biggl(\frac{1}{l-r}\biggr)^{2-2h}$ and the second inequality is due to the Lemma 3.1.

The estimation for the second term in $J_2$ is given by
\begin{eqnarray}\nn
&&
\quad E\Biggl|\displaystyle\sum_{r<l}\bigl(g(x_{2l-2}^{m+1}) - g({x_{k-1}^{n}})\bigr)\bigl(g({x_{2r-2}^{m+1}}) -g({x_{k-1}^{n}})\bigr)\bigl(g({x_{2l-1}^{m+1}}) - g({x_{2l-2}^{m+1}})\bigr)\nn\\
&&
\quad\quad\quad\quad\times\bigl(g({x_{2r-1}^{m+1}}) - g({x_{2r-2}^{m+1}})\bigr)\bigl(L_t^{x_{2r}^{m+1}}-L_t^{x_{2r-1}^{m+1}}\bigr)\bigl(L_t^{x_{2l}^{m+1}}-L_t^{x_{2l-1}^{m+1}}\bigr)\Biggr|\nn\\
&&
\leq c\, 2^{m-n}(\frac{1}{2^m})^{2h}(\frac{1}{2^n})^{2h}(\frac{1}{2^m})^{\alpha-1}\nn\\
&&
\leq c\,(\frac{1}{2^n})^{1+2h}(\frac{1}{2^m})^{2h+\alpha-2}.
\end{eqnarray}
Therefore, the bound for the term $EA_1^2$ is
\begin{equation}\label{eq: 4.15}
EA_1^2\leq c\{ t^{\frac{2(\alpha-1)}{\alpha}}\bigl(\frac{1}{2^n}\bigr)^{\alpha}\bigl(\frac{1}{2^m}\bigr)^{2h+\alpha-2}+(\frac{1}{2^n})^{1+2h}(\frac{1}{2^m})^{2h+\alpha-2}\}.
\end{equation}
Next we estimate the term $A_3$ in (\ref{eq: 4.12}), that is
\begin{eqnarray*}
&&
\quad E\displaystyle\biggl|\sum_{l}\Delta_{2l-1}^{m+1}Z\otimes\Delta_{2l}^{m+1}Z\otimes\Delta_{2l}^{m+1}Z\biggr|^2\nn\\
&&
=E\displaystyle\sum_{i,j,u}\displaystyle\sum_{l}\bigl(\Delta_{2l-1}^{m+1}Z^i\bigr)^2\bigl(\Delta_{2l}^{m+1}Z^j\bigr)^2\bigl(\Delta_{2l}^{m+1}Z^u\bigr)^2\nn\\
&&
\quad + \displaystyle\sum_{i, j, u}\displaystyle\sum_{r< l}E\biggl[\bigl(\Delta_{2r-1}^{m+1}Z^i\Delta_{2l-1}^{m+1}Z^i\bigr)\bigl(\Delta_{2r}^{m+1}Z^j\Delta_{2l}^{m+1}Z^j\bigr)\bigl(\Delta_{2r}^{m+1}Z^u\Delta_{2l}^{m+1}Z^u\bigr)\biggr]\nn\\
&&
\quad + \displaystyle\sum_{i}\displaystyle\sum_{r< l}E\bigl(\Delta_{2r-1}^{m+1}Z^i\Delta_{2l-1}^{m+1}Z^i\bigr)\bigl(\Delta_{2l}^{m+1}Z^i\bigr)^2\bigl(\Delta_{2r}^{m+1}Z^i\bigr)^2\nn\\
\end{eqnarray*}
\begin{eqnarray}\label{longequation}
&&
=\displaystyle\sum_{l}E\bigl(L_t^{x_{2l-1}^{m+1}} - L_t^{x_{2l-2}^{m+1}}\bigr)^2\bigl(L_t^{x_{2l}^{m+1}} - L_t^{x_{2l-1}^{m+1}}\bigr)^2\bigl(g(x_{2l}^{m+1})-g(x_{2l-1}^{m+1})\bigr)^2\nn\\
&&
\quad+\displaystyle\sum_{l}E\bigl(g(x_{2l-1}^{m+1}) - g({x_{2l-2}^{m+1}})\bigr)^2\bigl(g(x_{2l}^{m+1})-g(x_{2l-1}^{m+1})\bigr)^2\bigl(L_t^{x_{2l}^{m+1}} - L_t^{x_{2l-1}^{m+1}}\bigr)^2\nn\\
&&
\quad + \displaystyle\sum_{r< l}E\bigl(L_t^{x_{2r-1}^{m+1}}-L_t^{x_{2r-2}^{m+1}}\bigr)\bigl(L_t^{x_{2l-1}^{m+1}}-L_t^{x_{2l-2}^{m+1}}\bigr)\bigl(L_t^{x_{2r}^{m+1}}-L_t^{x_{2r-1}^{m+1}}\bigr)\nn\\
&&
\quad\quad\quad\bigl(L_t^{x_{2l}^{m+1}}-L_t^{x_{2l-1}^{m+1}}\bigr)\bigl(g({x_{2r}^{m+1}}) - g({x_{2r-1}^{m+1}})\bigr)\bigl(g({x_{2l}^{m+1}}) - g({x_{2l-1}^{m+1}})\bigr)\nn\\
&&
\quad + \displaystyle\sum_{r< l}E\bigl(g(x_{2r-1}^{m+1})-g(x_{2r-2}^{m+1})\bigr)\bigl(g(x_{2l-1}^{m+1})-g(x_{2l-2}^{m+1})\bigr)\bigl(g(x_{2r}^{m+1})-g(x_{2r-1}^{m+1})\bigr)\nn\\
&&
\quad\quad\quad\bigl(g(x_{2l}^{m+1})-g(x_{2l-1}^{m+1})\bigr)\bigl(L_t^{x_{2r}^{m+1}} - L_t^{x_{2r-1}^{m+1}}\bigr)\bigl(L_t^{x_{2l}^{m+1}} - L_t^{x_{2l-1}^{m+1}}\bigr)\nn\\
&&
\quad+ \displaystyle\sum_{r< l}E\bigl(L_t^{x_{2r-1}^{m+1}}-L_t^{x_{2r-2}^{m+1}}\bigr)\bigl(L_t^{x_{2l-1}^{m+1}}-L_t^{x_{2l-2}^{m+1}}\bigr)\bigl(L_t^{x_{2l}^{m+1}}-L_t^{x_{2l-1}^{m+1}}\bigr)^2\bigl(L_t^{x_{2r}^{m+1}}-L_t^{x_{2r-1}^{m+1}}\bigr)^2\nn\\
&&
\quad + \displaystyle\sum_{r< l}E(g(x_{2r-1}^{m+1})-g(x_{2r-2}^{m+1})\bigr)\bigl(g(x_{2l-1}^{m+1})-g(x_{2l-2}^{m+1})\bigr)\bigl(g(x_{2l}^{m+1})-g(x_{2l-1}^{m+1})\bigr)^2\bigl(g(x_{2r}^{m+1})-g(x_{2r-1}^{m+1})\bigr)^2\nn\\
&&
\leq\displaystyle\sum_{l}\biggl|E\bigl(L_t^{x_{2l-1}^{m+1}} - L_t^{x_{2l-2}^{m+1}}\bigr)^2\bigl(L_t^{x_{2l}^{m+1}} - L_t^{x_{2l-1}^{m+1}}\bigr)^2\bigl(g(x_{2l}^{m+1})-g(x_{2l-1}^{m+1})\bigr)^2\biggr|\nn\\
&&
\quad+\displaystyle\sum_{l}\biggl|E\bigl(g(x_{2l-1}^{m+1}) - g({x_{2l-2}^{m+1}})\bigr)^2\bigl(g(x_{2l}^{m+1})-g(x_{2l-1}^{m+1})\bigr)^2\bigl(L_t^{x_{2l}^{m+1}} - L_t^{x_{2l-1}^{m+1}}\bigr)^2\biggr|\nn\\
&&
\quad + \displaystyle\sum_{r< l}\biggl|E\bigl(L_t^{x_{2r-1}^{m+1}}-L_t^{x_{2r-2}^{m+1}}\bigr)\bigl(L_t^{x_{2l-1}^{m+1}}-L_t^{x_{2l-2}^{m+1}}\bigr)\bigl(L_t^{x_{2r}^{m+1}}-L_t^{x_{2r-1}^{m+1}}\bigr)\nn\\
&&
\quad\quad\quad\bigl(L_t^{x_{2l}^{m+1}}-L_t^{x_{2l-1}^{m+1}}\bigr)\bigl(g({x_{2r}^{m+1}}) - g({x_{2r-1}^{m+1}})\bigr)\bigl(g({x_{2l}^{m+1}}) - g({x_{2l-1}^{m+1}})\bigr)\biggr|\nn\\
&&
\quad + \displaystyle\sum_{r< l}\biggl|E\bigl(g(x_{2r-1}^{m+1})-g(x_{2r-2}^{m+1})\bigr)\bigl(g(x_{2l-1}^{m+1})-g(x_{2l-2}^{m+1})\bigr)\bigl(g(x_{2r}^{m+1})-g(x_{2r-1}^{m+1})\bigr)\nn\\
&&
\quad\quad\quad\bigl(g(x_{2l}^{m+1})-g(x_{2l-1}^{m+1})\bigr)\bigl(L_t^{x_{2r}^{m+1}} - L_t^{x_{2r-1}^{m+1}}\bigr)\bigl(L_t^{x_{2l}^{m+1}} - L_t^{x_{2l-1}^{m+1}}\bigr)\biggr|\nn\\
&&
\quad+ \displaystyle\sum_{r< l}\biggl|E\bigl(L_t^{x_{2r-1}^{m+1}}-L_t^{x_{2r-2}^{m+1}}\bigr)\bigl(L_t^{x_{2l-1}^{m+1}}-L_t^{x_{2l-2}^{m+1}}\bigr)\bigl(L_t^{x_{2l}^{m+1}}-L_t^{x_{2l-1}^{m+1}}\bigr)^2\bigl(L_t^{x_{2r}^{m+1}}-L_t^{x_{2r-1}^{m+1}}\bigr)^2\biggr|\nn\\
&&
\quad + \displaystyle\sum_{r< l}\biggl|E(g(x_{2r-1}^{m+1})-g(x_{2r-2}^{m+1})\bigr)\bigl(g(x_{2l-1}^{m+1})-g(x_{2l-2}^{m+1})\bigr)\bigl(g(x_{2l}^{m+1})-g(x_{2l-1}^{m+1})\bigr)^2\nn\\
&&
\quad\quad\quad\bigl(g(x_{2r}^{m+1})-g(x_{2r-1}^{m+1})\bigr)^2\biggr|.
\end{eqnarray}
Using Lemma 3.1 and Cauchy Schwarz inequality, we obtain the bound for the first term 
\begin{eqnarray}\nn
&&
\quad\displaystyle\sum_{l}\biggl|E\bigl(L_t^{x_{2l-1}^{m+1}} - L_t^{x_{2l-2}^{m+1}}\bigr)^2\bigl(L_t^{x_{2l}^{m+1}} - L_t^{x_{2l-1}^{m+1}}\bigr)^2\bigl(g(x_{2l}^{m+1})-g(x_{2l-1}^{m+1})\bigr)^2\biggr|\nn\\
&&
\leq c\:t^{\frac{2(\alpha-1)}{\alpha}}\frac{1}{2^n} \bigl(\frac{1}{2^m}\bigr)^{2h+2\alpha-3}.
\end{eqnarray}
The bound of the second term is
\begin{equation}
\displaystyle\sum_{l}\biggl|E\bigl(g(x_{2l-1}^{m+1}) - g({x_{2l-2}^{m+1}})\bigr)^2\bigl(g(x_{2l}^{m+1})-g(x_{2l-1}^{m+1})\bigr)^2\bigl(L_t^{x_{2l}^{m+1}} - L_t^{x_{2l-1}^{m+1}}\bigr)^2\biggr|
\leq c\,\frac{1}{2^n} \bigl(\frac{1}{2^m}\bigr)^{4h+\alpha-2}.
\end{equation}
Similar to (\ref{similar}), the bound of the third term is
\begin{eqnarray}\nn
&&
\quad\displaystyle\sum_{r< l}\biggl|E\bigl(L_t^{x_{2r-1}^{m+1}}-L_t^{x_{2r-2}^{m+1}}\bigr)\bigl(L_t^{x_{2l-1}^{m+1}}-L_t^{x_{2l-2}^{m+1}}\bigr)\bigl(L_t^{x_{2r}^{m+1}}-L_t^{x_{2r-1}^{m+1}}\bigr)\nn\\
&&
\quad\quad\quad\quad\bigl(L_t^{x_{2l}^{m+1}}-L_t^{x_{2l-1}^{m+1}}\bigr)\bigl(g({x_{2r}^{m+1}}) - g({x_{2r-1}^{m+1}})\bigr)\bigl(g({x_{2l}^{m+1}}) - g({x_{2l-1}^{m+1}})\bigr)\biggr|\nn\\
&&
\leq c\: t^{\frac{2(\alpha-1)}{\alpha}} \frac{1}{2^n}\bigl(\frac{1}{2^m}\bigr)^{2h+2\alpha-3}.
\end{eqnarray}

The bound of the fourth term is
\begin{eqnarray}\nn
&&
\quad\displaystyle\sum_{r< l}\biggl|E\bigl(g(x_{2r-1}^{m+1})-g(x_{2r-2}^{m+1})\bigr)\bigl(g(x_{2l-1}^{m+1})-g(x_{2l-2}^{m+1})\bigr)\bigl(g(x_{2r}^{m+1})-g(x_{2r-1}^{m+1})\bigr)\nn\\
&&
\quad\quad\quad\quad\bigl(g(x_{2l}^{m+1})-g(x_{2l-1}^{m+1})\bigr)\bigl(L_t^{x_{2r}^{m+1}} - L_t^{x_{2r-1}^{m+1}}\bigr)\bigl(L_t^{x_{2l}^{m+1}} - L_t^{x_{2l-1}^{m+1}}\bigr)\biggr|\nn\\
&&
\leq c\,\frac{1}{2^n}(\frac{1}{2^m})^{4h+\alpha-2}.
\end{eqnarray}
The estimation of the fifth term
\begin{equation}
\displaystyle\sum_{r< l}\biggl|E\bigl(L_t^{x_{2r}^{m+1}}-L_t^{x_{2r-1}^{m+1}}\bigr)^2\bigl(L_t^{x_{2l}^{m+1}}-L_t^{x_{2l-1}^{m+1}}\bigr)^2\bigl(L_t^{x_{2l-1}^{m+1}}-L_t^{x_{2l-2}^{m+1}}\bigr)\bigl(L_t^{x_{2r-1}^{m+1}}-L_t^{x_{2r-2}^{m+1}}\bigr)\biggr|
\end{equation}
is more involved. We use the following equation from \cite{Marcusr2006} in the estimation of the fifth term

\begin{equation}\label{Rosen1}
E\bigl(\displaystyle\prod_{i=1}^{n}L_t^{y_i}\bigr)=\displaystyle\sum_{\pi}E\biggl(\displaystyle\idotsint\limits_{0\leq t_1\leq t_2<\dots\leq t_n\leq t}\displaystyle\prod_{i=1}^{n}dL_{t_i}^{y_{\pi_i}}\biggr)
\end{equation}
where the sum in (\ref{Rosen1}) runs over all permutation $\pi$ of $\{1,2,\dots, n\}$.

In the following, we estimate 
\begin{equation}E\bigl(L_t^{x_{2r}^{m+1}}-L_t^{x_{2r-1}^{m+1}}\bigr)^2\bigl(L_t^{x_{2l}^{m+1}}-L_t^{x_{2l-1}^{m+1}}\bigr)^2\bigl(L_t^{x_{2l-1}^{m+1}}-L_t^{x_{2l-2}^{m+1}}\bigr)\bigl(L_t^{x_{2r-1}^{m+1}}-L_t^{x_{2r-2}^{m+1}}\bigr).\end{equation}

However, it suffices to only estimate a particular term in the summation $\pi$, it means we impose certain restriction on $\pi_i$ and only estimate terms associated with that $\pi_i$. Before we estimate this fifth term, we recall Lemma 2.4.6 from \cite{Marcusr2006}
\begin{lemma}\label{lemmaexpectation}
For any positive measurable function $f(t)$ and any $T\in[0,\infty)$ and $z\in \mathbb{R}$ we have
\[
\int_{0}^{T}f(t)dL_t^z=\int_{0}^{\infty}f(\tau_z(s))1_{\tau_z(s)<T}ds,\]
where $\left\{\tau(s), s\in\mathbb{R+}\right\}$ is a positive increasing stochastic process with stationary and independent increments.

If $H_t$ is a positive continuous $\mathcal{F}_t$ measurable function, F a positive $\mathcal{F}$ measurable function, and T a stopping time (possibly $T\equiv\infty$), then
\begin{equation}\label{willbeused}
E^x\biggl(\int_{0}^T H_t F\circ \theta_tdL_t^z\biggr)=E^z(F)E^x\biggl(\int_0^T H_tdL_t^z\biggr),
\end{equation}
where $\theta_t$ is a shift operator with $\theta_t\circ\theta_s=\theta_{t+s}$.
\end{lemma}
\noindent{\bf{Illustration}} In order to show how the Lemma \ref{lemmaexpectation} can be applied in what follows, we use the lemma to prove
\begin{eqnarray}\label{lemmaillu}&&
\quad E\biggl(\int_0^t\int_{t_1}^t(dL_{t_2}^{x_1}-dL_{t_2}^{x_2})(dL_{t_1}^{x_3}-dL_{t_1}^{x_4})\biggr)\nn\\ &&=E\int_0^t\biggl[E^{x_3}\biggl(\int_{0}^t dL_{t_2}^{x_1}-dL_{t_2}^{x_2}\biggr)dL_{t_1}^{x_3}-E^{x_4}\biggl(\int_{0}^t dL_{t_2}^{x_1}-dL_{t_2}^{x_2}\biggr)dL_{t_1}^{x_4}\biggr].\nn\\\end{eqnarray}

\begin{proof} First, we rewrite l.h.s. of (\ref{lemmaillu}) as following
\begin{eqnarray}\label{illustration3}&&
\quad\int_0^t\int_{t_1}^t(dL_{t_2}^{x_1}-dL_{t_2}^{x_2})(dL_{t_1}^{x_3}-dL_{t_1}^{x_4})\nn\\
&&
=\int_0^t\int_{t_1}^t dL_{t_2}^{x_1}.dL_{t_1}^{x_3}-\int_0^t\int_{t_1}^t dL_{t_2}^{x_1}.dL_{t_1}^{x_4}-\int_0^t\int_{t_1}^t dL_{t_2}^{x_2}.dL_{t_1}^{x_3}+\int_0^t\int_{t_1}^t dL_{t_2}^{x_2}.dL_{t_1}^{x_4}.\nn\\
\end{eqnarray}

We define, for example, $\int_{t_1}^tdL_{t_2}^{x_1}:=F(0,t)\circ\theta_{t_1}$ with $F(0,t):=\int_{0}^tdL_{t_2}^{x_1}.$ We apply Lemma \ref{lemmaexpectation} to the first term in (\ref{illustration3}) with $H_t\equiv 1$, then we have 
\begin{equation}
E\Biggl(\int_0^t\biggl(\int_{t_1}^tdL_{t_2}^{x_1}\biggr).dL_{t_1}^{x_3}\Biggr)=E^{x_3}\biggl(\int_{0}^t dL_{t_2}^{x_1}\biggr)E\biggl(\int_0^tdL_{t_1}^{x_3}\biggr)=E\biggl(L_{t}^{x_3}\biggr).E^{x_3}\biggl(\int_{0}^t dL_{t_2}^{x_1}\biggr),\\
\end{equation}

Similarly, we have 
\begin{align*}
-E\biggl(\int_0^t\int_{t_1}^t dL_{t_2}^{x_1}.dL_{t_1}^{x_4}\biggr)&=-E\biggl( L_{t}^{x_4}\biggr).E^{x_4}\biggl(\int_{0}^t dL_{t_2}^{x_1}\biggr),\\
-E\biggl(\int_0^t\int_{t_1}^t dL_{t_2}^{x_2}.dL_{t_1}^{x_3}\biggr)&=-E\biggl( L_{t}^{x_3}\biggr).E^{x_3}\biggl(\int_{0}^t dL_{t_2}^{x_2}\biggr),\\
E\biggl(\int_0^t\int_{t_1}^t dL_{t_2}^{x_2}.dL_{t_1}^{x_4}\biggr)&=E\biggl( L_{t}^{x_4}\biggr).E^{x_4}\biggl(\int_{0}^t dL_{t_2}^{x_2}\biggr).
\end{align*}
Hence, this concludes the proof, since
\begin{eqnarray}&&
\quad E\biggl(\int_0^t\int_{t_1}^t(dL_{t_2}^{x_1}-dL_{t_2}^{x_2})(dL_{t_1}^{x_3}-dL_{t_1}^{x_4})\biggr)\nn\\
&&
=E\biggl(L_{t}^{x_3}\biggr).E^{x_3}\biggl(\int_{0}^t dL_{t_2}^{x_1}-dL_{t_2}^{x_2}\biggr)-E\biggl(L_{t}^{x_4}\biggr).E^{x_4}\biggl(\int_{0}^t dL_{t_2}^{x_1}-dL_{t_2}^{x_2}\biggr)\nn\\
&&
=E\biggl[\int_{0}^{t}E^{x_3}\biggl(\int_{0}^t dL_{t_2}^{x_1}-dL_{t_2}^{x_2}\biggr).dL_{t_1}^{x_3}\biggr]-E\biggl[\int_0^tE^{x_4}\biggl(\int_{0}^t dL_{t_2}^{x_1}-dL_{t_2}^{x_2}\biggr).dL_{t_1}^{x_4}\biggr].
\end{eqnarray}
\end{proof}
Although there are many terms in the summation $\pi$ but there are only a finite number of terms. We estimate the following particular term only.  Other terms can be estimated similarly. First note
\begin{eqnarray}
&&
\quad E\int_{0}^{t}\int_{t_1}^{t}\int_{t_2}^{t}\int_{t_3}^{t}\int_{t_4}^{t}\int_{t_5}^{t}\biggl(dL_{t_6}^{x_{2l}}-dL_{t_6}^{x_{2l-1}}\biggr).\biggl(dL_{t_5}^{x_{2l}}-dL_{t_5}^{x_{2l-1}}\biggr)\nn\\
&&
\quad\quad\qquad\qquad\qquad\qquad.\biggl(dL_{t_4}^{x_{2r}}-dL_{t_4}^{x_{2r-1}}\biggr).\biggl(dL_{t_3}^{x_{2r}}-dL_{t_3}^{x_{2r-1}}\biggr)\nn\\
&&
\quad\quad\qquad\qquad\qquad\qquad.\biggl(dL_{t_2}^{x_{2l-1}}-dL_{t_2}^{x_{2l-2}}\biggr).\biggl(dL_{t_1}^{x_{2r-1}}-dL_{t_1}^{x_{2r-2}}\biggr)\nn\\
&&
=E\int_{0}^{t}\int_{t_1}^{t}\int_{t_2}^{t}\int_{t_3}^{t}\Biggl[\int_{t_4}^{t}E^{x_{2l}}\biggl(\int_{0}^{t}dL_{t_6}^{x_{2l}}-dL_{t_6}^{x_{2l-1}}\biggr). dL_{t_5}^{x_{2l}} \nn\\
&&
\quad\quad\qquad\qquad\qquad\qquad-E^{x_{2l-1}}\biggl(\int_{0}^{t}dL_{t_6}^{x_{2l}}-dL_{t_6}^{x_{2l-1}}\biggr).dL_{t_5}^{x_{2l-1}}\Biggr]\nn\\
&&
\quad\quad\qquad\qquad\qquad\qquad.\biggl(dL_{t_4}^{x_{2r}}-dL_{t_4}^{x_{2r-1}}\biggr).\biggl(dL_{t_3}^{x_{2r}}-dL_{t_3}^{x_{2r-1}}\biggr)\nn\\
&&
\quad\quad\qquad\qquad\qquad\qquad.\biggl(dL_{t_2}^{x_{2l-1}}-dL_{t_2}^{x_{2l-2}}\biggr).\biggl(dL_{t_1}^{x_{2r-1}}-dL_{t_1}^{x_{2r-2}}\biggr),
\end{eqnarray}
where we have used (\ref{willbeused}) with $H_t\equiv 1$ in the Lemma \ref{lemmaexpectation} for the two innermost integrals.

We use $p_t(x,y)$ to denote the transitional probability density function of the symmetric stable process X. We recall only partial results of Theorem 3.6.5 from \cite{Marcusr2006} which will be used in our estimation (as symmetric stable L\'evy processes is a class of Borel right processes, hence we simply replace the Borel right procesess in the original theorem by symmetric stable process instead.).

\begin{lemma}\label{speciallemma} Let X be a strongly symmetric stable process and assume that its $\beta$-potential density, $u^\beta(x,y)$, is finite for all x, y $\in S$ (S is the state space of the process.). Let $L_t^y$ be a local time of X at y, with 
\begin{equation*}
E^x\biggl(\int_{0}^{\infty}e^{-\beta t}dL_t^y\biggr)=u^\beta(x,y).
\end{equation*}
Then for every t
\begin{equation}\label{lemma2.4.3.1}
E^x(L_t^y)=\int_0^t p_s(x,y)ds.\end{equation}


\end{lemma}
\noindent One may write $u^{\beta}(y-x)=u^{\beta}(x,y)$ as pointed out in \cite{Marcusr2006}.

By (\ref{lemma2.4.3.1}) and the property $p_t(x,y)=p_t(x-y)$, then we have 
\begin{eqnarray}
&&
E^{x_{2l}}\biggl(\int_{0}^{t}dL_{t_6}^{x_{2l}}-dL_{t_6}^{x_{2l-1}}\biggr)=E^{x_{2l}}\biggl(L_{t}^{x_{2l}}-L_{t}^{x_{2l-1}}\biggr)\nn\\
&&
\qquad\qquad\quad\qquad\qquad\qquad\quad=\int_{0}^{t}\biggl(p_{s}(0)-p_{s}(x_{2l}-x_{2l-1})\biggr)ds\nn\\
&&
\qquad\qquad\quad\qquad\qquad\qquad\quad=\int_{0}^{t}\biggl(p_{\Delta_{t_6}}(0)-p_{\Delta_{t_6}}(x_{2l}-x_{2l-1})\biggr)dt_6.
\end{eqnarray}
In the last equality of (3.92), we have purposely used $dt_6$ instead of $ds$ to indicate that this is for the innermost integral.
\,Then,\,we have
\begin{eqnarray}\label{eq: 4.14}
&&
\quad E\int_{0}^{t}\int_{t_1}^{t}\int_{t_2}^{t}\int_{t_3}^{t}\Biggl[\int_{t_4}^{t}E^{x_{2l}}\biggl(\int_{0}^{t}dL_{t_6}^{x_{2l}}-dL_{t_6}^{x_{2l-1}}\biggr). dL_{t_5}^{x_{2l}} \nn\\
&&
\quad\quad\qquad\qquad\qquad\qquad-E^{x_{2l-1}}\biggl(\int_{t_5}^{t}dL_{t_6}^{x_{2l}}-dL_{t_6}^{x_{2l-1}}\biggr).dL_{t_5}^{x_{2l-1}}\Biggr]\nn\\
&&
\quad\quad\qquad\qquad\qquad\qquad.\biggl(dL_{t_4}^{x_{2r}}-dL_{t_4}^{x_{2r-1}}\biggr).\biggl(dL_{t_3}^{x_{2r}}-dL_{t_3}^{x_{2r-1}}\biggr)\nn\\
&&
\quad\quad\qquad\qquad\qquad\qquad.\biggl(dL_{t_2}^{x_{2l-1}}-dL_{t_2}^{x_{2l-2}}\biggr).\biggl(dL_{t_1}^{x_{2r-1}}-dL_{t_1}^{x_{2r-2}}\biggr)\nn\\
&&
=E\int_{0}^{t}\int_{t_1}^{t}\int_{t_2}^{t}\int_{t_3}^{t}\Biggl[\int_{t_4}^{t}\int_{0}^{t}\biggl(p_{\Delta_{t_6}}(0)-p_{\Delta_{t_6}}(x_{2l}-x_{2l-1})\biggr)dt_6. dL_{t_5}^{x_{2l}} \nn\\
&&
\quad\quad\qquad\qquad\qquad\qquad-\biggl(p_{\Delta_{t_6}}(x_{2l}-x_{2l-1})-p_{\Delta_{t_6}}(0)\biggr)dt_6.dL_{t_5}^{x_{2l-1}}\Biggr]\nn\\
&&
\quad\quad\qquad\qquad\qquad\qquad.\biggl(dL_{t_4}^{x_{2r}}-dL_{t_4}^{x_{2r-1}}\biggr).\biggl(dL_{t_3}^{x_{2r}}-dL_{t_3}^{x_{2r-1}}\biggr)\nn\\
&&
\quad\quad\qquad\qquad\qquad\qquad.\biggl(dL_{t_2}^{x_{2l-1}}-dL_{t_2}^{x_{2l-2}}\biggr).\biggl(dL_{t_1}^{x_{2r-1}}-dL_{t_1}^{x_{2r-2}}\biggr),\label{eq: important}
\end{eqnarray}
where $\Delta_{t_i}=t_i-t_{i-1}$. 

Iterating the procedure in (\ref{eq: 4.14}), we obtain the following, where we have used the notation $\Delta_r p(x)=p_r (0)- p_r(x)$
\begin{eqnarray*} 
&&
\quad E\int_{0}^{t}\int_{t_1}^{t}\int_{t_2}^{t}\int_{t_3}^{t}\int_{t_4}^{t}\int_{t_5}^{t}\biggl(dL_{t_6}^{x_{2l}}-dL_{t_6}^{x_{2l-1}}\biggr).\biggl(dL_{t_5}^{x_{2l}}-dL_{t_5}^{x_{2l-1}}\biggr)\nn\\
&&
\quad\quad\qquad\qquad\qquad\qquad.\biggl(dL_{t_4}^{x_{2r}}-dL_{t_4}^{x_{2r-1}}\biggr).\biggl(dL_{t_3}^{x_{2r}}-dL_{t_3}^{x_{2r-1}}\biggr)\nn\\
&&
\quad\quad\qquad\qquad\qquad\qquad.\biggl(dL_{t_2}^{x_{2l-1}}-dL_{t_2}^{x_{2l-2}}\biggr).\biggl(dL_{t_1}^{x_{2r-1}}-dL_{t_1}^{x_{2r-2}}\biggr)\nn\\
&&
=E\int_{0}^{t}\int_{t_1}^{t}\int_{t_2}^{t}\int_{t_3}^{t}\Biggl[\int_{t_4}^{t}\int_{0}^{t}\biggl(p_{\Delta_{t_6}}(0)-p_{\Delta_{t_6}}(x_{2l}-x_{2l-1})\biggr)dt_6. \biggl(dL_{t_5}^{x_{2l}} + dL_{t_5}^{x_{2l-1}}\biggr)\Biggr]\nn\\
&&
\quad\quad\qquad\qquad\qquad\qquad.\biggl(dL_{t_4}^{x_{2r}}-dL_{t_4}^{x_{2r-1}}\biggr).\biggl(dL_{t_3}^{x_{2r}}-dL_{t_3}^{x_{2r-1}}\biggr)\nn\\
&&
\quad\quad\qquad\qquad\qquad\qquad.\biggl(dL_{t_2}^{x_{2l-1}}-dL_{t_2}^{x_{2l-2}}\biggr).\biggl(dL_{t_1}^{x_{2r-1}}-dL_{t_1}^{x_{2r-2}}\biggr)\nn\\
&&
=E\int_{0}^{t}\int_{t_1}^{t}\int_{t_2}^{t}\Biggl[\int_{t_3}^{t}E^{x_{2r}}\int_{0}^{t}\int_{0}^{t}{\Delta_{t_6}}p(x_{2l}-x_{2l-1})dt_6. \biggl(dL_{t_5}^{x_{2l}} + dL_{t_5}^{x_{2l-1}}\biggr).dL_{t_4}^{x_{2r}}\nn\\
&&
\qquad\quad\quad-E^{x_{2r-1}}\int_{0}^{t}\int_{0}^{t}{\Delta_{t_6}}p(x_{2l}-x_{2l-1})dt_6. \biggl(dL_{t_5}^{x_{2l}} + dL_{t_5}^{x_{2l-1}}\biggr).dL_{t_4}^{x_{2r-1}}\Biggr]\nn\\
&&
\qquad\quad\quad.\biggl(dL_{t_3}^{x_{2r}}-dL_{t_3}^{x_{2r-1}}\biggr).\biggl(dL_{t_2}^{x_{2l-1}}-dL_{t_2}^{x_{2l-2}}\biggr).\biggl(dL_{t_1}^{x_{2r-1}}-dL_{t_1}^{x_{2r-2}}\biggr)\nn\\
\end{eqnarray*}
\begin{eqnarray*}
&&
=E\int_{0}^{t}\int_{t_1}^{t}\int_{t_2}^{t}\Biggl[\int_{t_3}^{t}\int_{0}^{t}\int_{0}^{t}{\Delta_{t_6}}p(x_{2l}-x_{2l-1})dt_6. \biggl(p_{\Delta_{t_5}}(x_{2r}-x_{2l})\nn\\
&&
\qquad\quad\quad\qquad\quad\quad\qquad\quad\qquad\quad\quad\qquad\quad\quad\quad+p_{\Delta_{t_5}}(x_{2r}-x_{2l-1})\biggr)dt_5.dL_{t_4}^{x_{2r}}\nn\\
&&
\qquad\quad\quad-{\Delta_{t_6}}p(x_{2l}-x_{2l-1})dt_6. \biggl(p_{\Delta_{t_5}}(x_{2r-1}-x_{2l})+
p_{\Delta_{t_5}}(x_{2r-1}-x_{2l-1})\biggr)dt_5.dL_{t_4}^{x_{2r-1}}\Biggr]\nn\\
&&
\qquad\quad\quad.\biggl(dL_{t_3}^{x_{2r}}-dL_{t_3}^{x_{2r-1}}\biggr).\biggl(dL_{t_2}^{x_{2l-1}}-dL_{t_2}^{x_{2l-2}}\biggr).\biggl(dL_{t_1}^{x_{2r-1}}-dL_{t_1}^{x_{2r-2}}\biggr)\nn\\
&&
=E\int_{0}^{t}\int_{t_1}^{t}\Biggl\{\int_{t_2}^{t}\int_{0}^{t}\int_{0}^{t}\int_{0}^{t}\biggl[\Delta_{t_6}p(x_{2l}-x_{2l-1})dt_6. \biggl(p_{\Delta_{t_5}}(x_{2r}-x_{2l})\nn\\
&&
\qquad\quad\quad\qquad\quad\quad\qquad\quad\qquad\quad\quad\qquad\quad\quad\quad+p_{\Delta_{t_5}}(x_{2r}-x_{2l-1})\biggr)dt_5.dL_{t_4}^{x_{2r}}\nn\\
&&
\qquad\quad-\Delta_{t_6}p(x_{2l}-x_{2l-1})dt_6. \biggl(p_{\Delta_{t_5}}(x_{2r-1}-x_{2l})+p_{\Delta_{t_5}}(x_{2r-1}-x_{2l-1})\biggr)dt_5.dL_{t_4}^{x_{2r-1}}\biggr]\nn\\
&&
\qquad\quad\quad\biggl(dL_{t_3}^{x_{2r}}-dL_{t_3}^{x_{2r-1}}\biggr)\Biggr\}.\biggl(dL_{t_2}^{x_{2l-1}}-dL_{t_2}^{x_{2l-2}}\biggr).\biggl(dL_{t_1}^{x_{2r-1}}-dL_{t_1}^{x_{2r-2}}\biggr)\nn\\
&&
=E\int_{0}^{t}\int_{t_1}^{t}\Biggl\{\int_{t_2}^{t}\int_{0}^{t}\int_{0}^{t}\int_{0}^{t}\Delta_{t_6}p(x_{2l}-x_{2l-1})dt_6. \biggl[\biggl(p_{\Delta_{t_5}}(x_{2r}-x_{2l})\nn\\
&&
\qquad\quad\quad+p_{\Delta_{t_5}}(x_{2r}-x_{2l-1})\biggr)dt_5.\biggl(p_{\Delta_{t_4}}(0)-p_{\Delta_{t_4}}(x_{2r-1}-x_{2r})\biggr)dt_4\nn\\
&&
\qquad\quad\quad- \biggl(p_{\Delta_{t_5}}(x_{2r-1}-x_{2l})+p_{\Delta_{t_5}}(x_{2r-1}-x_{2l-1})\biggr)dt_5.\biggl(p_{\Delta_{t_4}}(x_{2r-1}-x_{2r})-p_{\Delta_{t_4}}(0)\biggr)dt_4\biggr]\nn\\
&&
\qquad\quad\biggl(dL_{t_3}^{x_{2r}}-dL_{t_3}^{x_{2r-1}}\biggr)\Biggr\}\biggl(dL_{t_2}^{x_{2l-1}}-dL_{t_2}^{x_{2l-2}}\biggr).\biggl(dL_{t_1}^{x_{2r-1}}-dL_{t_1}^{x_{2r-2}}\biggr)\nn\\
&&
=E\int_{0}^{t}\int_{t_1}^{t}\Biggl[\int_{0}^{t}\int_{0}^{t}\int_{0}^{t}\int_{0}^{t}\Delta_{t_6}p(x_{2l}-x_{2l-1})dt_6. \biggl(p_{\Delta_{t_5}}(x_{2r}-x_{2l})\nn\\
&&
\qquad\qquad+p_{\Delta_{t_5}}(x_{2r}-x_{2l-1})+ p_{\Delta_{t_5}}(x_{2r-1}-x_{2l})+p_{\Delta_{t_5}}(x_{2r-1}-x_{2l-1})\biggr)dt_5\nn\\
&&
\qquad\quad\qquad.\Delta_{t_4}p(x_{2r-1}-x_{2r})dt_4.\biggl(dL_{t_3}^{x_{2r}}-dL_{t_3}^{x_{2r-1}}\biggr)\Biggr]\biggl(dL_{t_2}^{x_{2l-1}}-dL_{t_2}^{x_{2l-2}}\biggr).\biggl(dL_{t_1}^{x_{2r-1}}-dL_{t_1}^{x_{2r-2}}\biggr)\nn\\
&&
=E\int_{0}^{t}\int_{t_1}^{t}\Biggl[\int_{0}^{t}\int_{0}^{t}\int_{0}^{t}\int_{0}^{t}\Delta_{t_6}p(x_{2l}-x_{2l-1})dt_6.\biggl(p_{\Delta_{t_5}}(x_{2r}-x_{2l})\nn\\
&&
\quad\qquad+p_{\Delta_{t_5}}(x_{2r}-x_{2l-1}) 
+ p_{\Delta_{t_5}}(x_{2r-1}-x_{2l})+p_{\Delta_{t_5}}(x_{2r-1}-x_{2l-1})\biggr)dt_5\nn\\
&&
\quad\qquad\quad.\Delta_{t_4}p(x_{2r-1}-x_{2r})dt_4
.\biggl(p_{\Delta_{t_3}}(x_{2l-1}-x_{2r})-p_{\Delta_{t_3}}(x_{2l-2}-x_{2r})\nn\\
&&
\quad\qquad\quad-p_{\Delta_{t_3}}(x_{2l-1}-x_{2r-1})+p_{\Delta_{t_3}}(x_{2l-2}-x_{2r-1})\biggr)dt_3\biggl(dL_{t_2}^{x_{2l-1}}-dL_{t_2}^{x_{2l-2}}\biggr)\Biggr]\nn\\
&&
\qquad\quad.\biggl(dL_{t_1}^{x_{2r-1}}-dL_{t_1}^{x_{2r-2}}\biggr)\nn\\
\end{eqnarray*}
\begin{eqnarray}\nn\\\label{eq: importantequal01}
&&
=E\int_{0}^{t}\int_{0}^{t}\int_{0}^{t}\int_{0}^{t}\int_{0}^{t}\int_{0}^{t}\Delta_{t_6}p(x_{2l}-x_{2l-1})dt_6. \biggl(p_{\Delta_{t_5}}(x_{2l}-x_{2r})\nn\\
&&
\qquad\qquad\qquad\qquad\qquad\qquad\qquad+p_{\Delta_{t_5}}(x_{2l-1}-x_{2r})+ p_{\Delta_{t_5}}(x_{2l}-x_{2r-1})\nn\\
&&
\qquad\qquad\qquad\qquad\qquad\qquad+p_{\Delta_{t_5}}(x_{2l-1}-x_{2r-1})\biggr) dt_5.\Delta_{t_4}p(x_{2r}-x_{2r-1})dt_4\nn\\
&&
\qquad\qquad\qquad\qquad.\Biggl(p_{\Delta_{t_3}}(x_{2l-1}-x_{2r})-p_{\Delta_{t_3}}(x_{2l-2}-x_{2r})\nn\\
&&
\qquad\qquad\qquad\qquad\qquad\qquad+p_{\Delta_{t_3}}(x_{2l-2}-x_{2r-1})-p_{\Delta_{t_3}}(x_{2l-1}-x_{2r-1})\Biggr)dt_3\nn\\
&&
\qquad\qquad\qquad\qquad.\Biggl(p_{\Delta_{t_2}}(x_{2r-1}-x_{2l-1})-p_{\Delta_{t_2}}(x_{2r-2}-x_{2l-1})\nn\\
&&
\qquad\qquad\qquad\qquad\qquad\qquad +p_{\Delta_{t_2}}(x_{2r-2}-x_{2l-2})-p_{\Delta_{t_2}}(x_{2r-1}-x_{2l-2})\Biggr)dt_2\nn\\
&&
\qquad\qquad\qquad\qquad.\biggl(p_{{t_1}}(x_{2r-1})-p_{{t_1}}(x_{2r-2})\biggr)dt_1.
\end{eqnarray}
By the first inequality of (10.173) in \cite{Marcusr2006}, that is $p_t(x)\leq p_t(0)$. The last step of the (\ref{eq: importantequal01}) becomes
\begin{eqnarray*}\nn
&&
\quad\biggl|E\int_{0}^{t}\int_{0}^{t}\int_{0}^{t}\int_{0}^{t}\int_{0}^{t}\int_{0}^{t}\Delta_{t_6}p(x_{2l}-x_{2l-1})dt_6. \biggl(p_{\Delta_{t_5}}(x_{2l}-x_{2r})\nn\\
&&
\qquad\qquad\qquad\qquad\qquad\qquad\qquad+p_{\Delta_{t_5}}(x_{2l-1}-x_{2r})+ p_{\Delta_{t_5}}(x_{2l}-x_{2r-1})\nn\\
&&
\qquad\qquad\qquad\qquad\qquad\qquad+p_{\Delta_{t_5}}(x_{2l-1}-x_{2r-1})\biggr) dt_5.\Delta_{t_4}p(x_{2r}-x_{2r-1})dt_4\nn\\
&&
\qquad\qquad\qquad\qquad.\biggl(p_{\Delta_{t_3}}(x_{2l-1}-x_{2r})-p_{\Delta_{t_3}}(x_{2l-2}-x_{2r})\nn\\
&&
\qquad\qquad\qquad\qquad\qquad\qquad+p_{\Delta_{t_3}}(x_{2l-2}-x_{2r-1})-p_{\Delta_{t_3}}(x_{2l-1}-x_{2r-1})\biggr)dt_3\nn\\
&&
\qquad\qquad\qquad\qquad.\biggl(p_{\Delta_{t_2}}(x_{2r-1}-x_{2l-1})-p_{\Delta_{t_2}}(x_{2r-2}-x_{2l-1})\nn\\
&&
\qquad\qquad\qquad\qquad\qquad\qquad +p_{\Delta_{t_2}}(x_{2r-2}-x_{2l-2})-p_{\Delta_{t_2}}(x_{2r-1}-x_{2l-2})\biggr)dt_2\nn\\
&&
\qquad\qquad\qquad\qquad.\biggl(p_{{t_1}}(x_{2r-1})-p_{{t_1}}(x_{2r-2})\biggr)dt_1\biggr|\nn\\
&&
\leq E\int_{0}^{t}\biggl|\int_{0}^{t}\int_{0}^{t}\int_{0}^{t}\int_{0}^{t}\int_{0}^{t}\Delta_{t_6}p(x_{2l}-x_{2l-1})dt_6. \biggl(p_{\Delta_{t_5}}(x_{2l}-x_{2r})\nn\\
&&
\qquad\qquad\qquad\qquad\qquad\qquad\qquad+p_{\Delta_{t_5}}(x_{2l-1}-x_{2r})+ p_{\Delta_{t_5}}(x_{2l}-x_{2r-1})\nn\\
&&
\qquad\qquad\qquad\qquad\qquad\qquad+p_{\Delta_{t_5}}(x_{2l-1}-x_{2r-1})\biggr) dt_5.\Delta_{t_4}p(x_{2r}-x_{2r-1})dt_4\nn\\
&&
\qquad\qquad\qquad\qquad.\biggl(p_{\Delta_{t_3}}(x_{2l-1}-x_{2r})-p_{\Delta_{t_3}}(x_{2l-2}-x_{2r})\nn\\
&&
\qquad\qquad\qquad\qquad\qquad\qquad+p_{\Delta_{t_3}}(x_{2l-2}-x_{2r-1})-p_{\Delta_{t_3}}(x_{2l-1}-x_{2r-1})\biggr)dt_3\nn\\
\end{eqnarray*}
\begin{eqnarray}\label{eq:importantequal}\nn
&&
\qquad\qquad\qquad\qquad.\biggl(p_{\Delta_{t_2}}(x_{2r-1}-x_{2l-1})-p_{\Delta_{t_2}}(x_{2r-2}-x_{2l-1})\nn\\
&&
\qquad\qquad\qquad\qquad\qquad\qquad +p_{\Delta_{t_2}}(x_{2r-2}-x_{2l-2})-p_{\Delta_{t_2}}(x_{2r-1}-x_{2l-2})\biggr)dt_2\biggr|\nn\\
&&
\qquad\qquad\qquad\qquad.\biggl(\Delta_{t_1}p(x_{2r-2})\biggr)dt_1\nn\\
\end{eqnarray}
where $\Delta_{t}p(x)$ terms are defined as $\Delta_{t}p(x)=p_{t}(0)-p_{t}(x)$. The estimation of the term $\int_{0}^{t}\Delta_{t_6}p(.)ds$ in (\ref{eq:importantequal}) is given by
\begin{eqnarray}&&
\int_{0}^{t}\Delta_{t_6}p(x_{2l}-x_{2l-1})ds\leq \int_{0}^{\infty}\Delta_{t_6}p(x_{2l}-x_{2l-1}) ds\nn\\
&&
\qquad\qquad\,\,\qquad\qquad\qquad=\displaystyle\lim_{\alpha\to 0}(u^{\alpha}(0)-u^{\alpha}(x_{2l}-x_{2l-1}))\nn\\ 
&&
\qquad\qquad\,\,\qquad\qquad\qquad=c_{\alpha}\biggl|\frac{2l-(2l-1)}{2^{m+1}}\biggr|^{\alpha-1}\nn\\
&&
\qquad\qquad\,\,\qquad\qquad\qquad=\frac{c_{\alpha}}{2}\biggl|\frac{1}{2^{m+1}}\biggr|^{\alpha-1},
\end{eqnarray}
where we have used (4.90), (4.95)\footnote{(4.90) together with (4.95) states that $\lim_{\alpha\to 0}(u^{\alpha}(0)-u^{\alpha}(x))=\phi(x)=\frac{c_\alpha}{2}|x|^{\alpha-1}$ where $c_{\alpha}=\frac{2}{\pi}\int_{0}^{\infty}\frac{1-cost}{t^{\alpha}}dt$.} from \cite{Marcusr2006} and Lemma \ref{speciallemma} and $c_{\alpha}=\frac{2}{\pi}\int_{0}^{\infty}\frac{1-cost}{t^{\alpha}}dt$.\,\,The term $\int_{0}^{t}\Delta_{t_4}p(.)ds$ in (\ref{eq:importantequal}) can be estimated similarly. While the estimation of the term $\int_{0}^{t}\Delta_{t_1}p(x_{2r-2})dt_1$ is 
\begin{equation}
\int_{0}^{t}\Delta_{t_1}p(x_{2r-2})dt_1\leq c_{\alpha}\biggl|\frac{2r-2}{2^{m+1}}\biggr|^{\alpha-1}\leq c_{\alpha},
\end{equation}
where $\Delta_{t_1}p(x_{2r-2})=p_{t_1}(0)-p_{t_1}(x_{2r-2})$, $r\leq 2^{m-n}k$ and $k=1,\dots, 2^n.$


By (10.173) in \cite{Marcusr2006}, that is $p_t(x)\leq p_t(0)$ and $p_s(0)=d_{\alpha}s^{\frac{-1}{\alpha}}$, we can obtain the following estimation
\[
\int_{0}^t p_{\Delta_{t_5}}(.)ds\leq \int_{0}^t p_s(0)ds\leq \int_{0}^t d_\alpha s^{-1/\alpha}ds\leq \frac{\alpha d_{\alpha}}{\alpha-1} t^{\frac{\alpha-1}{\alpha}}.\]
for some constant $d_\alpha$. The positivity of the above four estimators follows from (10.173) in \cite{Marcusr2006} and $p_s(.)$ is the transitional probability density. 


 


By Lemma 3.17 and footnote 2, we have
\begin{eqnarray*}\nn
&&
\quad\int_{0}^{t}\Biggl(p_{\Delta_{t_3}}(x_{2l-1}-x_{2r})-p_{\Delta_{t_3}}(x_{2l-2}-x_{2r})\nn\\
&&
\qquad\qquad\qquad\qquad\qquad\qquad+p_{\Delta_{t_3}}(x_{2l-2}-x_{2r-1})-p_{\Delta_{t_3}}(x_{2l-1}-x_{2r-1})\Biggr)dt_3\nn\\
&&
\leq\displaystyle\lim_{\alpha\to 0}\biggl(u^{\alpha}(x_{2l-1}-x_{2r})-u^{\alpha}(0)+u^{\alpha}(0)-u^{\alpha}(x_{2l-2}-x_{2r})\nn\\
&&
\qquad\qquad\qquad\qquad\qquad\qquad+u^{\alpha}(x_{2l-2}-x_{2r-1})-u^{\alpha}(0)+u^{\alpha}(0)-u^{\alpha}(x_{2l-1}-x_{2r-1})\biggr)\nn\\
&&
=\phi(x_{2l-2}-x_{2r})-\phi(x_{2l-1}-x_{2r})+\phi(x_{2l-1}-x_{2r-1})-\phi(x_{2l-2}-x_{2r-1})\nn\\
\end{eqnarray*}
\begin{eqnarray}\nn\label{eq: remainestimation}
&&
=c_\alpha\biggl(
\bigl|\frac{2l-2-2r}{2^{m+1}}\bigr|^{\alpha-1}-\bigl|\frac{2l-1-2r}{2^{m+1}}\bigr|^{\alpha-1}+\bigl|\frac{2l-2r}{2^{m+1}}\bigr|^{\alpha-1}\nn\\
&&
\qquad-\bigl|\frac{2l-1-2r}{2^{m+1}}\bigr|^{\alpha-1}\biggr)\nn\\
&&
\leq c_\alpha \biggl(
\bigl|\frac{2l-1-2r-2l+2+2r}{2^{m+1}}\bigr|^{\alpha-1}+\bigl|\frac{2l-2r-2l+1+2r}{2^{m+1}}\bigr|^{\alpha-1}\biggr)\nn\\
&&
\leq c_\alpha \bigl(\frac{1}{2^m}\bigr)^{\alpha-1},
\end{eqnarray}
where the first inequality of (\ref{eq: remainestimation}) follows from the fact that $|x|^{\alpha-1}-|y|^{\alpha-1}\leq |x-y|^{\alpha-1}$ for $0<\alpha-1<1$. 

We can see that (\ref{eq: remainestimation}) is bounded from below by using the inequality $|x|^{\alpha-1}-|y|^{\alpha-1}\leq |x-y|^{\alpha-1}$ again and increasing property of $|.|^{\alpha-1}$ with $0<\alpha-1<1$ 
\begin{eqnarray}&&
\quad\bigl|\frac{2l-2-2r}{2^{m+1}}\bigr|^{\alpha-1}-\bigl|\frac{2l-1-2r}{2^{m+1}}\bigr|^{\alpha-1}+\bigl|\frac{2l-2r}{2^{m+1}}\bigr|^{\alpha-1}-\bigl|\frac{2l-1-2r}{2^{m+1}}\bigr|^{\alpha-1}\nn\\
&&
\geq \bigl|\frac{2l-2r-1}{2^{m+1}}\bigr|^{\alpha-1}-\bigl|\frac{1}{2^{m+1}}\bigr|^{\alpha-1}-\bigl|\frac{2l-1-2r}{2^{m+1}}\bigr|^{\alpha-1}+\bigl|\frac{2l-2r}{2^{m+1}}\bigr|^{\alpha-1}\nn\\
&&
\qquad\quad-\bigl|\frac{2l-1-2r}{2^{m+1}}\bigr|^{\alpha-1}\nn\\
&&
\geq -\bigl|\frac{1}{2^{m+1}}\bigr|^{\alpha-1}.
\end{eqnarray}
Similarly, we can prove that
\begin{eqnarray}&&
\quad\int_{0}^{t}\Biggl(p_{\Delta_{t_2}}(x_{2r-1}-x_{2l-1})-p_{\Delta_{t_2}}(x_{2r-2}-x_{2l-1})\nn\\
&&
\qquad\qquad\qquad\qquad\qquad\qquad +p_{\Delta_{t_2}}(x_{2r-2}-x_{2l-2})-p_{\Delta_{t_2}}(x_{2r-1}-x_{2l-2})\Biggr)dt_2\nn\\
&&
\leq c_\alpha \bigl(\frac{1}{2^m}\bigr)^{\alpha-1},
\end{eqnarray}
and it is bounded from below by the following
\begin{eqnarray}&&
\quad\bigl|\frac{2l-2r}{2^{m+1}}\bigr|^{\alpha-1}-\bigl|\frac{2l-2r+1}{2^{m+1}}\bigr|^{\alpha-1}+\bigl|\frac{2l-2r}{2^{m+1}}\bigr|^{\alpha-1}-\bigl|\frac{2l-1-2r}{2^{m+1}}\bigr|^{\alpha-1}\nn\\
&&
\geq \bigl|\frac{2l-2r}{2^{m+1}}\bigr|^{\alpha-1}-\bigl|\frac{2l-2r}{2^{m+1}}\bigr|^{\alpha-1}-\bigl|\frac{1}{2^{m+1}}\bigr|^{\alpha-1}+\bigl|\frac{2l-2r}{2^{m+1}}\bigr|^{\alpha-1}\nn\\
&&
\qquad\quad-\bigl|\frac{2l-1-2r}{2^{m+1}}\bigr|^{\alpha-1}\nn\\
&&
\geq -\bigl|\frac{1}{2^{m+1}}\bigr|^{\alpha-1}.
\end{eqnarray}
Therefore, the estimation of the fifth term in (\ref{longequation}) is
\begin{eqnarray}\nn
&&
\quad\displaystyle\sum_{r< l}\biggl|E\bigl(L_t^{x_{2r}^{m+1}}-L_t^{x_{2r-1}^{m+1}}\bigr)^2\bigl(L_t^{x_{2l}^{m+1}}-L_t^{x_{2l-1}^{m+1}}\bigr)^2\bigl(L_t^{x_{2l-1}^{m+1}}-L_t^{x_{2l-2}^{m+1}}\bigr)\bigl(L_t^{x_{2r-1}^{m+1}}-L_t^{x_{2r-2}^{m+1}}\bigr)\biggr|\nn\\
&&
\leq c\, t^{\frac{\alpha-1}{\alpha}}\biggl|\frac{1}{2^{m+1}}\biggr|^{4\alpha-6}\biggl(\frac{1}{2^{n}}\biggr)^2,
\end{eqnarray}
where c is a generic constant.
 
\noindent The upper bound for the sixth term in (\ref{longequation}) is
\begin{eqnarray}\nn
&&
\quad\displaystyle\sum_{r< l}\biggl|E(g(x_{2r-1}^{m+1})-g(x_{2r-2}^{m+1})\bigr)\bigl(g(x_{2l-1}^{m+1})-g(x_{2l-2}^{m+1})\bigr)\bigl(g(x_{2l}^{m+1})-g(x_{2l-1}^{m+1})\bigr)^2\nn\\
&&
\quad\quad\quad\bigl(g(x_{2r}^{m+1})-g(x_{2r-1}^{m+1})\bigr)^2\biggr|\nn\\
&&
\leq c\,\frac{1}{2^n}\bigl(\frac{1}{2^m}\bigr)^{6h-1}.
\end{eqnarray}

Therefore, we have
\begin{eqnarray}\label{eq: 4.16}\nn
&&
\quad E\displaystyle\biggl|\sum_{l}\Delta_{2l-1}^{m+1}Z\otimes\Delta_{2l}^{m+1}Z\otimes\Delta_{2l}^{m+1}Z\biggr|^2\nn\\
&&
\leq C\,\biggl[\bigl(\frac{1}{2^n}\bigr)^{1+2h}\bigl(\frac{1}{2^m}\bigr)^{2h+\alpha-2}+t^{\frac{2(\alpha-1)}{\alpha}}\bigl(\frac{1}{2^n}\bigr)^{\alpha}\bigl(\frac{1}{2^m}\bigr)^{2h+\alpha-2}+t^{\frac{2(\alpha-1)}{\alpha}}\bigl(\frac{1}{2^n}\bigr)\bigl(\frac{1}{2^m}\bigr)^{2h+2\alpha-3}\nn\\
&&
\quad\qquad+\bigl(\frac{1}{2^n}\bigr)\bigl(\frac{1}{2^m}\bigr)^{4h+\alpha-2}+\bigl(\frac{1}{2^n}\bigr)\bigl(\frac{1}{2^m}\bigr)^{6h-1}+t^{\frac{2(\alpha-1)}{\alpha}}\bigl(\frac{1}{2^n}\bigr)^2(\frac{1}{2^m})^{3\alpha-5}\biggr].
\end{eqnarray}
Hence, we have proved the following proposition.
\begin{proposition}\label{ThirdLevelVariationProposition1}
For a continuous path $Z_x$ which satisfies (\ref{eq: 4.2}), then for the case $m\geq n$
\begin{eqnarray}\label{eq: 4.17}
&&
\quad\sum_{k=1}^{2^n}E\biggl|{\bf{Z}}(m+1)_{x_{k-1}^n, x_k^n}^3-{\bf{Z}}(m)_{x_{k-1}^n, x_k^n}^3\biggr|^{\frac{\theta}{3}}\nn\\
&&
\leq C\biggl[\bigl(\frac{1}{2^n}\bigr)^{\frac{1+2h}{6}\theta-1}\bigl(\frac{1}{2^m}\bigr)^{\frac{2h+\alpha-2}{6}\theta}+\bigl(\frac{1}{2^n}\bigr)^{\frac{\alpha}{6}\theta-1}\bigl(\frac{1}{2^m}\bigr)^{\frac{2h+\alpha-2}{6}\theta}+ \bigl(\frac{1}{2^n}\bigr)^{\frac{1}{6}\theta-1}\bigl(\frac{1}{2^m}\bigr)^{\frac{2h+2\alpha-3}{6}\theta} \nn\\
&&
\qquad+\bigl(\frac{1}{2^n}\bigr)^{\frac{\theta}{6}-1}\bigl(\frac{1}{2^m}\bigr)^{\frac{4h+\alpha-2}{6}\theta}+\bigl(\frac{1}{2^n}\bigr)^{\frac{\theta}{6}-1}\bigl(\frac{1}{2^m}\bigr)^{\frac{6h-1}{6}\theta}+  \bigl(\frac{1}{2^n}\bigr)^{\frac{2}{6}\theta-1}\bigl(\frac{1}{2^m}\bigr)^{\frac{3\alpha-5}{6}\theta}\biggr],
\end{eqnarray}\end{proposition}
\noindent where $C$ is a generic constant depends on $\theta, h$, $\textbf{w}_1(x',x'')$, and $c$ in (\ref{eq: 4.2}).

\begin{theorem}\label{MainthirdlevelTheorem}
Let  $\alpha\in(\frac{3}{2},2)$, $\frac{2}{3-\alpha}\leq q<4$.
Then for a continuous path $Z_x$ satisfy (\ref{eq: 4.2}), there exists a unique ${\bf Z^3}$ and a simplex $\Delta$ taking values in $(\mathbb{R}^2)^{\otimes 3}$ such that
\begin{equation}
\displaystyle\sup_{D}\biggl(\displaystyle\sum_{l}\bigl|{\bf Z}(m)_{x_{l-1},x_l}^3-{\bf Z}_{x_{l-1},x_l}^3\bigr|^{\frac{\theta}{3}}\biggr)^{\frac{3}{\theta}}\to 0,
\end{equation}
both almost surely and in $L^1(\Omega,\mathcal{F},\mathcal{P})$ as $m\to\infty$, for some $\theta$ such that $\frac{4}{2h+\alpha-1}<\theta<4.$
\end{theorem}
\begin{proof}From \cite{Lyons2002}, we have
\begin{eqnarray*}\nn
&&
\quad E\displaystyle\sup_{D}\bigl|{\bf Z}(m+1)_{x_{l-1},x_l}^3-{\bf Z}(m)_{x_{l-1},x_l}^3\bigr|^{\frac{\theta}{3}}\nn\\
&&
\leq C_2\sum_{n=1}^{\infty}n^\gamma \sum_{k=1}^{2^n}E\bigl|{\bf Z}(m+1)_{x_{k-1}^n,x_k^n}^3-{\bf Z}(m)_{x_{k-1}^n,x_k^n}^3\bigr|^{\frac{\theta}{3}}\nn\\
&&
\quad + C_3E\biggl(\sum_{n=1}^{\infty}n^\gamma\sum_{k=1}^{2^n}\bigl|{\bf Z}(m+1)_{x_{k-1}^n,x_k^n}^1-{\bf Z}(m)_{x_{k-1}^n,x_k^n}^1\bigr|^{{\theta}}\biggr)^{\frac{1}{3}}\nn\\
&&
\quad\quad\quad \times\biggl(\sum_{n=1}^{\infty}n^\gamma\sum_{k=1}^{2^n}\bigl|{\bf Z}(m+1)_{x_{k-1}^n,x_k^n}^2\bigr|^{\frac{\theta}{2}}+\bigl|{\bf Z}(m)_{x_{k-1}^n,x_k^n}^2\bigr|^{\frac{\theta}{2}}\biggr)^{\frac{2}{3}}\nn\\
\end{eqnarray*}
\begin{eqnarray}\nn
&&
\quad+ C_4E\biggl(\sum_{n=1}^{\infty}n^\gamma\sum_{k=1}^{2^n}\bigl|{\bf Z}(m+1)_{x_{k-1}^n,x_k^n}^2-{\bf Z}(m)_{x_{k-1}^n,x_k^n}^2\bigr|^{\frac{\theta}{2}}\biggr)^{\frac{2}{3}}\nn\\
&&
\quad\quad\quad \times\biggl(\sum_{n=1}^{\infty}n^\gamma\sum_{k=1}^{2^n}\bigl|{\bf Z}(m+1)_{x_{k-1}^n,x_k^n}^1\bigr|^{{\theta}}+\bigl|{\bf Z}(m)_{x_{k-1}^n,x_k^n}^1\bigr|^{{\theta}}\biggr)^{\frac{1}{3}}\nn\\
&&
\quad +C_5E\biggl(\sum_{n=1}^{\infty}n^\gamma\sum_{k=1}^{2^n}\bigl|{\bf Z}(m+1)_{x_{k-1}^n,x_k^n}^1-{\bf Z}(m)_{x_{k-1}^n,x_k^n}^1\bigr|^{\theta}\biggr)^{\frac{1}{3}}\nn\\
&&
\quad\quad\quad \times\biggl(\sum_{n=1}^{\infty}n^\gamma\sum_{k=1}^{2^n}\bigl|{\bf Z}(m+1)_{x_{k-1}^n,x_k^n}^1\bigr|^{{\theta}}+\bigl|{\bf Z}(m)_{x_{k-1}^n,x_k^n}^1\bigr|^{{\theta}}\biggr)^{\frac{2}{3}}.
\end{eqnarray}
By Propositions \ref{ThirdLevelVariationProposition} and \ref{ThirdLevelVariationProposition1}, the estimate of the first term on the r.h.s. of (3.107) is
\begin{eqnarray*}
&&
\quad\sum_{n=1}^{\infty}n^\gamma E\sum_{k=1}^{2^n}\bigl|{\bf Z}(m+1)_{x_{k-1}^n,x_k^n}^3-{\bf Z}(m)_{x_{k-1}^n,x_k^n}^3\bigr|^{\frac{\theta}{3}}\nn\\
&&
\leq C \sum_{n=m}^{\infty}n^\gamma \bigl(\frac{1}{2^{n+m}}\bigr)^{\frac{h\theta-1}{2}}+C\:\sum_{n=1}^{m-1}n^\gamma\biggl[\bigl(\frac{1}{2^n}\bigr)^{\frac{1+2h}{6}\theta-1}\bigl(\frac{1}{2^m}\bigr)^{\frac{2h+\alpha-2}{6}\theta}\nn\\
&&
\qquad+ \bigl(\frac{1}{2^n}\bigr)^{\frac{1}{6}\theta-1}\bigl(\frac{1}{2^m}\bigr)^{\frac{2h+2\alpha-3}{6}\theta} +\bigl(\frac{1}{2^n}\bigr)^{\frac{\theta}{6}-1}\bigl(\frac{1}{2^m}\bigr)^{\frac{4h+\alpha-2}{6}\theta}\nn\\
&&
\qquad+\bigl(\frac{1}{2^n}\bigr)^{\frac{\alpha}{6}\theta-1}\bigl(\frac{1}{2^m}\bigr)^{\frac{2h+\alpha-2}{6}\theta}+\bigl(\frac{1}{2^n}\bigr)^{\frac{\theta}{6}-1}\bigl(\frac{1}{2^m}\bigr)^{\frac{6h-1}{6}\theta}+  \bigl(\frac{1}{2^n}\bigr)^{\frac{2}{6}\theta-1}\bigl(\frac{1}{2^m}\bigr)^{\frac{4\alpha-6}{6}\theta}\biggr]\nn\\
&&
\leq C\:\bigl(\frac{1}{2^{m}}\bigr)^{\frac{h\theta-1}{2}}+C\:\sum_{n=1}^{m-1}n^\gamma\biggl[ \bigl(\frac{1}{2^n}\bigr)^{\frac{2h+1}{6}\theta-1}\bigl(\frac{1}{2^m}\bigr)^{\frac{2h+\alpha-2}{6}\theta} +\bigl(\frac{1}{2^n}\bigr)^{\frac{\alpha}{6}\theta-1}\bigl(\frac{1}{2^m}\bigr)^{\frac{2h+\alpha-2}{6}\theta}  \nn\\
&&
\quad\quad+\bigl(\frac{1}{2^n}\bigr)^{\frac{\theta}{6}-\frac{2}{3}+\epsilon}\bigl(\frac{1}{2^m}\bigr)^{\frac{2h+2\alpha-3}{6}\theta-\frac{1}{3}-\epsilon}+\bigl(\frac{1}{2^n}\bigr)^{\frac{\theta}{6}-\frac{2}{3}+\epsilon}\bigl(\frac{1}{2^m}\bigr)^{\frac{4h+\alpha-2}{6}\theta-\frac{1}{3}-\epsilon}
\nn\\
&&
\quad\quad+\bigl(\frac{1}{2^n}\bigr)^{\frac{\theta}{6}-\frac{2}{3}+\epsilon}\bigl(\frac{1}{2^m}\bigr)^{\frac{6h-1}{6}\theta-\frac{1}{3}-\epsilon}+\bigl(\frac{1}{2^n}\bigr)^{\frac{2\theta}{6}-1}\bigl(\frac{1}{2^m}\bigr)^{\frac{4\alpha-6}{6}\theta}\biggr]\nn\\
&&
:= K_m.
\end{eqnarray*}

For $h>\frac{1}{4}$ and $\alpha>\frac{3}{2}$, one can choose $\theta$ sufficiently close to 4 such that $\frac{2h+1}{6}\theta-1>0,\,\frac{2h+\alpha-2}{6}\theta>0,\,\frac{\alpha}{6}\theta-1>0,\,\frac{2\theta}{6}-1>0,\,\frac{4\alpha-6}{6}\theta>0.$ 
We choose $\epsilon>0$ with $\frac{2}{3}-\frac{\theta}{6}<\epsilon<min\{\frac{2h+2\alpha-3}{6}\theta-\frac{1}{3}, \frac{4h+\alpha-2}{6}\theta-\frac{1}{3}, \frac{6h-1}{6}\theta-\frac{1}{3}\}= \frac{6h-1}{6}\theta-\frac{1}{3}$, hence
\begin{eqnarray}\label{eq: 4.23}
&&
K_m \leq C\:\biggl[\bigl(\frac{1}{2^{m}}\bigr)^{\frac{h\theta-1}{2}}+\bigl(\frac{1}{2^m}\bigr)^{\frac{2h+\alpha-2}{6}\theta} +\bigl(\frac{1}{2^m}\bigr)^{\frac{2h+2\alpha-3}{6}\theta-\frac{1}{3}-\epsilon} +
\bigl(\frac{1}{2^m}\bigr)^{\frac{4h+\alpha-2}{6}\theta-\frac{1}{3}-\epsilon}\nn\\
&&
\quad\qquad+\bigl(\frac{1}{2^m}\bigr)^{\frac{6h-1}{6}\theta-\frac{1}{3}-\epsilon}
+\bigl(\frac{1}{2^m}\bigr)^{\frac{4\alpha-6}{6}\theta}\biggr].
\end{eqnarray}
Therefore, if we sum up all m, we would have $\displaystyle\sum_m K_m<\infty$. This shows that $\bigl({\bf Z}(m)^3\bigr)_{m\in N}\in (\mathbb{R}^2)^{\otimes 3}$ is a Cauchy sequence in $\theta$-variation distance. In other words, it has a limit as $m\to\infty$, denote it by ${\bf Z}^3\in(\mathbb{R}^2)^{\otimes 3}$. By Lemma 3.3.3 in \cite{Lyons2002}, ${\bf Z}^3$ has finte $\theta$-variation.\,\,Together with the convergence result of the second level path and first level path, we complete the proof of the theorem. \end{proof}

Based on Chapter 5 in \cite{Lyons2002}, for any Lipschitz one form in the sense of Stein $\widehat{f}: \mathbb{R}^2\to L(\mathbb{R}^2, \mathbb{R}^2),$ the almost rough path $Y = (1, Y_{a,b}^1, Y_{a, b}^2, Y_{a, b}^3)$ is given  by
\begin{align*}
Y_{a,b}^1&=\widehat{f}({\bf Z}_a){\bf Z}_{a,b}^1+\widehat{f}^2({\bf Z}_a){\bf Z}_{a,b}^2 +\widehat{f}^3({\bf Z}_a){\bf Z}_{a,b}^3,\\
Y_{a,b}^2&=(\widehat{f}({\bf Z}_a)\otimes\widehat{f}({\bf Z}_a)){\bf Z}_{a,b}^2+ (\widehat{f}({\bf Z}_a)\otimes\widehat{f}^2({\bf Z}_a))\displaystyle\int_{s<u_1<u_2<t}d{\bf Z}_{s,u_1}^1\otimes d{\bf Z}_{s,u_2}^2,\\
&\quad+ (\widehat{f}^2({\bf Z}_a)\otimes\widehat{f}({\bf Z}_a))\displaystyle\int_{s<u_1<u_2<t}d{\bf Z}_{s,u_1}^2\otimes d{\bf Z}_{s,u_2}^1,\\
Y_{a,b}^3&= (\widehat{f}({\bf Z}_a)\otimes\widehat{f}({\bf Z}_a)\otimes\widehat{f}({\bf Z}_a)){\bf Z}_{a,b}^3.
\end{align*}

Consider one form $\widehat{f}: \mathbb{R}^2\to L(\mathbb{R}^2, \mathbb{R}^2)$ defined as
\[
\widehat{f}(z)\xi = (v, yv),\]
where $z = (x,y)$ and $\xi=(v,w).$ For the general case, it is defined as
\[
\widehat{f}^{k+1}(v,w){(\bf v)}=\biggl(0, \sum_{j}d^k \widehat{f}(w)(w_{k+1}^j,\dots,w_2^j)v_1^j\biggr)\]
for all ${\bf v} =\sum_{j}(v_{k+1}^j, w_{k+1}^j)\otimes\dots\otimes (v_{2}^j, w_{2}^j)\otimes (v_{1}^j, w_{1}^j).$

We use the notation $\int_{a}^{b}\widehat{f}({\bf Z})d{\bf Z}^n$ to denote the n-th degree term of  $\int_{a}^{b}\widehat{f}({\bf Z})d{\bf Z}$. When $n=1$, we have
\begin{equation}
\int_{-\infty}^{\infty}\widehat{f}({\bf Z})d{\bf Z}^1=\biggl(\int_{-\infty}^{\infty}dL_t^x, \int_{-\infty}^{ \infty}g(x)dL_t^x\biggr).
\end{equation}
One can therefore use the almost rough path to construct the unique rough path $\int_{-\infty}^{\infty}\widehat{f}({\bf Z})d{\bf Z}$ with roughness $\theta$ in $T^{(3)}(\mathbb{R}^2)$. In particular
\begin{eqnarray}\nn
&&
\quad\int_{-\infty}^{\infty}\widehat{f}({\bf Z})d{\bf Z}^1\nn\\
&&
=\displaystyle\lim_{m(D)\to 0}\sum_{i=1}^{r}Y_{x_{i-1},x_i}^1\nn\\
&&
=\displaystyle\lim_{m(D)\to 0}\sum_{i=1}^{r}\biggl[\widehat{f}({\bf Z}_{x_{i-1}})({\bf Z}_{x_{i-1}, x_i}^1)+\widehat{f}^2({\bf Z}_{x_{i-1}})({\bf Z}_{x_{i-1}, x_i}^2)+ \widehat{f}^3({\bf Z}_{x_{i-1}})({\bf Z}_{x_{i-1}, x_i}^3)\biggr].\nn\\
\end{eqnarray}
The above integral is well defined as the limit of the almost rough path. In particular, we have
\begin{eqnarray}\nn
&&
\widehat{f}({\bf Z}_a)({\bf Z}_{a, b}^1)+\widehat{f}^2({\bf Z}_a)({\bf Z}_{a, b}^2)+ \widehat{f}^3({\bf Z}_a)({\bf Z}_{a, b}^3)=\biggl(L_t^b - L_t^a,g(a)(L_t^b - L_t^a)\biggr) \nn\\
&&
\qquad\qquad\qquad\qquad\qquad\qquad\qquad\qquad\qquad\qquad+ \biggl(0, ({\bf Z}_{a,b}^2)_{2,1}\biggr)\biggr)
\end{eqnarray}
as $ \widehat{f}^3$ is equal to zero for this particular linear one-form. 
Hence, we have the following corollary.
\begin{corollary}\label{MainthirdlevelCorollary}
Let $\frac{3}{2}<\alpha<2$ and g be a continuous function with bounded $q$-variation, $\frac{2}{3-\alpha}<q<4$. 
 Then, the integral $\int_{a}^{b}g(x)dL_t^x$ for $(a,b)\in\Delta$ can be defined as
\begin{equation}
\int_{a}^{b}g(x)dL_t^x=\displaystyle\lim_{m(D)\to 0}\biggl[\sum_{i=1}^{r}g(x_{i-1})(L_t^{x_i} - L_t^{x_{i-1}}) + ({\bf Z}_{x_{i-1}, x_i}^2)_{2,1}\biggr].
\end{equation}
\end{corollary}
 
\subsection{Convergence of rough path integrals for the third level path}
In this section, we will prove the convergence of the rough path integral of the third level path in the $\theta$-variation topology.
 
\begin{proposition}\label{convergencepropthirdlevel}
Let $\frac{3}{2}<\alpha<2$, $\frac{2}{3-\alpha}\leq q<4$.\,Moreover, let $Z_j(x): = (L_t^x, g_j(x)),$ and $ Z(x):=(L_t^x, g(x))$, where $g_j(\cdot), g(\cdot)$ are both continuous and of bounded $q$-variation.\,Suppose $g_j(x)\to g(x)$ as $j\to\infty$ uniformly and the control function ${\bf w}_j(x,y)$ of $g_j$ converges to the control function ${\bf w}(x,y)$ of $g$ as $j\to\infty$ uniformly in $(x,y)$.\,Then as $j\to\infty$ such that the geometric rough path ${\bf Z}_{j}(\cdot)$ associated with $Z_{j}(\cdot)$ converges to the geometric rough path ${\bf Z}(\cdot)$ associated with $Z(\cdot)$ a.s. in $\theta$-variation topology as $j\to\infty$. Here, 
we choose a $\theta$ such that $\frac{4}{2h+\alpha-1}<\theta<4$. In particular, $\int_{-\infty}^{\infty} g_{j}(x)dL_t^x\to \int_{-\infty}^{\infty} g(x)dL_t^x$ a.s. as $j\to\infty.$
\end{proposition}
\begin{proof} By the reasoning given above and under the conditions given in the proposition, one can obtain the geometric rough path ${\bf Z}_j(\cdot)$ associated with $Z_j(\cdot)$, and also the smooth rough path ${\bf Z}_j(m)$.  Here the ${\bf Z}_j$ is defined as  ${\bf Z}_j = (1, {\bf Z}_j^1, {\bf Z}_j^2,  {\bf Z}_j^3)$ while ${\bf Z} = (1, {\bf Z}^1, {\bf Z}^2, {\bf Z}^3)$, similarly we have ${\bf Z}_j(m) = (1, {\bf Z}_j^1(m), {\bf Z}_j^2(m), {\bf Z}_j^3(m))$ while ${\bf Z}(m) = (1, {\bf Z}^1(m), {\bf Z}^2(m), {\bf Z}^3(m))$. The convergence of ${\bf Z}_j(\cdot)\to {\bf Z}(\cdot)$  in $\theta$-variation means the convergence of corresponding level path in $\theta$-variation when $g_j(\cdot)\to g(\cdot)$ as $j\to\infty$. We have discussed the convergence of first and second level in Proposition \ref{continuitysecondlevel}.  By similar argument as in Proposition \ref{continuitysecondlevel}, one can show the convergence of the third level path.
\end{proof}
Next, we prove that in fact the proposition above is also true for function $g$ being of bounded $q$-variation  $(\frac{2}{3-\alpha}\leq q<4$), but not being continuous.
 
\begin{theorem}
Let g(x) be a  c\`adl\`ag path with bounded $q$-variation $(\frac{2}{3-\alpha}\leq q<4$). Then
\[
\int_{x'}^{x''} L_t^xdg(x) = \int_{x'}^{\tau_{\delta}(x'')} L_{t,\delta}(y)dg_{\delta}(y).\]
\end{theorem}
The proof of the Theorem 3.22 is similar to Theorem \ref{continuousthmsecondlevel}.\,Based on Theorem 3.22 and Proposition \ref{convergencepropthirdlevel}, one can prove the following proposition
\begin{proposition}
Let $\alpha\in(\frac{3}{2},2)$, $\frac{2}{3-\alpha}\leq q<3$, one can choose $\theta$ such that $\frac{4}{2h+\alpha-1}<\theta<3$.
\,\,Moreover, let $Z_j(x): = (L_t^x, g_j(x)),\, Z(x):=(L_t^x, g(x))$,\,\,where $g_j(.), g(.)$ are both of bounded q-variation, and $g_j$ is continuous and $g$ is c\`adl\`ag with decomposition $g=g_c+h$, where $g_c$ is the continuous part of $g$ and $h$ is the jump part of $g$. Suppose $g_j=g_{c_j}+h_j$ with control function ${\bf w}_{cj}$ and ${\bf w}_{hj}$ such that $g_{cj}\to g_c$ and  ${\bf w}_{cj}\to{\bf w}_{c}$ uniformly, $h_j$ satisfying conditions (3.47) in Proposition \ref{propadded}. Then 
\begin{equation*}
\int_{-\infty}^{\infty} g_j(x)dL_t^x\to\int_{-\infty}^{\infty} g(x)dL_t^x \quad a.s.\quad as\quad j\to\infty.
\end{equation*}
\end{proposition}

Recall $\rho$ as the mollifier and define
\begin{equation*}
h_j(x)=\int_0^2 \rho(z)h(x-\frac{z}{j})dz
\end{equation*}
and $h_{j\delta}$ in the same way as $G_\delta$, so $h_{j\delta}=h_j$ as $h_j$ is continuous.\,Define $h_\delta$ in the same way as $G_\delta$, then 
\begin{equation*}
h_{\delta j}(y)=\int_0^2 \rho(z)h_\delta(y-\frac{z}{j})dz.
\end{equation*}
Thus by the integration by parts formula and Fubini theorem, we have
\begin{eqnarray*}\nn
&&
\quad\int_{-\infty}^{\infty} L_t^\delta(y)dh_{\delta j}(y)\nn\\
&&
=-\int_{-\infty}^{\infty} h_{\delta j}(y)dL_t^\delta(y)\nn\\
&&
=-\int_{-\infty}^{\infty}\int_0^2 \rho(z)h_\delta(y-\frac{z}{j})dzdL_t^\delta(y)\nn\\
&&
=-\int_0^2\int_{-\infty}^{\infty} h_\delta(y-\frac{z}{j})dL_t^\delta(y)\rho(z)dz.\nn\\
\end{eqnarray*}

By Theorem \ref{continuousthmsecondlevel}, we have
\begin{equation*}
\int_{-\infty}^{\infty}h_\delta(y-\frac{z}{j})dL_t^\delta(y)=\int_{-\infty}^{\infty}h(x-\frac{z}{j})dL_t^x.
\end{equation*}
Hence, 
\begin{eqnarray}\nn
&&
\quad\int_{-\infty}^{\infty}L_t^{\delta}(y)dh_{\delta_j}(y)\nn\\
&&
=-\int_0^2\rho(z)\int_{-\infty}^{\infty} h(x-\frac{z}{j})dL_t^xdz\nn\\
&&
=-\int_{-\infty}^{\infty}\int_0^2 h(x-\frac{z}{j})\rho(z)dzdL_t^x\nn\\
&&
=-\int_{-\infty}^{\infty}h_j(x)dL_t^x\nn\\
&&
=\int_{-\infty}^{\infty}L_t^xdh_j(x)\nn\\
&&
=\int_{-\infty}^{\infty}L_t^\delta(x)dh_{j\delta}(x),
\end{eqnarray}

The last equality follows from Theorem \ref{continuousthmsecondlevel}.
This indicates that condition (\ref{three star}) in Proposition \ref{propadded}, which is also needed in Proposition 3.23, is satisfied.

The following theorem summarizes the main results of the paper.
\begin{theorem}Let  $X = (X_t)_{t\geq0}$ be a symmetric $\alpha$-stable process and $f: \mathbb{R} \to \mathbb{R}$ be absolutely continuous, locally bounded function and has $(\alpha-1)^{th}$ fractional order derivative ${\triangledown^{\alpha-1}_-f(x)}$ which is locally bounded.\,Assume ${\triangledown^{\alpha-1}_-f(x)}$ is of bounded $q$-variation, where $1\leq q < 4$.\,Then we have the following extended version of It\^o's formula
\begin{eqnarray}\nn
&&
f(X_t) = f(X_0) + \int_{0}^{t}\triangledown_- f(X_s)dX_s\nn\\
&&
\qquad\qquad+ \int_{0}^{t}\int_{\mathbb{R}}\biggl(f(X_{s-} +y)-f(X_{s-})\biggr)
\tilde{N}(dy,ds)-{C_{\alpha}}\int_{-\infty}^{\infty}\triangledown^{\alpha-1}_-f(x)d_x L_t^x,\nn\\
\end{eqnarray}
where $C_{\alpha} = \frac{\pi^{\frac{1}{2}}\Gamma(1-\frac{2}{\alpha})}{\alpha 2^{\alpha-1}\Gamma(\frac{1+\alpha}{2})}$.
 
The integral $\int_{-\infty}^{\infty}\triangledown^{\alpha-1}_-f(x)d_x L_t^x$ is a Lebesgue-Stieltjes integral when q=1, a Young integral when $1 < q < \frac{2}{3-\alpha}$ for $1<\alpha<2$ and a rough path integral when $\frac{2}{3-\alpha}\leq q < 4$ for $\frac{3}{2}<\alpha<2$ respectively.
\end{theorem}
\begin{proof}
We have already showed that the integral $\int_{-\infty}^{\infty}\triangledown^{\alpha-1}_-f(x)d_x L_t^x$ can be defined as a Young integral when $1 < q < \frac{2}{3-\alpha}$ for $1<\alpha<2$.\,\,Based on the result proved for level 1, level 2, level 3 path and (3.113), as well as  applying a standard smoothing argument and taking limit using Proposition 3.23, one can define the integral $\int_{-\infty}^{\infty}\triangledown^{\alpha-1}_-f(x)d_x L_t^x$ as a rough path integral. 
\end{proof}

\section*{Acknowledgements}
The authors are grateful to the anonymous referees for their constructive comments which have helped to improve significantly the earlier version of this paper.\,We would also like to thank Dr Chunrong Feng for very useful conversations about this paper.\,Huaizhong acknowledges the financial support of Royal Society Newton Advanced Fellowship NA150344.



\end{document}